\documentclass[11pt]{article}
\usepackage{amsmath,amssymb}
\usepackage{graphicx}
\usepackage{mathrsfs}
\usepackage{bbm}
\usepackage{xcolor}
\usepackage{stmaryrd}
\usepackage{upgreek}
\usepackage{esint}
\usepackage{appendix}

\newtheorem{remark}{Remark}
\newtheorem{theorem}{Theorem}
\newtheorem{lemma}{Lemma}

\newtheorem{proposition}{Proposition}

\newtheorem{corollary}{Corollary}

\def \R{{\mathbb{R}}}

\def \Z{{\mathbb{Z}}}

\title{Diffusion limit of the Vlasov equation in the weak turbulent regime}
\author{Claude Bardos\footnotemark[1] \\
  Laboratoire J.-L. Lions,
  Sorbonne Universit\'e, \\ 
  BC 187, 4 place Jussieu,
  75252 Paris, Cedex 05, France.\\
  \and\\
  Nicolas Besse \footnotemark[2] \\
  Laboratoire J.-L. Lagrange,   \\
  Observatoire de la C\^ote d'Azur, Universit{\'e} C\^ote d'Azur, \\
  Bd de l'observatoire CS 34229, 06300 Nice, Cedex 4, France.\\  
} 

\begin{document}

\maketitle


\renewcommand{\thefootnote}{\fnsymbol{footnote}}

\footnotetext[1]{
  claude.bardos@gmail.com
}

\footnotetext[2]{
  Author to whom correspondence should be addressed: Nicolas.Besse@oca.eu
}

\renewcommand{\thefootnote}{\arabic{footnote}}

\begin{abstract}
  In this paper we study the Hamiltonian dynamics of charged particles
  subject to a non-self-consistent stochastic electric field, when the
  plasma is in the so-called weak turbulent regime. We show that the
  asymptotic limit of the Vlasov equation is a diffusion equation in
  the velocity space, but homogeneous in the physical space.  We
  obtain a diffusion matrix, quadratic with respect to the electric
  field, which can be related to the diffusion matrix of
  the resonance broadening theory and of the quasilinear theory, depending
  on whether the typical autocorrelation time of particles is finite
  or not.  In the self-consistent deterministic case, we show that the
  asymptotic distribution function is homogenized in the space
  variables, while the electric field converges weakly to zero. We
  also show that the lack of compactness in time for the electric
  field is necessary to obtain a genuine diffusion limit.  By contrast, the time
  compactness property leads to a ``cheap'' version of the Landau
  damping: the electric field converges strongly to zero, implying the
  vanishing of the diffusion matrix, while the distribution function
  relaxes, in a weak topology, towards a spatially homogeneous
  stationary solution of the Vlasov-Poisson system.
\end{abstract}

\section{Introduction}
\label{s:intro}
Here, we are interested in a problem of particle diffusion which is
produced by the wave-particle interaction.  In plasma physics, the
wave-particle interaction is an important phenomenon, which stands at
the root of Landau damping, of wave heating, of numerous
instabilities, and of some regimes of anomalous transport in
magnetically confined plasmas. This work is closely related to the
so-called quasilinear (QL) theory, which describes the nonlinear
relaxation of the weak warm beam-plasma instability through the
derivation of a diffusion equation in the velocity variable conjugated
with the prediction of an associated diffusion coefficient.  This
topic has led to a longstanding controversy that is not solved yet 
\cite{VVS62, DP62, DU66, RG71, DU72, ALP79, GSS80, BGD81, LP83a,
  LP83b, LP84, CEV90, TDM91, CD92, LD93a, LD93b, HDOS95, DC97, SS97,
  LP99, EE02, EE03, EE03b, BEEB11, EP10, Els12, Esc10,
  LPA18}. References cited above are not exhaustive but testify to the
huge literature on this subject. For a brief history on the
development of the QL theory, we refer the reader to references
\cite{BEEB11, LPA18}.  Furthermore the QL diffusion coefficient is
quite frequently used for modelling particle transport in different
branches of plasma physics, such as laser-plasma interaction or
 magnetized plasma turbulence. Since the QL approximation is
ubiquitous, particularly in kinetic modelling, it is then important to
assess, in the most rigorous possible way, whether the QL theory is valid or not.
A complete and rigorous proof of the QL theory goes beyond the purpose
of this paper, which aims at taking stock of what can or cannot be
rigorously proven at this time.

We now sketch this problem in dimension one (see, e.g., \cite{Esc10}
for an intuitive introduction and  Chapters 8-9 of \cite{Dav72} or
Chapter 7 of \cite{EE03b} for a more exhaustive one). We consider
a two-dimensional distribution function of
particles in the two-dimensional phase-space $(x,v)$.
This distribution function is initially given by a one-dimensional (in $v$) spatially
uniform (in $x$) beam-plasma system.  This beam corresponds to a gentle and
small bump on the tail of the electronic plasma velocity distribution
function (see, e.g., Section~9.4 of \cite{KT73}). The study of the bump-on-the-tail instability dates
back the pioneering work of Buneman \cite{B59} on the two-stream instability, where
each stream is considered as a mono-kinetic beam. 
Using the Nyquist method (see, e.g.,  Section~9.6 of \cite{KT73}) Penrose \cite{P60} derived a criterion, the so-called
Penrose criterion for instability,
under which the beam-plasma system distribution function (spatially
uniform and one-dimensional in the velocity variable) constitutes an unstable equilibrium
(i.e. an unstable stationary solution of the Vlasov--Poisson equations).
Then, any initial small perturbations of the beam-plasma system are
destabilized by the inversion of the electron population corresponding
to the positive slope interval of the velocity distribution. This
gives rise to electrostatic waves, which first grow linearly until the
beginning of a saturation stage, where the amplitude of waves reaches a
non-negligible value. In this resulting wave spectrum, the particle
dynamics becomes chaotic enough in their range of phase velocities, so
that  the  bump is eroded with eventually a plateau formation in the
distribution function.  Simultaneously, there is a transfer of
momentum from particles to electric waves, generating a turbulent
spectrum of waves.  This scenario was first predicted on a theoretical
basis \cite{VVS62, DP62} by considering the wave-particle interaction
as perturbative and neglecting all nonlinear wave-wave interactions in
the Vlasov--Poisson equation, except for their effect on the space-averaged
distribution function $f=f(t,v)$. This led to the set of QL
equations coupling the distribution function $f$ and the Fourier modes
$E(t,k)$ of the electric field,
\begin{eqnarray}
&&\partial_t f(t,v) -\partial_v (D_{QL}(t,v) \partial_v
  f(t,v))=0, \label{EDQL}\\ &&\partial_t |E(t,k)|^2 = 2\gamma(t,k)
  |E(t,k)|^2, \label{EDOE}
\end{eqnarray}  
where the QL diffusion coefficient is given by
\begin{equation}
\label{CDQL}
D_{QL}(t,v) =\pi \sum_{k\in\Z} |E(t,k)|^2\delta(\omega(t,k) -kv).
\footnote{In fact, this diffusion coefficient is a formal approximation for small $\gamma$
  of the quasilinear diffusion coefficient
  $D_{QL}(t,v) = \sum_{k\in\Z} |E(t,k)|^2 \frac{\gamma(t,k)}{(\omega(t,k)-kv)^2 +\gamma^2(t,k)}$
  in the resonant region of velocity space that corresponds to particle velocities
  satisfying $\omega(t,k)-kv=0$.}
\end{equation}
The real functions $(t,k)\mapsto\omega(t,k)$ and $(t,k)\mapsto \gamma(t,k)$
satisfy the following dispersion equation,
\begin{equation}
  \label{DEW}
\begin{aligned}  
&\mathbb{D}(k,\omega(t,k)+{\rm i}\gamma(t,k))=0, \quad \mbox{ with }\\
& \mathbb{D}(k,\omega(t,k)+{\rm i}\gamma(t,k)):= 1+
\frac{\omega_p^2}{k^2} \int_{\R} dv\, \frac{ k\partial_v
  f(t,v)}{\omega(t,k)-kv+{\rm i}\gamma(t,k)},
\end{aligned}
\end{equation}
where $\omega_p$ is  the plasma frequency \cite{VVS62, DP62,  Dav72,
  KT73}. The system \eqref{EDQL}-\eqref{DEW} is a closed and
self-consistent system of equations.  We recall that in the case of
the gentle-bump-on-the-tail instability an approximate solution of the
dispersion equation is given (see, e.g., \cite{KT73}) by the so-called
Bohm--Gross relation,
\begin{equation}
  \label{BGR}
\omega^2(t,k)\simeq\omega^2(k):=\omega_p^2(1+3k^2\lambda_D^2),
\end{equation}
with the Debye length $\lambda_D:= v_{th}/\omega_p$ and the thermal velocity squared
$v_{th}^2:= \int f v^2 dv/\int f dv$. The approximate growth rate is given by
\begin{equation}
  \label{QLGR}
\gamma(t,k) \simeq\frac \pi 2  \frac{\omega_p^2}{k^2}\,\omega(k)\int_{\R} dv \, \delta(\omega(k) -kv) k \partial_v
f(t,v).
\end{equation}
Therefore the system constituted by \eqref{EDQL}-\eqref{EDOE} and
\eqref{BGR}-\eqref{QLGR}, is also a closed and self-consistent
system. This approximate solution relies on the following assumptions:
$v\mapsto f(t,v)$ is even, $\gamma/\omega \ll 1$ (weak instability)
and $k\lambda_D \ll 1$ (long wavelength approximation) (see, e.g.,
\cite{KT73} for more details).  Let us note that the dispersion
equation and its approximate solution are the same as for the Landau
damping case, where the damping rate $\gamma$ given by \eqref{QLGR} is
negative because the slope of $f$ so is.

We must emphasize that, even
from a physical and physicists point of view, the derivation of
quasilinear theory from either a deterministic or a probabilistic
approach is actually not clear.  Indeed the original 1962 derivation
\cite{DP62,VVS62}, briefly exposed above, is deterministic. Right
after there were many other derivations of the QL theory, most of them
(see, e.g., \cite{Dav72, AAPSS75, Bal05} and references therein)
appeal to some statistical arguments, like the random phase
approximation (RPA), and invoke some time/space decorrelation
hypotheses.  From a numerical point of view, it has been shown in
\cite{BEEB11}, that a statistical ensemble average of solutions of
the Vlasov--Poisson system is required to recover a QL description
of the long time behavior of the weak warm beam-plasma system.

In this work, we consider both the self-consistent deterministic case
and  the non-self-consistent stochastic case. Here, the term
``self-consistent'' means that the electric field is produced by the
particles themselves, through the coupling with the Poisson equation.  In
the self-consistent deterministic case, we show that the asymptotic
distribution function is homogenized in the space variables, while the
self-consistent electric field converges weakly to zero. As already observed
in related works (e.g. \cite{BGC97}), we show
that the lack of time compactness for the electric field is compulsory
to obtain a non-trivial and thus a diffusion limit for the Vlasov
equation. By contrast, the time compactness property leads to a cheap
version of the Landau damping, where the electric field converges
strongly to zero (entailing a null diffusion matrix) and the
distribution function converges weakly to a spatially homogeneous
stationary solution of the Vlasov-Poisson system. Actually the difficult
part is to show a non-zero diffusion limit in the presence of fast time oscillations.
Using a Duhamel formula, we formally derive a diffusion equation for the asymptotic
distribution function, which depends only on the time and velocity
variables.  Unfortunately we are not
able to justify rigorously this diffusion  limit.  This task requires
a new approach, which will be the matter of a future work. It is
worthwhile to mention that in the nonlinear regime, the saturation of
the weak warm beam-plasma instability generates in phase space a type
of turbulence, which has a very close connection to Hamiltonian chaos
theory (see, e.g., \cite{BE97, Esc18} and references therein).  A
complete treatment of the self-consistent deterministic case remains
an open issue, the proof of which must be based at least (but not
only) on the same ingredients than those used for proving the Landau
damping  \cite{MV10}, and more particularly on the control of nonlinear
wave-wave interactions (e.g. plasma echoes).  In contrast with Landau
damping, the main and not the least difficulty is that perturbations
are not arbitrarily small, since wave amplitudes are amplified by the instability.
From a mathematical point of view, this makes the nonlinear wave-wave
interactions more difficult to control, especially for showing that
the latter remains negligible at least at the end of the relaxation
process.  The proof of QL diffusion for this ``inverse landau
damping'' problem remains a challenge.  Nevertheless, studying the
non-self-consistent problem remains meaningful. Indeed, it was observed
in numerical simulations of the self-consistent problem \cite{BEEB11},
that when the distribution function is enough phase-space homogenized
(after quite a long time), the problem falls into the
non-self-consistent framework, even in the strong nonlinear regime.

As explained in plasma physics literature (see, e.g., \cite{Dav72,
  AAPSS75}), diffusion in the QL theory comes from the time
decorrelation property of the electric field, which can be considered
as a random field.  We then place ourselves in similar modelling
hypotheses. This second framework is then  closer to the case of
particles evolving in a given bath of (random) waves \cite{CEV90,
  BE98,  EE03b, EP10, Els12} or particles subject to a reversible
reflection law, which has  convenient mixing properties \cite{BGC97}.
As a result, we prove that the asymptotic limit of the Vlasov equation
is a diffusion equation in the velocity space, where the diffusion
matrix is given by the space-time autocorrelation function of the
stochastic electric field, the so-called Reynolds electric stress tensor.
Hence the diffusion matrix is quadratic
with respect to the electric field. By specializing a little bit more
the structure of our electric field, we recover, at least from a
formal point of view, the diffusion matrix predicted by the QL theory.
Our diffusion matrix can also be related to a refinement of the QL
theory called the resonance broadening theory \cite{DU66, Aam67,
  Wei69, Pri69, Dav72}.  For the present problem and to our knowledge,
our results have not been found in the literature so far. For the
proof of the non-self-consistent stochastic case,  we follow the
strategy introduced in the paper \cite{PV03} which relies on
short-time decorrelation properties.  These techniques have been
successfully used in various physical contexts \cite{LV04, BPS06,
  Bra02, CDG07, GP07}.

The outline of this paper is as follows. Section~\ref{s:hwtr} describes
the weak turbulent regime, which is characterized by some
dimensionless parameters.  In Section~\ref{s:SCDC}, we deal with the
self-consistent deterministic case.  In Section~\ref{s:NSCSC}, we deal
with the non-self-consistent stochastic case.  Section~\ref{ss:tef}
collects all the hypotheses on the stochastic electric field.
Section~\ref{s:mthm} contains our main result about the diffusion
limit of the Vlasov equation, the proof of which is done in
Section~\ref{s:proof}. Finally Section~\ref{s:LinkKT}  connects our
result with some kinetic turbulence theories of plasma physics such as
the resonance broadening theory and  the  quasilinear theory.

\section{The weak turbulent regime}
\label{s:hwtr}

\subsection{Dimensionless parameters}
\label{ss:dp}
The  Vlasov--Poisson system, describing the self-consistent evolution
of the distribution function of particles $f=f(t,x,v)$ in an
electrostatic plasma, reads
\begin{eqnarray}
\label{eqn:V}
&&\partial_t f + v\cdot\nabla_xf + \frac{q}{m}E\cdot \nabla_{v}f =0,
\\ && \label{eqn:GL} E = - \nabla \Phi, \quad  -\Delta \Phi
=\frac{q}{\varepsilon_0} \left(\int_{\R^d}dv\, f - 1\right).
\end{eqnarray}
Here $t\in\R$, $x\in \mathbb{T}^d:=(\R/2\pi\Z)^d$, and ${v}\in\R^d$,
represent respectively time, position, and velocity of particles of
charge $q$ and mass $m$, which are accelerated by the ``turbulent''
electric field $E=E(t,x)$. Since the plasma is globally neutral we
have
\begin{equation}
  \label{GNP}
\frac{1}{(2\pi)^{d}}\int_{ \mathbb{T}^d} dx \int_{\R^d} dv\, f =1.
\end{equation}
In order to have a well-posed problem, we must add the zero-mean
electrostatic condition
\begin{equation}
  \label{ZMEC}
\int_{ \mathbb{T}^d}dx\, E= 0.
\end{equation}
Indeed the condition \eqref{ZMEC} is necessary to invert the Laplacian operator $\Delta$.
The phase space is denoted by $Q:=\mathbb{T}^d \times \R^d$.  In order
to write the Vlasov--Poisson system \eqref{eqn:V}-\eqref{eqn:GL} in a
dimensionless form, we need to introduce a time unit $\hat{t}$, a
length unit $\hat{x}$, a velocity unit $\hat{v}$, and typical
amplitudes $\widehat{E}$,  $\widehat{\Phi}$ and $\hat{f}$ for the
electric field, the electric potential and the distribution function
respectively.  The dimensionless variables and physical quantities
read
\begin{equation}
\label{eqn:adimVQ}
t'=\frac{t}{\hat{t}}, \ \ \  x'=\frac{x}{\hat{x}},
\ \ \ v'=\frac{v}{\hat{v}}, \ \  \ E'=\frac{E}{\widehat{E}},
\ \ \ \Phi'=\frac{\Phi}{\widehat{\Phi}}, \ \ \ f'=\frac{f}{\hat{f}}. 
\end{equation}
We set
\begin{equation}
  \label{def:n}
  \hat{n}:=\hat{f}\hat{v}^d,
\end{equation}
the typical value of the macroscopic (charge) density of particles.
Using  the Poisson equation \eqref{eqn:GL}, we obtain the following
dimensional equation,
\begin{equation}
\label{eqn:dimE}
  \widehat{E}=\hat{n}\hat{x}q/\varepsilon_0.
\end{equation}  
In an electrostatic plasma, the typical length scale is the Debye
length $\lambda_D$, while the typical velocity is the thermal velocity
$v_{th}$. The plasma frequency $\omega_p$, which is related to a
typical fast oscillation time of an electrostatic plasma, is then
given by
\begin{equation}
  \label{def:pf}
  \omega_p=\frac{v_{th}}{\lambda_D}.
\end{equation}  
The typical electric and kinetic energies are respectively
$\mathcal{E}_{el}=\varepsilon_0|\widehat{E}|^2$, and
$\mathcal{E}_{kin}= \hat{n}m\hat{v}^2$. The distribution function $f$
has a typical  evolution/relaxation  time $\tau_{rel}$, while the
turbulent electric field has two time scales. A slow time scale
$\tau_{L}$ is associated with the instantaneous growth or damping rate
$\gamma_L:=1/\tau_L$ of the electric field, while a fast time scale is
related to both the wave (electric field) autocorrelation time
$\tau_{ac}$ and the particle autocorrelation time $\tau_D$. The
time $\tau_{ac}$ is the lapse of time needed for a resonant particle,
traveling at the same velocity as the phase velocity of a typical wave,
to cross the localized spatial extent of the oscillatory electric
field disturbance. This time can also be seen as the time needed for a
resonant particle to resolve the finite frequency width of the wave
spectrum.  In other words the time  $\tau_{ac}$ can be seen  as the
turnover or the life time of a typical wave measured or felt by a
resonant particle traveling at the same velocity as the phase velocity
of this wave.  Then the synchronization between a wave and a
particle occurs in a lapse of time of the order of $\tau_{ac}$, during
which they interact by momentum transfer.  The time $\tau_D$ is the
autocorrelation or spreading time of particles, i.e.  the lapse of
time after which two close particles or orbits are completely
separated from each other.  In the plasma physics literature, the time
$\tau_D$ is called the Dupree time \cite{DU66}. A particle
distribution function evaluated at two different times separated by a
time interval of the order of $\tau_D$ is then decorrelated.  The
relaxation time $\tau_{rel}$ of the distribution function is then of
the order of $\tau_D$.  In the self-consistent case, where the Poisson
equation is used to compute the electric field from the particle
distribution function, this implies that two evaluations in time of
the electric field, separated by  a time interval of the order of
$\tau_D$, are also decorrelated.  We now set
\begin{equation}
\label{def:txv}
\hat{t}:= \tau_L, \quad \hat{x}:= \lambda_D, \quad \hat{v}:= v_{th}.
\end{equation}
Let $\varepsilon \in (0,1)$ be a small dimensionless parameter and
$\uptau \in]0,+\infty]$ be a positive dimensionless parameter, which may
    be finite or infinite.  Then, the weak turbulence regime of an
    electrostatic plasma is defined by  (see, e.g., Chapter~7 in
    \cite{Dav72}),
\begin{equation}
  \label{def:wtrc}
  \frac{\mathcal{E}_{el}}{\mathcal{E}_{kin}}=\frac{\varepsilon_0|\widehat{E}|^2}{\hat{n}m\hat{v}^2}=
  \varepsilon, \quad \frac{1}{\omega_p \hat{t}} = \varepsilon^2, \quad
  \frac{\tau_{ac}}{\hat{t}}= \varepsilon^2, \quad
  \frac{\tau_D}{\hat{t}}=\uptau.
\end{equation}
Using \eqref{eqn:adimVQ}-\eqref{def:wtrc}  and dropping the prime
notation for dimensionless variables and physical quantities, we
obtain from \eqref{eqn:V}-\eqref{eqn:GL}, the dimensionless
Vlasov--Poisson equations,
\begin{eqnarray}
\label{eqn:V2}
&&\partial_t f^\varepsilon +
\frac{v}{\varepsilon^2}\cdot\nabla_xf^\varepsilon +
\frac{E^\varepsilon}{\varepsilon}\cdot \nabla_{v}f^\varepsilon =0,
\quad (t,x,v) \in \R^+\times\mathbb{T}^d \times \R^d, \\
\label{eqn:GL2}
&&E^\varepsilon = - \nabla \Phi^\varepsilon, \quad  -\Delta
\Phi^\varepsilon = \int_{\R^d} dv\, f^\varepsilon -1.
\end{eqnarray}
The global neutrality condition \eqref{GNP} and the zero-mean electrostatic
condition \eqref{ZMEC} keep the same. We just have to substitute
$f^\varepsilon$ to $f$ in \eqref{GNP} and $E^\varepsilon$ to $E$ in
\eqref{ZMEC}.

\subsection{Notation}
\label{ss:Not}
In the rest of this paper, the notation $\langle \cdot, \cdot \rangle$
denotes the duality bracket between the space of distributions
$\mathcal{D}'(\R^+\times Q)$ and the space $\mathcal{D}(\R^+\times Q)$
of indefinitely differentiable functions with compact support in
$\R^+\times Q$.
The $L^2$-scalar product on the phase space $Q=\mathbb{T}^d\times \R^d$ is  defined by
\begin{equation}
  \label{def:L2SP}
( f, g ) := \frac{1}{(2\pi)^d}\int_{\mathbb{T}^d} dx \int_{\R^d} dv\, f g^\ast, \quad \forall f,g \in L^2(Q),  
\end{equation}
where the notation $(\cdot)^\ast$ stands for the complex conjugate. We
then have, for $g\in L_{\rm loc}^1(\R^+\times Q)$,
\begin{equation*}
  \label{def:bracket}
\langle g, \varphi \rangle := \int_{\R^+} dt\, (g,\varphi), \quad
\forall \varphi \in \mathcal{D}(\R^+\times Q).
\end{equation*}
We denote the space average on the torus $ \mathbb{T}^d$ by
\begin{equation*}
\label{def:A}
\fint dx\, g(t,x,v) := \frac{1}{(2\pi)^d} \int_{\mathbb{T}^d} dx\, 
g(t,x,v).
\end{equation*}
The one-parameter family of functions $\{g^\varepsilon\}_{\varepsilon> 0}$,
which we call sequences (respectively, subsequences) 
by abuse of language, must be understood as
generalized sequences (respectively, subsequences) such as nets (respectively, subnets)
in the sense of Moore--Smith or filters (respectively, finer filters) in the sense of Cartan
(for more details see, e.g., references \cite{NB89, Wil70}).
We also use the notation $\overline{g^\varepsilon}$ to denote the cluster point,
at least in the sense of distributions, of a family of functions
$\{g^\varepsilon\}_{\varepsilon> 0}$. We next define the free-flow
operator by
\begin{equation*}
  \label{def:ffo}
  \mathcal{L}:=v\cdot \nabla_x.
\end{equation*}
We note $t\mapsto S_t^\varepsilon$ the group on $L^p(Q)$,
$1\leq p\leq \infty$, generated by the free-flow operator
$\varepsilon^{-2}\mathcal{L}$.  Then, an explicit formula for the
group $t\mapsto S_t^\varepsilon$ is given by
\begin{equation}
  \label{def:Set}
  (S_t^\varepsilon g)(x,v)= \exp\Big(-\frac{t}{\varepsilon^2}
  \mathcal{L}\Big)g(x,v)=g(x-vt/\varepsilon^2,v), \quad \forall g\in
  L^p(Q).
\end{equation}
Eventually, the symbol $| \cdot |$ denotes either the modulus or the Euclidean norm
depending on whether we deal with complex/real scalars or vectors.       

\section{The self-consistent deterministic case}
\label{s:SCDC}
In this section we deal with the self-consistent deterministic case.

\subsection{Ergodic theorem}
We observe two orthogonal behaviors for the asymptotic limit of the
Vlasov-Poisson system \eqref{eqn:V2}-\eqref{eqn:GL2}, depending on
whether one makes or not the hypothesis of time compactness. The
cornerstone of such observations is the ergodic property of the
free-flow operator $\mathcal{L}$ on the torus,  we recall in
\begin{lemma}{(Ergodicity of the free flow on the torus)}
  \label{lem:EFF}
  Let $g\in L^p(Q)$, with $1\leq p \leq \infty$, satisfy
  \[
  \mathcal{L} g =0, \quad in \ \ \mathcal{D}'(Q).
  \] 
  Then $g=\fint dx\, g$ ($g$ is independent of the variable $x$).  
\end{lemma}

\begin{proof}
  We first integrate the free-flow operator by using characteristic
  curves, and second we use the spatial Fourier transform of the
  obtained solution. Indeed, the characteristic curves
  $(X(\tau),V(\tau))$ of the free flow satisfy the ODEs
  $\dot{X}(\tau)=V(\tau)$, $\dot{V}(\tau)=0$, with initial conditions
  $X(0)=x_0$, and $V(0)=v_0$. Its   solution is  given by
  $(X(\tau)=x_0+v_0\tau, \, V(\tau)=v_0)$. Since
  $(dg/d\tau)(X(\tau),V(\tau))= (v\cdot\nabla_{x}
  g)(X(\tau),V(\tau))=0$, we obtain $g(X(\tau),V(\tau))=g(X(0),V(0))$
  or $g(x,v)=g(x-v\tau,v)$, for a.e. $(x,v)\in Q$ and all $\tau\in\R^+$.
  The Fourier transform in space of this last equation gives
  \begin{equation*}
    \label{eqn:lemEFF:1}
    \big(1-\exp({\rm i} k\cdot v \tau  )\big) \hat{g}(k,v)=0, \quad
    \forall k\in \Z^d, \ v\in \R^d, \ \tau \in \R.
  \end{equation*}
  This relation and $g\in L^p(Q)$, with $1\leq p \leq \infty$  ($g$ cannot be a Dirac mass in velocity),
  imply that the support of $\hat g$ is contained in
  the set ${\rm supp}(\hat g):=\{(k,v)\in\Z^d \times \R^d \ |\ k\cdot
  v \tau \in 2\pi\Z, \ \forall \tau\in\R\}$. For any $\delta, \, T, \, r, \, R > 0$, such that $ \delta <T$,
  and $r<R$, the Lebesgue measure of the set ${\rm supp}(\hat g)$ for $\delta < \tau < T$
  and $ r <|v| <R$ is zero.  This forces $\hat g$ to
  be equal to zero for all $k\neq 0$ and then ${\rm supp}(\hat
  g):=\{(k,v) \ | \ k=0, \ \  v\in\R^d \}$. Therefore $g=\fint dx\,g$
  and $g$ is independent of the variable $x$. This ends the proof of
  Lemma~\ref{lem:EFF}. 
\end{proof}

\begin{remark}
  The Proof of Lemma~\ref{lem:EFF} is reminiscent of the Proof of Theorem 2.1 in
  \cite{CM98} but it is not exactly the same. Indeed, here we do not
  use the Riemann--Lebesgue lemma (in the Fourier dual variable of
  $v$) whereas \cite{CM98} does. More precisely the proof of Theorem 2.1 in
  \cite{CM98}, which is based  on the weak formulation of the equation
  $g(x,v)=g(x-v\tau,v)$ against continuous compactly supported test
  functions, uses first  the Fourier transform in the phase space $Q$
  to switch from real variables to Fourier variables in the weak
  formulation,  and it uses second the Riemann--Lebesgue lemma
  together with the  Lebesgue dominated convergence theorem to pass to
  the limit.
\end{remark}
  
Based on the ergodicity of the free flow, we state 

\begin{theorem}
  \label{th:SCDC}
Let $\{f_0^\varepsilon\}_{\varepsilon>0}$ be a sequence of non-negative
initial data  and $C_0$ be a positive constant such that
\begin{eqnarray*}
&&\|f_0^\varepsilon \|_{L^1(Q)} +\|f_0^\varepsilon \|_{L^\infty(Q)}  \leq C_0,  \quad \int_{Q}
f_0^\varepsilon |v|^2\,dxdv \leq C_0,  \quad \mbox{ and } \quad\\
&& \bigg\|E_0^\varepsilon:=\nabla \Delta^{-1} \bigg(\int_{\R^d}
f_0^\varepsilon dv-1\bigg) \bigg\|_{L^2(\mathbb{T}^d)} \leq C_0.
\end{eqnarray*}
Let $(f^\varepsilon,E^\varepsilon)_{\varepsilon >0}$, be a sequence of
weak solutions of the Vlasov--Poisson system
\eqref{eqn:V2}-\eqref{eqn:GL2}, with  initial data
${f^\varepsilon}_{|_{t=0}}=f_0^\varepsilon$, the existence of which
has been proved in \cite{Ars75, DL88a, DL88b, BGP00} for all
$\varepsilon >0$.  Then,
\begin{itemize}
\item [$i)$] There exists a function $f=f(t,v)$, independent of the
  variable $x$, such that $f\in L^\infty(\R^+; L^1\cap
  L^\infty(\R^d))$, and  up to subsequences one has,
  \begin{eqnarray*}
    f^\varepsilon &\rightharpoonup& f \ \mbox{ in
    }\  L^\infty(\R^+;L^\infty(Q)) \ \ \mbox{weak}\!-\!\ast, \\ \fint
    dx\, f^\varepsilon &\rightharpoonup& f  \ \mbox{ in }
    \ L^\infty(\R^+;L^\infty(\R^d)) \ \ \mbox{weak}\!-\!\ast.  
  \end{eqnarray*}
\item[$ii)$] The electric field $E^\varepsilon$ converges weakly to
  zero  as $\varepsilon \rightarrow 0$, more precisely, 
  \begin{equation*}
    E^\varepsilon \rightharpoonup 0 \ \mbox{ in }
    \ L^\infty(\R^+;W^{1,1+2/d}(\mathbb{T}^d))
    \ \ \mbox{weak}\!-\!\ast.
  \end{equation*}
\item[$iii)$] The expression
  \[
  \nabla_v \cdot \fint dx\, \frac{E^\varepsilon
    f^\varepsilon}{\varepsilon}, 
  \]
  is uniformly (with respect to $\varepsilon$) bounded in
  $\mathcal{D}'(\R^+\times \R^d)$; hence, up to a subsequence, it
  converges in  $\mathcal{D}'(\R^+\times \R^d)$ and we obtain
  \begin{eqnarray}
    &&\partial_t f + \nabla_v \cdot \overline {\fint dx\,
      \frac{E^\varepsilon f^\varepsilon}{\varepsilon}} = 0, \quad
    \mbox{ in } \  \mathcal{D}'(\R^+\times \R^d),
    \label{preDiffEq_1}\\
    && {f}_{|_{t=0}} = \fint dx\, f_0. \label{preDiffEq_2}
  \end{eqnarray}
\item[$iv)$] Let $d\leq 4$. Moreover, if we suppose that there exists
  a constant $\kappa$, independent of $\varepsilon$ such that
  \begin{equation}
    \label{LCIT}
    \| E^\varepsilon \|_{W_{\rm loc}^s(\R^+;L^1(\mathbb{T}^d))} \leq
    \kappa, \  \mbox{ with } s>0, \quad \mbox{ and } \quad \  \|
    \partial_t \Phi^\varepsilon \|_{L_{\rm
        loc}^1(\R^+;L^1(\mathbb{T}^d))} \leq \kappa,
  \end{equation}
  then, 
  \begin{eqnarray}
    \fint dx \,\frac{E^\varepsilon f^\varepsilon}{\varepsilon}
    &\rightharpoonup& 0
    \ \mbox{ in } \ \mathcal{D}'([0,T]\times \R^d),
    \label{wlofEFem1}\\ E^\varepsilon &\rightarrow & 0 \ \mbox{
      in } \ L^1([0,T]\times \mathbb{T}^d) \ \ \mbox{strong},
    \nonumber
  \end{eqnarray}
  as $\varepsilon \rightarrow 0$, and equations \eqref{preDiffEq_1}-\eqref{preDiffEq_2} degenerate
   into the following equations, 
   \begin{eqnarray}
     &&\partial_t f  = 0
     \quad \mbox{ in } \  \mathcal{D}'([0,T]\times\R^d), \label{eqn:trivial} \\
     && {f}_{|_{t=0}} = \fint dx\,
     f_0. \label{eqn:cite}
   \end{eqnarray}  
\end{itemize}
\end{theorem}  

\begin{proof}
  Since $\|f_0^\varepsilon \|_{L^1(Q)} + \|f_0^\varepsilon \|_{L^\infty(Q)}\leq  C_0 <\infty $, by weak
  compactness arguments there exists a function $f_0\in L^1\cap L^\infty(Q)$
  such that $f_0^\varepsilon$ (up to a subsequence) converges in
  $L^\infty(Q)$ weak--$\ast$ to $f_0$. Indeed from
  $\| f_0^\varepsilon\|_{L^1(Q)} + \| f_0^\varepsilon\|_{L^\infty(Q)} \leq C_0 <\infty$,
  and using standard weak compactness theorems, we obtain that
  there exists $f_0\in \mathcal{M}_b\cap L^\infty(Q)$ such that $f_0^\varepsilon \rightharpoonup f_0$
  in  $\mathcal{M}_b\cap L^\infty(Q)$ weak--$\ast$. Here, $\mathcal{M}_b(Q)$ is the set of bounded measures on $Q$. Moreover
  $\| f_0^\varepsilon\|_{L^1(Q)} + \| f_0^\varepsilon\|_{L^\infty(Q)} \leq C_0 <\infty$ implies
  that   $\| f_0^\varepsilon\|_{L^p(Q)}  \leq C_0 <\infty$,  for $1 \leq p \leq \infty$.
  Therefore we have also $f_0^\varepsilon \rightharpoonup f_0$  in  $L^p(Q)$ weak--$\ast$, 
  for $1 < p \leq \infty$. This and the De La Vall\'e--Poussin theorem on the criterion for uniform equi-integrability
  implies that the family $\{f_0^\varepsilon\}_{\varepsilon>0}$
  is uniformly equi-integrable. Finally, using Dunford--Pettis theorem, uniform equi-integrability implies that we also
  have  $f_0^\varepsilon \rightharpoonup f_0$  in  $L^1(Q)$ weak. Therefore,
  $f_0\in L^1\cap L^\infty(Q)$.
  
  From the standard theory of
  existence of weak solutions for the Vlasov--Poisson system
  \cite{Ars75, DL88a, DL88b, BGP00}, we obtain for all $\varepsilon
  >0$,
  \begin{equation}
    \label{ABF}
    \| f^\varepsilon(t)\|_{L^p(Q)} \leq \| f_0^\varepsilon\|_{L^p(Q)}
    \leq  C_0 <\infty,  \quad 1\leq p \leq \infty.
  \end{equation}
  Then, by weak compactness arguments there exists a function $f\in
  L^\infty(\R^+; L^1\cap L^\infty(Q))$, such that  $f^\varepsilon$
  (up to a subsequence) converges in $L^\infty(\R^+;L^\infty(Q))$
  weak--$\ast$ to $f$, and $\fint dx\, f^\varepsilon$  (up to a
  subsequence) converges in $L^\infty(\R^+;L^\infty(\R^d))$
  weak--$\ast$ to $\fint dx\, f$.  From properties of weak solutions
  for the Vlasov--Poisson system, weak solutions of
  \eqref{eqn:V2}-\eqref{eqn:GL2} satisfy the following a priori bound:
  \begin{equation}
    \label{tecl}
   \mathcal{E}(t) \leq  \mathcal{E}(0), \quad \forall t\geq 0, 
  \end{equation}
  where the total energy $\mathcal{E}(t)$ is defined by
  \begin{equation*}
    \label{tecl_2}
   \mathcal{E}(t):=\frac12\int_{Q} dxdv\,|v|^2f^\varepsilon(t,x,v)  +
   \frac{\varepsilon}{2}\int_{\mathbb{T}^d} dx\, |E^\varepsilon(t,x)|^2. 
  \end{equation*}
  From \eqref{tecl} and the initial data assumptions of
  Theorem~\ref{th:SCDC}, we infer that there exists a constant $K_0$,
  depending on $C_0$ such that
  \begin{equation}
    \label{eqn:keef}
    \int_{Q} dxdv\, |v|^2f^\varepsilon(t,x,v)  \leq K_0, \quad \mbox{ and
    } \quad \|E^\varepsilon\|_{L^\infty(\R^+;L^2(\mathbb{T}^d))} \leq
    \frac{K_0 }{\sqrt{\varepsilon}}.
  \end{equation}
  Taking $\varphi \in \mathcal{D}(\R^+\times Q)$ as a test function
  and using the $L^2$-scalar product \eqref{def:L2SP}, the weak
  formulation of Vlasov equation \eqref{eqn:V2} reads
  \begin{equation}
   \label{wfV2}
    \varepsilon^2\langle f^\varepsilon, \partial_t \varphi\rangle +
    \langle f^\varepsilon, v\cdot \nabla_x \varphi\rangle +
    \varepsilon\langle f^\varepsilon, E^\varepsilon \cdot \nabla_v
    \varphi\rangle=0.
  \end{equation}
  Using the $L^\infty$-bound for $f^\varepsilon$,  and a priori
  estimate \eqref{eqn:keef} for the electric field $E^\varepsilon$, we
  obtain
  \begin{equation}
    \label{eqnT1}
   |\varepsilon^2\langle f^\varepsilon, \partial_t \varphi\rangle|
   \leq \varepsilon^2 (2\pi)^{-d}\|\partial_t\varphi\|_{L^1(\R^+\times
     Q)}\|f^\varepsilon\|_{L^\infty(\R^+\times Q)} \leq
   C\varepsilon^2,
  \end{equation}
  and
  \begin{eqnarray}
   \varepsilon|\langle f^\varepsilon, E^\varepsilon \cdot \nabla_v
   \varphi\rangle| & \leq& \varepsilon
   \|E^\varepsilon\|_{L^\infty(\R^+;L^2( \mathbb{T}^d))}
   \|f^\varepsilon\|_{L^\infty(\R^+; L^2( Q))}
   \|\nabla_v\varphi\|_{L^1(\R^+; L^2(\R^d; L^\infty(\mathbb{T}^d))}
   \nonumber \\ &\leq& C\sqrt{ \varepsilon}. \label{eqnT3}
  \end{eqnarray}
  Using \eqref{eqnT1}-\eqref{eqnT3} to pass to the limit $\varepsilon
  \rightarrow 0$ in \eqref{wfV2}, we obtain
  \begin{equation*}
    v\cdot \nabla_x f = 0 \quad \mbox{ in }\  \mathcal{D}'(\R^+\times
    Q).  \label{eqnT4}
    \end{equation*} 
  From Lemma~\ref{lem:EFF}, we infer that $f$ is independent of $x$
  and  $\fint dx\, f(t)=f(t)$, for a.e. $t>0$. This proves point
  $i)$.  For proving point $ii)$,  we first define the charge
  density $\rho^\varepsilon$ by
  \[
  \rho^\varepsilon = \int_{\R^d} dv\, f^\varepsilon.
  \]
  From a standard interpolation inequality (see, e.g., \cite{DL88a,
    DL88b, BGP00}), there exists a constant $C_d$ depending on $d$
  such that
  \begin{equation}
    \label{eqn:II1}
    \| \rho^\varepsilon\|_{L^\infty(\R^+;L^{1+2/d}(\mathbb{T}^d))}
    \leq C_d \| f^\varepsilon\|_{L^\infty(\R^+\times Q)}^{2/(2+d)}
    \left\|  f^\varepsilon
    |v|^2\right\|_{L^\infty(\R^+;L^1(Q))}^{d/(d+2)}\leq \kappa_0<
    \infty,
  \end{equation}
  where the constant $\kappa_0$ depends on $C_0$, but is independent
  of $\varepsilon$. From the Poisson equation \eqref{eqn:GL2}, the bound
  \eqref{eqn:II1} on the charge density $\rho^\varepsilon$ and
  standard elliptic regularity estimates, we obtain
  \begin{equation}
    \label{eqn:II2}
    \| E^\varepsilon\|_{L^\infty(\R^+;W^{1,1+2/d}(\mathbb{T}^d))} \leq
    c_0 <\infty, 
  \end{equation}
  where the constant $c_0$ depends on initial data but is independent
  of $\varepsilon$.  Then,  by weak compactness there exists a
  function $E\in L^\infty(\R^+;W^{1,1+2/d}(\mathbb{T}^d))$ such that
  $E^\varepsilon$ (up to a subsequence) converges in $
  L^\infty(\R^+;W^{1,1+2/d}(\mathbb{T}^d))$ weak--$\ast$ to $E$. To
  determine the limit point $E$, we use  the Poisson equation
  \eqref{eqn:GL2}. Observing that
  \[
  \int_{\R^d}dv\,  f =1,
  \]
  and passing to the limit $\varepsilon\rightarrow 0$ in the Poisson
  equation \eqref{eqn:GL2}, we obtain
  \[
  \Delta \Phi =0, \ \ \mbox{ in } \ \mathcal{D}'(\R^+\times
  \mathbb{T}^d),
  \]
  which leads to $\Phi=0$ and $E=0$ in $\mathcal{D}'(\R^+\times
  \mathbb{T}^d)$. This ends the proof of point $ii)$.  For 
  point $iii)$, using the $L^2$-scalar product \eqref{def:L2SP}, we
  first write the following weak formulation of the Vlasov equation
  \eqref{eqn:V2} being previously averaged in space,
    \begin{equation}
      \label{AVVE}
      \langle  f^\varepsilon, \partial_t \varphi \rangle +
      \left\langle \frac{E^\varepsilon f^\varepsilon}{\varepsilon},
      \nabla_v\varphi \right\rangle = 0, \quad \forall \varphi \in
      \mathcal{D}(\R^+\times \R^d).
  \end{equation}
    Using a priori estimates \eqref{ABF} or point $i)$ of
    Theorem~\ref{th:SCDC}, we obtain from \eqref{AVVE},
  \begin{equation*}
    \left|  \int_{\R^+} dt \int_{\R^d} dv\,    \varphi \nabla_v\cdot
    \fint dx\, \frac{E^\varepsilon f^\varepsilon}{\varepsilon} \right|
    \leq \|f^\varepsilon\|_{L^\infty([0,T]\times Q)}
    \|\partial_t \varphi \|_{L^1(\R^+\times \R^d)} \leq C <\infty,
  \end{equation*} 
  where $C$ is independent of $\varepsilon$. This implies that
  $\nabla_v\cdot \fint dx\, E^\varepsilon f^\varepsilon/\varepsilon $
  (up to a subsequence) converges in $\mathcal{D}'(\R^+\times
  \R^d)$. Then, using   point $i)$ of Theorem~\ref{th:SCDC}, we can
  pass to the limit $\varepsilon \rightarrow 0$ in \eqref{AVVE} to
  obtain   equation \eqref{preDiffEq_1}.    For proving point
  $iv)$,  we start by establishing some strong convergence properties
  for the sequences $E^\varepsilon$ and $\Phi^\varepsilon$. Using
  $\| E^\varepsilon \|_{W_{\rm loc}^s(\R^+;L^1(\mathbb{T}^d))} \leq
  \kappa$, with $s>0$ (assumption \eqref{LCIT} of
  Theorem~\ref{th:SCDC}), and $E^\varepsilon\in
  L^\infty(\R^+;W^{1,1+2/d}(\mathbb{T}^d))$, we obtain, from a
  Lions--Aubin theorem \cite{Sim87}, that the sequence
  \begin{equation}
    \label{prop:sce}
    E^\varepsilon \ \mbox{ is compact in } \  L^{1+2/d}([0,T] \times
    \mathbb{T}^d) \ \mbox{ strong}, \quad \forall T>0.
  \end{equation}
  We next deal with the sequence  $\Phi^\varepsilon$.  Using the Poisson
  equation \eqref{eqn:GL2}, the bound \eqref{eqn:II1} on the charge
  density $\rho^\varepsilon$ and standard elliptic regularity
  estimates, we obtain that $\Phi^\varepsilon \in
  L^\infty(\R^+;W^{2,1+2/d}(\mathbb{T}^d))$. Using the Sobolev
  embedding $W^{s,p}(\R^d)\hookrightarrow W^{r,q}(\R^d)$, with
  $\ s>r$, $\ d>(s-r)p\ $ and $\ p\leq q \leq dp/(d-(s-r)p)$, we
  obtain
  \begin{equation}
    \label{prop:ess}
    L^\infty(\R^+;W^{2,1+2/d}(\mathbb{T}^d)) \hookrightarrow
    L^\infty(\R^+;W^{\delta,1+d/2}(\mathbb{T}^d)),
  \end{equation}
  with  $\max\{0,2-d/(1+2/d)\} < \delta <\max\{ 2, (-d^2+4d+4)/(d+2)
  \}$, and $d\leq 4$.  Using the embedding \eqref{prop:ess}, and the bound
  $\| \partial_t \Phi^\varepsilon \|_{L_{\rm
      loc}^1(\R^+;L^1(\mathbb{T}^d))} \leq \kappa$, (assumption
  \eqref{LCIT} of Theorem~\ref{th:SCDC}), we obtain from a
  Lions--Aubin theorem \cite{Sim87}, that the sequence
  \begin{equation}
    \label{prop:scphi}
    \Phi^\varepsilon \ \mbox{ is compact in } \  L^{1+d/2}([0,T]
    \times \mathbb{T}^d) \ \mbox{ strong}, \quad \forall T>0.
  \end{equation}
  Multiplying  the Vlasov equation \eqref{eqn:V2} by
  $\varepsilon\Phi^\varepsilon$, then averaging in space,
  multiplying the result by a test function $\varphi \in
  \mathcal{D}([0,T]\times \R^d)$ and finally integrating with respect
  to the time and  velocity variables, we obtain
  \begin{multline}
    \label{eqn:PhiV}
    \varepsilon \int_0^T dt \int_{\R^d} dv \fint dx \, \varphi
    \Phi^\varepsilon \partial_t f^\varepsilon + \frac{1}{\varepsilon}
    \int_0^T dt  \int_{\R^d} dv \fint dx \, \varphi  \Phi^\varepsilon
    v\cdot\nabla_x f^\varepsilon \\ +  \int_0^T dt  \int_{\R^d} dv
    \fint dx \, \varphi \Phi^\varepsilon \nabla_v\cdot (E^\varepsilon
    f^\varepsilon) =0.
  \end{multline}
  Using integration by parts, we obtain from \eqref{eqn:PhiV},
  \begin{multline}
    \label{eqn:PhiV2}
    \varepsilon \int_0^T dt  \int_{\R^d} dv \fint dx \, f^\varepsilon
    ( \Phi^\varepsilon  \partial_t \varphi + \varphi \partial_t
    \Phi^\varepsilon) - \int_0^T dt  \int_{\R^d} dv\, \varphi v \cdot
      \fint dx \,   \frac{E^\varepsilon f^\varepsilon }{\varepsilon}  \\ +
    \int_0^T dt  \int_{\R^d} dv \fint dx \,  f^\varepsilon
    \Phi^\varepsilon E^\varepsilon\cdot \nabla_v \varphi  =0.
  \end{multline}
  Using the $L^\infty$-bound \eqref{ABF} for $f^\varepsilon$ and
  assumption \eqref{LCIT}, we obtain, for the first term of
  \eqref{eqn:PhiV2},
  \begin{multline}
    \label{eqn:PV2T1}
      \varepsilon\left| \int_0^T dt  \int_{\R^d} dv \fint dx \,
      f^\varepsilon  ( \Phi^\varepsilon  \partial_t \varphi + \varphi
      \partial_t \Phi^\varepsilon) \right| \\ \leq \varepsilon
      (2\pi)^{-d}\|f^\varepsilon\|_{L^\infty([0,T]\times Q)} \left( \|
      \partial_t\varphi \|_{L^\infty(0,T;L^1(\R^d))} +  \| \varphi
      \|_{L^\infty(0,T;L^1(\R^d))} \right) \\\left( \|  \partial_t
      \Phi^\varepsilon \|_{L^\infty(0,T;L^1(\mathbb{T}^d))} +
      \|\Phi^\varepsilon \|_{L^\infty(0,T;L^1(\mathbb{T}^d))}  \right)
      \leq C\varepsilon.
  \end{multline}
  Using \eqref{prop:sce} and \eqref{prop:scphi}, we infer that the
  product $\Phi^\varepsilon E^\varepsilon$ converges strongly in
  $L^1([0,T] \times\mathbb{T}^d)$. Using this strong convergence, the
  weak convergence of $f^\varepsilon$ in $L^\infty([0,T]\times Q)$
  weak--$*$,  and the fact that the limit point of $E^\varepsilon$
  vanishes, we obtain for the third term of \eqref{eqn:PhiV2},
  \begin{equation}
    \label{eqn:PV2T3}
    \left| \int_0^T dt  \int_{\R^d} dv \fint dx \,  f^\varepsilon
    \Phi^\varepsilon E^\varepsilon\cdot \nabla_v \varphi\right |
    \longrightarrow 0, \ \mbox{ as } \ \varepsilon \rightarrow 0.
  \end{equation}
  Using  \eqref{eqn:PhiV2}-\eqref{eqn:PV2T3}, we obtain
  \begin{equation}
    \label{eqn:PV2T2}
   \int_0^T dt \int_{\R^d} dv \,\varphi v\cdot \fint dx \frac{ E^\varepsilon
      f^\varepsilon}{\varepsilon}\longrightarrow 0, \ \mbox{ as } \ \varepsilon
    \rightarrow 0.
  \end{equation}
  Choosing a test function $\varphi$ such that $0_{\R_v^d} \notin {\rm supp} (\varphi)$,
  and passing to the limit $\varepsilon\rightarrow 0$ in \eqref{AVVE},
  we then obtain  from \eqref{eqn:PV2T2},
  \begin{equation}
    \label{eqn:TrivEq}
      \int_0^T dt \int_{\R^d} dv \, f \partial_t \varphi=0. 
  \end{equation}  
  Since $f\in L^\infty([0,T]\times \R^d)$, equation \eqref{eqn:TrivEq} is valid
  for any $\varphi \in \mathcal{D}([0,T]\times\R^d)$; hence we obtain
  \eqref{wlofEFem1}-\eqref{eqn:cite}, which ends the proof.  
\end{proof}  

Few remarks on Theorem~\ref{th:SCDC} are in order.

\begin{remark}{(electric energy)}
  \label{rk:EE}
  A priori estimate \eqref{eqn:keef} for $E^\varepsilon$ is not
  optimal. In fact, we have $E^\varepsilon \in L^\infty(\R^+;
  L^2(\mathbb{T}^d))$ for $d\leq 4$, uniformly with respect to
  $\varepsilon$.  Indeed, for $d\leq 2$, it is obvious from estimate
  \eqref{eqn:II2}. For $3 \leq d \leq 4$, it comes from estimate
  \eqref{eqn:II2} and the Sobolev embedding
  $W^{1,1+2/d}(\R^d)\hookrightarrow L^{q}(\R^d)$, with $q =
  d(d+2)/(d-2)/(d+1)$ and $d>2$.
\end{remark}
  
\begin{remark}{(time compactness)}
\begin{itemize}
\item[1.]  To obtain time compactness there are a priori three
  ways. The first one is to obtain time compactness for the electric
  field by using standard control on the charge current
  $j^\varepsilon$ (see, e.g., \cite{DL88a, DL88b, BGP00}), and  the
  Amp\`ere equation given by $ \varepsilon^2\partial_t E^\varepsilon
  + j^\varepsilon=0.  $ This method fails because the presence of the factor
  $\varepsilon^2$ in front the time partial derivative does not give
  uniform bound  (with respect to $\varepsilon$) for $\partial_t
  E^\varepsilon$.  We obtain the same result from the charge
  conservation law (to obtain time compactness on the charge density
  $\rho^\varepsilon$, and thus on $E^\varepsilon$ via the Poisson
  equation), since the latter can be recovered by applying the spatial
  divergence operator to the Amp\`ere equation.  The second method is
  to use averaging lemmas \cite{DL89a}. With this method, we only
  obtain compactness in the space variables but not in the time
  variable, because in the limit $\varepsilon \rightarrow 0$, the term
  $ \varepsilon^2\partial_t f^\varepsilon$ in the Vlasov equation
  disappears \cite{DL89a, SRL09}. A third way is to obtain time
  compactness for the distribution function instead of the electric
  field. For this, we can show uniform convergence with respect to
  time in a weak topology for the phase-space variables. Showing time
  equi-continuity for the distribution function requires using the
  Vlasov equation. Here again, the presence of the factor $\varepsilon^2$ in
  front of the term $\partial_t f^\varepsilon$ in the Vlasov equation
  makes this method to fail.
\item[2.]  Point $iv)$ of Theorem~\ref{th:SCDC} shows that the lack of
  time compactness is in fact a necessary condition for obtaining a
 genuine or a non-degenerate diffusion equation in the limit  $\varepsilon \rightarrow 0$.
  \end{itemize}
\end{remark}

\begin{remark}{(boundary conditions)}
  As shown in the proof of Theorem~\ref{th:SCDC}, the relation $v\cdot
  \nabla_x f=0$, is a direct consequence of a priori estimates. By a
  direct computation using Fourier series (see the proof of
  Lemma~\ref{lem:EFF}) it has been proved  that the limit point $f$ is
  independent of the space variable $x$.  This is the ergodic property
  of the torus. Obviously the same property is true when the torus is
  replaced by any domain where the free flow trajectory
  $(x_0,v_0)\mapsto (x_0+v_0t,v_0)$ with specular reflections at the
  boundary are dense (this is a definition of ergodicity).  Extending
  the present analysis to this more general case may be very useful.
\end{remark}  

\begin{remark}{(``cheap'' Landau damping)}
  \begin{itemize}
    \item[1.]  What we proved for the rescaled Vlasov--Poisson system,
      given by \eqref{eqn:V2}-\eqref{eqn:GL2}, is a ``cheap'' version
      of the Mouhot--Villani version of the Landau damping
      \cite{MV10}, i.e.  that (under convenient hypotheses of
      regularity for initial conditions and smallness for initial
      perturbations) the self-consistent electric field $E$ of the
      Vlasov--Poisson system \eqref{eqn:V}-\eqref{eqn:GL} vanishes
      strongly when $t\rightarrow +\infty$, while the distribution $f$
      relaxes, in a weak topology, towards a spatially homogeneous
      stationary solution of the Vlasov-Poisson system. Indeed, if
      $\varepsilon$ is the ratio of the electric field $E(t)$ of
      \eqref{eqn:V}-\eqref{eqn:GL} to the electric field
      $E^\varepsilon(t)$ of \eqref{eqn:V2}-\eqref{eqn:GL2} (implying
      that $|E^\varepsilon(t)| > |E(t)|$), with the change of time
      scale $t\rightarrow t/\varepsilon^2$, the Vlasov--Poisson system
      \eqref{eqn:V}-\eqref{eqn:GL} becomes  the rescaled Vlasov--Poisson
      system \eqref{eqn:V2}-\eqref{eqn:GL2}. In other words, the limit
      $t\rightarrow +\infty$ in \eqref{eqn:V}-\eqref{eqn:GL} is
      equivalent  to the limit  $\varepsilon \rightarrow 0$ in
      \eqref{eqn:V2}-\eqref{eqn:GL2}, and $\varepsilon =1/\sqrt{t}$
      stands for the smallest rate at which the electric field $E(t)$
      tends to zero when  $t\rightarrow +\infty$.
  \item[2.]  By considering the rescaled Vlasov--Poisson system
    \eqref{eqn:V2}-\eqref{eqn:GL2} in the framework of the Landau
    damping, we observe that under the hypotheses of \cite{MV10} the
    electric field converges strongly to zero. From
    Theorem~\ref{th:SCDC}, this strong convergence corresponds to a
    zero diffusion. The Mouhot--Villani result \cite{MV10} is obtained
    for small perturbations (in some analytic norms) of a stable
    equilibrium profile (in velocity variables).  Here, we are
    interested in unstable equilibrium  profiles that lead to a non-zero
    diffusion in the velocity space.
  \end{itemize}  
\end{remark}

The velocity diffusion operator should arise  when we pass to the limit in
the term $ \varepsilon^{-1} \nabla_v \cdot  \fint dx\,  E^\varepsilon
f^\varepsilon.  $ A rigorous proof of this fact remains an open issue
and will be the matter of a future work.  Nevertheless, we can show,
at least formally, what is the structure of this term by using a
simple iteration of the Duhamel formula.  This is the aim of the next
section.

\subsection{Duhamel formula and Fick-type law}
\label{ss:IDF}
Here, we derive formally a Fick-type law for the flux term
\[
\overline{\fint dx \, \frac{E^\varepsilon
    f^\varepsilon}{\varepsilon}}, 
\]
appearing in \eqref{preDiffEq_1}.  Most of developments of this
section are formal, but they allow us to point out the difficulties for
showing rigorously the diffusion limit.  This Fick-type law can be obtained from two
ways. The first way is a global in time approach, which involves the
initial condition $f_0^\varepsilon$, while the second one, a local in
time approach, does not. Each approach has its advantages (Lemmas~\ref{lem:T1AG}~and~\ref{lem:QAL}) and
drawbacks (Remarks~\ref{RKOI1}~and~\ref{RKOI2}). In addition, for both approaches, the absence
of time decorrelation properties prevent us to
determine the structure and the properties of the diffusion
matrix. Nevertheless a formal WKB approximation allows us to obtain a
non-negative diffusion matrix in the non-self-consistent case.

\subsubsection{Global in time approach}
\label{ss:GIA}
Using the Duhamel formula and \eqref{def:Set}, we obtain from the Vlasov
equation \eqref{eqn:V2}, the following representation formula for
$f^\varepsilon(t)$, solution to \eqref{eqn:V2}-\eqref{eqn:GL2},
\begin{equation}
  \label{DFFF}
  f^\varepsilon(t) = S_t^\varepsilon f_0^\varepsilon
  -\frac{1}{\varepsilon}\int_0^tds \,S_{t-s}^\varepsilon E^\varepsilon(s)
  \cdot \nabla_v f^\varepsilon(s).
\end{equation}
Substituting \eqref{DFFF} into
\begin{equation*}
 -\int_{\R^+}dt \int_{\R^d} dv \, \varphi \nabla_v \cdot \fint dx\,  \frac{E^\varepsilon f^\varepsilon}{\varepsilon}
  =\frac{1}{\varepsilon} \int_{\R^+}dt \int_{\R^d} dv \fint dx\,  \nabla_v\varphi \cdot E^\varepsilon
 f^\varepsilon, \quad \forall \varphi\in
 \mathcal{D}(\R^+\times \R^d),
\end{equation*}
we obtain
\begin{equation}
  \label{decT1T2}
  -\int_{\R^+}dt \int_{\R^d} dv \,  \varphi \nabla_v \cdot \fint dx\,  \frac{E^\varepsilon f^\varepsilon}{\varepsilon}
  = T_1^\varepsilon(\varphi) + T_2^\varepsilon(\varphi),
\end{equation}  
where 
\begin{equation*}
  \label{def:T1}
  T_1^\varepsilon(\varphi):= \int_{\R^+} dt
  \int_{\R^d} dv\,  \frac{1}{\varepsilon}
  \nabla_v\varphi(t,v)\cdot \fint dx\, E^\varepsilon(t,x)
  f_0^\varepsilon(x-vt/\varepsilon^2,v),
\end{equation*}
and
\begin{multline*}
  \label{def:T2}
    T_2^\varepsilon(\varphi):= -\int_{\R^+} dt
   \int_{\R^d} dv\,
   \frac{1}{\varepsilon^2}\nabla_v\varphi(t,v)\cdot
   \\ \int_0^t ds\fint dx\,E^\varepsilon(t,x)
   E^\varepsilon(s,x-v(t-s)/\varepsilon^2)
   \cdot(\nabla_vf^\varepsilon)(s,x-v(t-s)/\varepsilon^2,v).
\end{multline*}
For the term $T_1^\varepsilon$ we have
\begin{lemma}
\label{lem:T1AG}
  Assume that $f_0^\varepsilon$ satisfies the hypotheses of
  Theorem~\ref{th:SCDC}. In addition we suppose that there exists
  a constant $\mathcal{C}_0$, independent of $\varepsilon$, such that for $|\alpha| \leq 1$,
  \begin{equation}
    \label{mixingH}
    \begin{aligned}
       & 
      \sum_{k\in \Z_*^d} \left(|k|^{-1} \| \partial_v^\alpha \hat{f}_0^\varepsilon(k)\|_{L^1(\R^d)}\right)^2
      \leq \mathcal{C}_0, \  & \mbox{ if } d=1, & \  \mbox{ and, }\\
     & \sum_{k\in \Z_*^d} \left(|k|^{-2} \| \partial_v^\alpha \hat{f}_0^\varepsilon(k)\|_{L^1(\R^d)}\right)^{1+2/d}
      \leq \mathcal{C}_0, \    & \mbox{ if } d\geq 2.
    \end{aligned}
    \end{equation} 
  Then
  \begin{equation}
    \label{eqnT1:1}
    T_1^\varepsilon \rightharpoonup 0 \  \mbox{ in }
    \   \mathcal{D}'(\R^+\times \R^d).
  \end{equation}
\end{lemma}  
\begin{proof}
  Using Fourier series  and the zero-mean electrostatic condition
  \eqref{ZMEC}, we rewrite the term $ T_1^\varepsilon$ as
\begin{equation*}
  \label{pT1_0}
  T_1^\varepsilon(\varphi) = \int_{\R^+} dt
  \int_{\R^d} dv  \sum_{k\in \Z_\ast^d}\, \frac{1}{\varepsilon}
  \nabla_v\varphi(t,v)\cdot \widehat{E}^\varepsilon(t,-k)
  \widehat{f}_0^\varepsilon(k,v) \exp(-{\rm
    i} k\cdot v t/\varepsilon^{2}).
\end{equation*}
Using a velocity integration by parts, we obtain
\begin{equation*}
  \label{pT1_1}
  T_1^\varepsilon(\varphi) = -{\rm i}
  \varepsilon \int_{\R^+} dt\, \frac{1}{t}\int_{\R^d} dv
  \sum_{k\in \Z_\ast^d}\, \nabla_v \cdot\Big( \frac{k}{|k|^2}
  \nabla_v\varphi(t,v)\cdot
  \widehat{E}^\varepsilon(t,-k) \widehat{f}_0^\varepsilon(k,v) \Big)
  \exp(-{\rm i} k\cdot v t/\varepsilon^{2}),
\end{equation*}
which leads to
\begin{eqnarray}
  \label{pT1_2}
  | T_1^\varepsilon(\varphi)| &\leq&
  \varepsilon\int_{\R^+} dt\, \frac{1}{t} \sum_{k\in
    \Z_\ast^d}\, |\widehat{E}^\varepsilon(t,k)| |k|^{-1} \nonumber
  \\ && \big( \|\nabla_v^2\varphi(t)\|_{L^{\infty}(\R^d)}
  \| \widehat{f}_0^\varepsilon(k)\|_{L^1(\R^d)}
  +\|\nabla_v\varphi(t)\|_{L^{\infty}(\R^d)} \|\nabla_v
  \widehat{f}_0^\varepsilon(k)\|_{L^1(\R^d)} \big).
\end{eqnarray}
Using the bound \eqref{eqn:II2} and the Hausdorff-Young inequality, we obtain
for $d\geq 2$,
\begin{equation}
  \label{eqnBEE2}
  \||k|\widehat{E}^\varepsilon\|_{L^\infty(\R^+;\ell^{1+d/2}(\Z^d))}
  \leq
  \|{E}^\varepsilon\|_{L^\infty(\R^+;W^{1,1+2/d}(\mathbb{T}^d))} \leq
  c_0.
\end{equation}
Using H\"older's inequality, \eqref{eqnBEE2} and the assumption \eqref{mixingH}, we obtain from
\eqref{pT1_2},
\begin{equation*}
  \label{pT1_3}
  | T_1^\varepsilon(\varphi)| \leq
  \varepsilon  2c_0\mathcal{C}_0^{d/(d+2)}
  \|\varphi/t\|_{L^{1}(\R^+;W^{2,\infty}(\R^d))}.
\end{equation*}
In the same way, using  Remark~\ref{rk:EE} and the Cauchy-Scharwz
inequality, we obtain for $d=1$,
\begin{equation*}
  \label{pT1_4}
  |T_1^\varepsilon(\varphi)| \leq
  \varepsilon 2 
  \mathcal{C}_0^{1/2} 
  \|{E}^\varepsilon\|_{L^\infty(\R^+;L^2(\mathbb{T}^d))}
  \|\varphi/t\|_{L^{1}(\R^+;W^{2,\infty}(\R^d))},
\end{equation*}
which ends the proof of Lemma~\ref{lem:T1AG}.
\end{proof}

\begin{remark}
  In Lemma~\ref{lem:T1AG}, the regularity assumption for $f_0^\varepsilon$
  might be refined but with the presence of the factor
  $\varepsilon^{-1}$ in the term $T_1^\varepsilon$, some mixing-type
  hypotheses seem compulsory.
\end{remark}
  
We now deal with the term $T_2^\varepsilon$.   Performing the change of
time variable $s=t-\varepsilon^2\sigma$, followed by the change of
space variable $x'=x-\sigma v$, and  using a velocity integration by
parts and $x$-periodicity, we obtain
\begin{multline}
\label{eqnT2:1}  
 T_2^\varepsilon(\varphi) = \int_{\R^+} dt \int_{\R^d}
dv \fint dx \int_0^{t/\varepsilon^2} d\sigma \,
\\ f^\varepsilon(t-\varepsilon^2\sigma,x,v) \nabla_v \cdot
\big( E^\varepsilon(t-\varepsilon^2\sigma,x) \otimes
E^\varepsilon(t,x+ v\sigma ) \nabla_v \varphi(t,v)\big).
\end{multline}
Using the time characteristic function $\chi_{[0,t/\varepsilon^2]}(\sigma)$,
equation \eqref{eqnT2:1} can be recast as
\begin{equation}
\label{eqnT2:2}  
T_2^\varepsilon(\varphi) = 
{J}^\varepsilon(\varphi) +  {M}^\varepsilon(\varphi),
\end{equation}
where
\begin{equation*}
  \label{def:Q}
  {J}^\varepsilon(\varphi) = \int_{\R^+} dt
  \int_{\R^d} dv\,  f(t,v) \nabla_v \cdot
  \bigg(\int_{\R^+} d\sigma \fint dx\,
  \chi_{[0,t/\varepsilon^2]}(\sigma) E^\varepsilon(t-\varepsilon^2\sigma,x) \otimes
  E^\varepsilon(t,x+ v\sigma )\, \nabla_v\varphi(t,v)\bigg),
\end{equation*}
and
\begin{multline*}
\label{def:M}
 {M}^\varepsilon(\varphi)  := \int_{\R^+} dt
\int_{\R^d} dv\,  \int_{\R^+}d\sigma \\ \fint dx\,  \chi_{[0,t/\varepsilon^2]}(\sigma)
\big(f^\varepsilon(t-\varepsilon^2\sigma,x,v) - f(t,v) \big) \nabla_v
\cdot \big( E^\varepsilon(t-\varepsilon^2\sigma,x) \otimes
E^\varepsilon(t,x+ v\sigma ) \nabla_v \varphi(t,v)\big).
\end{multline*}
If we assume that
\begin{equation}
  \label{limQexists}
  \lim_{\varepsilon\rightarrow 0}  \, {J}^\varepsilon(\varphi)
   \ \ \mbox{exists},
\end{equation}  
and 
\begin{equation}
  \label{limM0}
  \lim_{\varepsilon\rightarrow 0} \,  {M}^\varepsilon (\varphi)
   =0,
\end{equation}  
then we obtain from  \eqref{eqnT2:2}, 
\begin{equation}
\label{eqnT2:4}  
\lim_{\varepsilon\rightarrow 0} \,T_2^\varepsilon(\varphi)
=  \int_{\R^+} dt \int_{\R^d} dv\,  f(t,v) \nabla_v \cdot
\big(\mathscr{D}(t,v)^T \nabla_v\varphi(t,v)\big),
\end{equation}
with
\begin{eqnarray}
  \mathscr{D}(t,v)&:=& \lim_{\varepsilon\rightarrow 0}
  \int_{\R^+}d\sigma 
  \fint dx\, \chi_{[0,t/\varepsilon^2]}(\sigma)
  E^\varepsilon(t,x+ v\sigma )\otimes  E^\varepsilon(t-\varepsilon^2\sigma,x) \nonumber\\ &=&
  \lim_{\varepsilon\rightarrow 0} \int_{\R^+}d\sigma \fint dx\,
  \chi_{[0,t/\varepsilon^2]}(\sigma) 
  E^\varepsilon(t,x)\otimes E^\varepsilon(t-\varepsilon^2\sigma,x-v\sigma).\label{def:CDif}
\end{eqnarray}
Using \eqref{decT1T2}, \eqref{eqnT1:1} and \eqref{eqnT2:4} to pass to
the limit $\varepsilon \rightarrow 0$ in \eqref{AVVE}, we obtain the
following diffusion equation,
\begin{equation}
  \label{def:DifEq}
  \partial_t f(t,v) - \nabla_v \cdot  \big(\mathscr{D}(t,v) \nabla_v
  f(t,v)\big)=0, \ \ \mbox{ in } \  \mathcal{D}'([0,T]\times \R^d).
\end{equation}
Few remarks are now in order.
\begin{remark}{(open issues)}
  \label{RKOI1}
  \begin{itemize}
 \item[1.] All computations involving the term $T_2^\varepsilon$ are
   formal and must be justified in a convenient functional
   framework. In order to
   justify \eqref{limQexists}, we have to show that
   \begin{equation*}
     \label{divvEE}
     R^\varepsilon(t,\sigma,x,v):= \chi_{[0,t/\varepsilon^2]}(\sigma)\nabla_v \cdot \big(
     E^\varepsilon(t-\varepsilon^2\sigma,x) \otimes
     E^\varepsilon(t,x+ v\sigma ) \nabla_v \varphi(t,v)\big)
   \end{equation*}
   converges weakly in $L^1(\R_t^+\times \R_\sigma^+ \times Q)$.  In
   order to prove \eqref{limM0} and justify \eqref{eqnT2:4}, we have
   to show that $R^\varepsilon$ converges strongly in
   $L^1(\R_t^+\times \R_\sigma^+ \times Q)$, since
   $f^\varepsilon(t-\varepsilon^2\sigma,x,v) - f(t,v)  \rightharpoonup
   0 $ in $L^\infty(\R_t^+\times \R_\sigma^+ \times Q)$ weak--$\ast$.
   We observe that a crucial point is to obtain enough integrability
   with respect the time variable $\sigma$, uniformly in $\varepsilon$.
   
\item[2.] As already observed, the bound $\|
  E^\varepsilon\|_{L^\infty(\R^+;W^{1,1+2/d}(\mathbb{T}^d))} \leq c_0
  <\infty$, does not imply strong convergence (because of the lack of
  time control or compactness) for the electric field $E^\varepsilon$,
  which would help to justify the above formal computations for the
  term $T_2^\varepsilon$.  However, this lack of time compactness is in
  fact necessary if we do not want to obtain a trivial equation, as
  stated in  point $iv)$ of Theorem~\ref{th:SCDC}. Indeed, from 
  point $iv)$ of Theorem~\ref{th:SCDC}, time compactness entails a
  strong convergence to zero of the electric field
  $E^\varepsilon$. This implies the vanishing of the diffusion matrix
  $\mathscr{D}$ given by  \eqref{def:CDif}. Without time compactness,
  the electric field $E^\varepsilon$ always converges weakly to zero,
  but not the quadratic electric tensor $E^\varepsilon\otimes
  E^\varepsilon$ (this is a property of weak convergence), which
  implies a non-trivial diffusion matrix $\mathscr{D}$.  Therefore,
  weak convergence seems mandatory to obtain a diffusion limit.
\item[3.]  Instead of time compactness, time decorrelation properties
  could help to justified rigorously above computations.  In the
  presence of a non-self-consistent but stochastic electric field,
  with convenient hypotheses, some time decorrelation properties allow
  us to justify rigorously the limit of the Vlasov equation \eqref{eqn:V2}
  towards diffusion equations \eqref{def:CDif}-\eqref{def:DifEq}.
  This is the object of Section~\ref{s:NSCSC}. 
  \end{itemize}
\end{remark}

\begin{remark}{(periodic or quasi-periodic time oscillations)}
  Since the defect of time compactness means that the system contains
  fast oscillations in time, it would be tempting to apply the
  analysis of this section to the case of a non-self-consistent
  deterministic electric field (satisfying convenient regularity
  assumptions) with two time scales, one being slow and not periodic,
  and the other being fast and (quasi-)periodic. Such a standard
  homogenization problem would lead to solve a hierarchy of equations
  where the free-streaming operator $v\cdot \nabla_x$, with periodic
  boundary condition, must be inverted at each stage of the
  hierarchy. Since the free-flow operator is not a Fredholm operator,
  the Fredholm alternative does not hold for such transport
  equation. On the contrary, when considering a non-self-consistent
  stochastic electric field the situation is completely different and
  we can obtain a diffusion behavior for the statistical average of
  the distribution function.  This is what is done in
  Section~\ref{s:NSCSC}.
\end{remark}
    {\bf \noindent An explicit form of the diffusion matrix in the
      non-self-consistent deterministic case}.  \\
\newline
\noindent
Here we pursue a little bit further the above formal analysis to
explicit the structure of the diffusion matrix \eqref{def:CDif} by
constructing a well-suited non-self-consistent deterministic electric
field.  This gives an example of diffusion matrix \eqref{def:CDif},
which is not zero and non-negative.  Introducing the Fourier series
decomposition of $E^\varepsilon$,
\begin{equation*}
\label{eqn_FSD}  
E^\varepsilon(t,x) =\sum_{k\in \Z^d} e^{\, {\rm i}k\cdot
  x}\widehat{E}^\varepsilon(t,k),
\end{equation*}
we assume the formal WKB expansion for the Fourier mode
$\widehat{E}^\varepsilon(t,k)$,
\begin{equation}
  \label{WKB_exp}
\widehat{E}^\varepsilon(t,k) = \sum_{j\geq 0} \varepsilon^j
\widehat{E}_j(t,k,\Omega(t,k)/\varepsilon^2),
\end{equation}
where complex vector-valued functions $(k,\tau) \mapsto
\widehat{E}_j(t,k,\tau)$ are $2\pi$-periodic with respect to the
variable $\tau$.  Here, functions
$\widehat{E}_j(t,k,\Omega(t,k)/\varepsilon^2)$ are Hermitian, i.e.
\[
\widehat{E}_j^\ast(t,k,\Omega(t,k)/\varepsilon^2)=\widehat{E}_j(t,-k,\Omega(t,-k)/\varepsilon^2),
\]
and the real-valued function $k\mapsto\Omega(t,k)$ is odd with respect
to the variable $k$. As a first approximation of \eqref{WKB_exp}, we
obtain
\begin{equation}
  \label{WKB_app}
  \widehat{E}^\varepsilon(t,k) = \widehat{E}_0(t,k)\exp\left(-{\rm i}
  \frac{\Omega(t,k)}{\varepsilon^2}\right) + \mathcal{O}(\varepsilon),
\end{equation}
where the real vector-valued function $k\mapsto \widehat{E}_0(t,k)$ is
even with respect to the variable $k$.   Using \eqref{WKB_app} and
time Taylor expansions,  we obtain from the definition of the
diffusion matrix \eqref{def:CDif},
\begin{align}
  \mathscr{D}(t,v)  = & \lim_{\varepsilon \rightarrow 0}
  \int_{\R^+}d\sigma \sum_{k\in \Z^d}   \chi_{[0,t/\varepsilon^2]}(\sigma)\exp\big({\rm
    i}  k\cdot v \sigma\big)
 \widehat{E}^\varepsilon(t,k) \otimes
 \widehat{E}^\varepsilon(t-\varepsilon^2\sigma,-k) 
 \nonumber\\ =&\lim_{\varepsilon
    \rightarrow 0} \bigg(  \int_{\R^+}d\sigma  \sum_{k\in
    \Z^d}  \chi_{[0,t/\varepsilon^2]}(\sigma) \exp\big({\rm i} [ k\cdot v \sigma +
    \Omega(t-\varepsilon^2\sigma,k)/\varepsilon^2 -
    \Omega(t,k)/\varepsilon^2]\big) \nonumber\\ &
    \widehat{E}_0(t,k)\otimes \widehat{E}_0(t-\varepsilon^2\sigma,-k) 
  +  \mathcal{O}(\varepsilon) \bigg)\nonumber\\ =&\lim_{\varepsilon
    \rightarrow 0} \bigg( \int_{\R^+}d\sigma \sum_{k\in
    \Z^d}  \chi_{[0,t/\varepsilon^2]}(\sigma) \exp\big(-{\rm i}\sigma
  [\partial_t\Omega(t,k)- k\cdot v ]\big)
  \widehat{E}_0(t,k) \otimes  \widehat{E}_0(t,k) +
  \mathcal{O}(\varepsilon)\bigg).\label{twqld}
\end{align}
Using   
\[
\lim_{\varepsilon \rightarrow 0 }\int_{\R^+}  d\sigma\, \chi_{[0,t/\varepsilon^2]}(\sigma) e^{-{\rm i} \sigma \tau } = \pi \delta(\tau) -
    {\rm i\,}{\rm p.v.}\left(\frac{1}{\tau}\right) \   \mbox{ in } \mathcal{D}'(\R), 
\]
and parity of functions $\Omega(t,k)$ and $\widehat{E}_0(t,k)$, we
obtain from \eqref{twqld},
\begin{equation}
  \mathscr{D}(t,v)  = \pi  \sum_{k\in \Z^d}  \widehat{E}_0(t,k)
  \otimes  \widehat{E}_0(t,k) \delta\big(\partial_t\Omega(t,k)-
  k\cdot v\big).\label{SCDQL}
\end{equation}
If we assume $\Omega(t,k)=\omega(k) t$, then
\eqref{SCDQL} is the diffusion matrix given by the quasilinear theory
\cite{KT73, Dav72}, i.e.
\begin{equation}
  \label{DQLTXTB}
\mathscr{D}(t,v)  = \pi  \sum_{k\in \Z^d}  \widehat{E}_0(t,k) \otimes
\widehat{E}_0(t,k) \delta\big(\omega(k)-k\cdot v \big).
\end{equation}

\begin{remark}
  \label{rem:DEGPHI}
  If the electric field $E^\varepsilon$ derives from
  a potential $\Phi^\varepsilon$, i.e., $E^\varepsilon(t,x)=-\nabla \Phi^\varepsilon(t,x)$, then,
  following the above computations, a WKB expansion of $\Phi^\varepsilon(t,x)$ similar to \eqref{WKB_exp} leads 
  to the diffusion matrix
\begin{equation*}
\mathscr{D}(t,v)  = \pi  \sum_{k\in \Z^d}  |\widehat{E}_0(t,k)|^2\, \frac{k\otimes k}{|k|^2}\,
\delta\big(\omega(k)-k\cdot v \big).
\end{equation*}
\end{remark}
  
\subsubsection{Local in time approach}
\label{ss:LITA}
We first integrate, with respect to the time variable, the space-averaged Vlasov equation
\begin{equation*}
\label{eqn_pre_dif}
 \partial_t \fint dx\, f^\varepsilon + \nabla_v \cdot
 \fint dx \,\frac{E^\varepsilon f^\varepsilon}{\varepsilon}=0, 
\end{equation*}
between the time $t$ and $t+\theta$, with $\theta> 0$.  Then, we
multiply the result by $\varphi\in  \mathcal{D}(\R^+\times \R^d)$ and 
we perform an integration with respect to the time and velocity variables.  
Finally, using the $L^2$-scalar product \eqref{def:L2SP} and a
velocity integration by parts, we obtain
\begin{equation}
\label{lita:eqn:1}
\bigg\langle \frac{f^\varepsilon(t+\theta) -
  f^\varepsilon(t)}{\theta}, \varphi \bigg\rangle
=\frac{1}{\varepsilon\theta}\int_{\R^+} dt \int_{\R^d}dv
\int_{t}^{t+\theta}ds \fint dx \,
f^\varepsilon(s){E^\varepsilon(s)}\cdot \nabla_v \varphi(t,v).
\end{equation}
Using the Duhamel formula and the notation \eqref{def:Set}, we obtain from the Vlasov
equation \eqref{eqn:V2}, the following representation formula for
$f^\varepsilon(s)$,
\begin{equation}
  \label{lita:eqn:2}
  f^\varepsilon(s) = S_{s-t+\hat{\theta}}^\varepsilon
  f^\varepsilon(t-\hat{\theta}) -\frac{1}{\varepsilon}
  \int_{t-\hat{\theta}}^s d\sigma\, S_{s-\sigma}^\varepsilon
  E^\varepsilon(\sigma) \cdot \nabla_v f^\varepsilon(\sigma),
\end{equation}
with $\hat{\theta}$ being an arbitrary non-negative time.  Substituting
\eqref{lita:eqn:2} into \eqref{lita:eqn:1}, we obtain
\begin{equation}
\label{lita:eqn:3}
\bigg\langle \frac{f^\varepsilon(t+\theta) -
  f^\varepsilon(t)}{\theta}, \varphi \bigg\rangle = 
\mathcal{T}_{1}^\varepsilon(\varphi)  + 
\mathcal{T}_{2}^\varepsilon(\varphi), 
\end{equation}
where
\begin{equation}
\label{lita:eqn:T1}
\mathcal{T}_{1}^\varepsilon(\varphi) :=  \int_{\R^+}
dt \int_{\R^d}dv \int_{t}^{t+\theta}ds \fint dx\, \frac{1}{\varepsilon \theta}
 {E^\varepsilon(s)}\cdot \nabla_v\varphi(t,v)
S_{s-t+\hat{\theta}}^\varepsilon f^\varepsilon(t-\hat{\theta}),
\end{equation}
and 
\begin{multline*}
\label{lita:eqn:T2}
\mathcal{T}_{2}^\varepsilon( \varphi) :=  \int_{\R^+}
dt \int_{\R^d}dv
\\ \int_{t}^{t+\theta}ds\int_{t-\hat{\theta}}^{s}d\sigma \fint dx\,
\frac{1}{\varepsilon^2 \theta} S_{s-\sigma}^\varepsilon
f^\varepsilon(\sigma) \nabla_v \cdot \Big( S_{s-\sigma}^\varepsilon
E^\varepsilon(\sigma,x)\otimes E^\varepsilon(s,x) \nabla_v
\varphi(t,v)\Big). 
\end{multline*}
For the term $ \mathcal{T}_1^\varepsilon$, we assume that
\begin{equation}
  \label{T1AL}
  \mathcal{T}_1^\varepsilon \rightharpoonup 0 \  \mbox{ in }
  \   \mathcal{D}'(\R^+\times \R^d).
\end{equation}

We now deal with the term $\mathcal{T}_{2}^\varepsilon$. Using
$x$-periodicity and the change of time variable $s'=s-t$, the term
$\mathcal{T}_{2}^\varepsilon$ becomes
\begin{multline}
\label{lita:eqn:T2-2}
\mathcal{T}_{2}^\varepsilon(\varphi) =  \int_{\R^+}
dt \int_{\R^d}dv
\int_{0}^{\theta}ds\int_{t-\hat{\theta}}^{t+s}d\sigma \fint dx\,
\\ \frac{1}{\varepsilon^2 \theta}
f^\varepsilon(\sigma,x,v) \nabla_v \cdot \Big(
E^\varepsilon(\sigma,x)\otimes
E^\varepsilon(t+s,x+v(t+s-\sigma)/\varepsilon^2) \nabla_v
\varphi(t,v)\Big). 
\end{multline}
Using the change of time variable
$\sigma'=(t+s-\sigma)/\varepsilon^2$, equation \eqref{lita:eqn:T2-2}
becomes
\begin{multline}
\label{lita:eqn:T2-3}
\mathcal{T}_{2}^\varepsilon(\varphi) =  \int_{\R^+}
dt \int_{\R^d}dv
\int_{0}^{\theta}ds\int_{0}^{(s+\hat{\theta})/\varepsilon^2}d\sigma \fint dx\,
\\ \frac{1}{\theta} 
f^\varepsilon(t+s-\varepsilon^2\sigma,x,v) \nabla_v \cdot \Big(
E^\varepsilon(t+s-\varepsilon^2\sigma,x)\otimes
E^\varepsilon(t+s,x+v\sigma) \nabla_v \varphi(t,v)\Big). 
\end{multline}
Taking $\theta=\varepsilon^2 \uptau\ $ and $\ \hat{\theta}= \varepsilon^2
\eta$, and using the change of time variable $s'=s/\varepsilon^2$,
followed by the change of time variable $\sigma'=\sigma -\eta$,
equation \eqref{lita:eqn:T2-3} becomes
\begin{multline*}
\label{lita:eqn:T2-4}
\mathcal{T}_{2}^\varepsilon(\varphi) =  \int_{\R^+}
dt \int_{\R^d}dv \, \frac{1}{\uptau}
\int_{0}^{\uptau}ds\int_{-\eta}^{s}d\sigma   \fint dx\,
f^\varepsilon(t+(s-\eta-\sigma)\varepsilon^2,x,v)  \\ \nabla_v \cdot
\Big( E^\varepsilon(t+ (s-\eta-\sigma)\varepsilon^2,x)\otimes
E^\varepsilon(t+s\varepsilon^2,x+v(\sigma+\eta)) \nabla_v
\varphi(t,v)\Big). 
\end{multline*}
This equation can be recast as
\begin{equation}
  \label{T2-DecQM}
  \mathcal{T}_{2}^\varepsilon(\varphi)= 
  \mathcal{J}^\varepsilon(\varphi) + 
  \mathcal{M}^\varepsilon(\varphi),   
\end{equation}
where
\begin{multline*}
\label{lita:eqn:T2-5}  
\mathcal{J}^\varepsilon(\varphi)=  \int_{\R^+} dt
\int_{\R^d}dv f(t,v)   \nabla_v \cdot \bigg(  \frac{1}{\uptau}
\int_{0}^{\uptau}ds\int_{-\eta}^{s}d\sigma \fint dx\,  \\ 
E^\varepsilon(t+ (s-\eta-\sigma)\varepsilon^2,x)\otimes
E^\varepsilon(t+s\varepsilon^2,x+v(\sigma+\eta)) \nabla_v
\varphi(t,v)\bigg), 
\end{multline*}
and
\begin{multline*}
 \mathcal{M}^\varepsilon(\varphi)  := \int_{\R^+} dt
\int_{\R^d} dv\,  \frac{1}{\uptau}
\int_{0}^{\uptau}ds\int_{-\eta}^{s}d\sigma \fint dx\, \big
(f^\varepsilon(t+(s-\eta-\sigma)\varepsilon^2,x,v) -
f(t,v)\big)\\ \nabla_v \cdot \Big( E^\varepsilon(t+
(s-\eta-\sigma)\varepsilon^2,x)\otimes
E^\varepsilon(t+s\varepsilon^2,x+v(\sigma+\eta)) \nabla_v \varphi(t,v)\Big).
\end{multline*}
The next lemma justifies that in the case where $\uptau$ and $\eta$
are finite, the  term $ \mathcal{J}^\varepsilon(\varphi)$
has a limit as $\varepsilon \rightarrow 0$. Defining 
\begin{equation}
  \label{defDeps}
   \mathscr{D}^\varepsilon(t,v):= \frac{1}{\uptau}
   \int_{0}^{\uptau}ds\int_{-\eta}^{s}d\sigma  \fint dx\,
   E^\varepsilon(t+s\varepsilon^2,x+v(\sigma+\eta))\otimes
    E^\varepsilon(t+ (s-\eta-\sigma)\varepsilon^2,x),
\end{equation}
 we have 
\begin{lemma}
  \label{lem:QAL} Let  $\uptau$ and $\eta$ be finite. Then, $\mathcal{J}^\varepsilon$ has a limit in
  $\mathcal{D}'(\R^+ \times \R^d)$ such that 
  \begin{equation}
    \label{WLfD}
    \lim_{\varepsilon \rightarrow 0}\,
    \mathcal{J}^\varepsilon(\varphi)  = \int_{\R^+}dt \int_{\R^d} dv \, f \nabla_v
    \cdot (\mathscr{D}^T\nabla_v \varphi), \quad \forall \varphi
    \in \mathcal{D}(\R^+ \times \R^d),
  \end{equation}
where $\mathscr{D}$ is the weak limit of $\mathscr{D}^\varepsilon$ (up
to a subsequence) in the following sense,
\begin{equation}
  \label{W11WLD}
  \mathscr{D}^\varepsilon \rightharpoonup \mathscr{D}  \ \ \mbox{ in }
  \ \ L_{\rm loc}^1(\R^+;W_{\rm loc}^{1,1}(\R^d)) \ \  \mbox{weak},
\end{equation}
and
\begin{equation}
  \label{LinftyWLD}
  \mathscr{D}^\varepsilon \rightharpoonup \mathscr{D}  \ \ \mbox{ in }
  \ \ L^\infty(\R^+;L^{\infty}(\R^d)) \ \ \mbox{weak}\!-\!\ast.
\end{equation}

\end{lemma}
\begin{proof}
  Since $f \in L^\infty(\R^+ \times\R^d)$, to prove
  \eqref{WLfD}-\eqref{W11WLD} it is sufficient to show that the term
  \begin{equation*}
    \label{lita:rk:1}
    g^\varepsilon(t,v):=
    \nabla_v\cdot\big(\mathscr{D}^\varepsilon(t,v)^T\nabla_v
    \varphi(t,v)\big)\\
  \end{equation*}  
  converges weakly in $L^1(\R^+\times \R^d)$ as
  $\varepsilon\rightarrow 0$, i.e. for the weak topology
  $\sigma(L^1,L^\infty)$.  For this, we appeal to the Dunford--Pettis
  theorem.  First we show that $g^\varepsilon \in L^1(\R^+\times
  \R^d)$ uniformly in $\varepsilon$. For this, we use the fact that
  weak solutions of  \eqref{eqn:V2}-\eqref{eqn:GL2} are such that
  $E^\varepsilon \in L^\infty(\R^+;W^{1,1+2/d}(\mathbb{T}^d))$
  uniformly with respect to $\varepsilon$.  For $d=2$, we use the $L^2$
  duality and  $E^\varepsilon \in L^\infty(\R^+;L^2(\mathbb{T}^d))$
  uniformly in $\varepsilon$ (see Remark~\ref{rk:EE}). For $d=3$, we
  use the $L^p$-$L^{q}$ duality with $(p,q)=(1+2/d,1+d/2)$  and   the
  Sobolev embedding $W^{1,1+2/d}(\mathbb{T}^d)\hookrightarrow
  L^{1+d/2}(\mathbb{T}^d)$. Since the case $d=1$ is simpler (using the
  regularity $E^\varepsilon \in L^\infty(\R^+ \times \mathbb{T}^d)$
  uniformly in $\varepsilon$), we only give the proof for $d\geq
  2$. Then, for $d\geq 2$, using H\"older's inequality, we obtain
   \begin{align*}
     \|g^\varepsilon\|_{L^1(\R^+\times \R^d)}  \leq \ &  \bigg \|
     \frac{1}{\uptau}\sum_{i,j}
     \int_{0}^{\uptau}ds\int_{-\eta}^{s}d\sigma \fint dx\, (\sigma+\eta)
     \\ &  E_i^\varepsilon(t+ (s-\eta-\sigma)\varepsilon^2,x) \,(\partial_{x_j}
     E^\varepsilon)(t+s\varepsilon^2,x+v(\sigma+\eta)) \partial_{v_j}
     \varphi(t,v) \bigg\|_{L^1(\R^+\times \R^d)} \\ &
     +\bigg\|\sum_{i,j} \frac{1}{\uptau}
     \int_{0}^{\uptau}ds\int_{-\eta}^{s}d\sigma \fint dx\,(\sigma+\eta)
     \\ & E_j^\varepsilon(t+
     (s-\eta-\sigma)\varepsilon^2,x)E_j^\varepsilon(t+s\varepsilon^2,x+v(\sigma+\eta))
     \partial_{v_i v_j}^2 \varphi(t,v)
     \bigg\|_{L^1(\R^+\times \R^d)} \\ \leq \ &  (\uptau^2 + 3\uptau
     \eta
     +3\eta^2)(2\pi)^{-d}\|E^\varepsilon\|_{L^\infty(\R^+;L^{1+d/2}(\mathbb{T}^d))}
     \\ &\left\{ \|\nabla_x
     E^\varepsilon\|_{L^\infty(\R^+;L^{1+2/d}(\mathbb{T}^d))} \|
     \nabla_v \varphi\|_{L^1(\R^+\times \R^d)}  +
     \|E^\varepsilon\|_{L^\infty(\R^+; L^{1+2/d}(\mathbb{T}^d))} \|
     \nabla_v^2 \varphi\|_{L^1(\R^+\times \R^d)}  \right\} \\ < \ &\infty.
   \end{align*}
For showing uniform equi-integrability of the family $g^\varepsilon$, we take
$A \subset \R^+\times \R^d$ such that $|A| \leq \delta$, with $\delta$
small. Following the above computations we obtain
\begin{eqnarray*}
  \|g^\varepsilon\|_{L^1(A)}  &\leq& (\uptau^2 + 3\uptau \eta
  +3\eta^2)(2\pi)^{-d}\|E^\varepsilon\|_{L^\infty(\R^+;L^{1+d/2}(\mathbb{T}^d))}
  \\ && \left\{ \|\nabla_x
  E^\varepsilon\|_{L^\infty(\R^+;L^{1+2/d}(\mathbb{T}^d))} \| \nabla_v
  \varphi\|_{L^1(A)}  +
  \|E^\varepsilon\|_{L^\infty(\R^+;L^{1+2/d}(\mathbb{T}^d))} \|
  \nabla_v^2 \varphi\|_{L^1(A)}  \right\}
  \\ &\leq&  (\uptau^2 + 3\uptau \eta
  +3\eta^2)|A|(2\pi)^{-d}\|E^\varepsilon\|_{L^\infty(\R^+;L^{1+d/2}(\mathbb{T}^d))}
  \\ && \left\{ \|\nabla_x
  E^\varepsilon\|_{L^\infty(\R^+;L^{1+2/d}(\mathbb{T}^d))} \| \nabla_v
  \varphi\|_{L^\infty(\R^+\times \R^d)}  +
  \|E^\varepsilon\|_{L^\infty(\R^+;L^{1+2/d}(\mathbb{T}^d))} \|
  \nabla_v^2 \varphi\|_{L^\infty(\R^+\times \R^d)}  \right\}
  \\ &\lesssim& \delta,
\end{eqnarray*}  
which shows uniform equi-integrability of the family $g^\varepsilon$. It
remains to prove \eqref{LinftyWLD}.  Using Fourier series in the space
variables, the matrix \eqref{defDeps} rewrites as 
\begin{equation}
  \label{defDepsFourier}
\mathscr{D}^\varepsilon(t,v) = \frac{1}{\uptau}\int_0^\uptau ds
\int_{-\eta}^{s} d\sigma \sum_{k\in \Z^d}  e^{\,{\rm i}
  k\cdot v (\sigma+\eta)}
\widehat{E}^\varepsilon(t+s\varepsilon^2,k)\otimes
\widehat{E}^\varepsilon(t+
(s-\eta-\sigma)\varepsilon^2,-k).  
\end{equation}
From \eqref{eqnBEE2}, there exists $\lambda >0$ such that
\begin{equation}
  \label{eqnBEE3}
  |\widehat{E}^\varepsilon(t,k)|\leq c_0 (1+|k|)^{-(1+\lambda)}.
\end{equation}
Using \eqref{defDepsFourier}-\eqref{eqnBEE3} and weak compactness, we
obtain \eqref{LinftyWLD}, which ends the proof of Lemma~\ref{lem:QAL}.
\end{proof}

If we now assume
\begin{equation}
  \label{limM0AL}
   \lim_{\varepsilon \rightarrow 0}\, \mathcal{M}^\varepsilon(\varphi) =0,
\end{equation}
then, using Lemma~\ref{lem:QAL} and \eqref{T1AL}, we obtain from
\eqref{lita:eqn:3} and \eqref{T2-DecQM} the diffusion equation
\eqref{def:DifEq} with 
\begin{eqnarray}
  \mathscr{D}(t,v)&:=&    \lim_{\varepsilon \rightarrow
    0}\frac{1}{\uptau} \int_{0}^{\uptau}ds\int_{-\eta}^{s}d\sigma
  \fint dx\,
  E^\varepsilon(t+s\varepsilon^2,x)
  \otimes E^\varepsilon(t+
  (s-\eta-\sigma)\varepsilon^2,x-v(\sigma+\eta)) 
 \label{def:CDif:lita}\\
 &=& \lim_{\varepsilon \rightarrow 0} \frac{1}{\uptau}\int_0^\uptau ds
 \int_{-\eta}^{s} d\sigma \sum_{k\in \Z^d}  e^{\,{\rm i}
   k\cdot v (\sigma+\eta)}
 \widehat{E}^\varepsilon(t+s\varepsilon^2,k)
 \otimes\widehat{E}^\varepsilon(t+
 (s-\eta-\sigma)\varepsilon^2,k)^\ast. \nonumber
\end{eqnarray}

Few remarks are now in order.

\begin{remark}{(open issues)}
  \label{RKOI2}
  \begin{itemize}
  \item[1.] In order to justify  \eqref{limM0AL}, we have to prove
    that the term
    \[
    \nabla_v \cdot \Big( E^\varepsilon(t+
    (s-\eta-\sigma)\varepsilon^2,x)\otimes
    E^\varepsilon(t+s\varepsilon^2,x+v(\sigma+\eta)) \nabla_v
    \varphi(t,v)\Big),
    \]
    converges strongly in $L^1(\R_t^+ \times \R_s^+ \times \R_\sigma^+
    \times Q)$ as $\varepsilon \rightarrow 0$. 
  \item[2.] Show that $\mathcal{T}_1^\varepsilon
    \rightharpoonup 0$ in $\mathcal{D}'(\R^+\times \R^d)$ remains
    an open issue. Nevertheless, we may expect that there exist some mixing-type hypotheses, which
    could justify such limit.
 \item[3.]  Taking $\hat{\theta}=\varepsilon^2 \uptau$ in equation
   \eqref{lita:eqn:T1}, the term $ \mathcal{T}_{1}^\varepsilon$ is
   reminiscent of the term \eqref{dduhamel3} appearing in the case of
   the non-self-consistent stochastic electric field (see
   Section~\ref{s:NSCSC}).  Let us note that in the stochastic case,
   without additional regularity assumptions on weak solutions, this
   term vanishes by using a time decorrelation property between
   $E^\varepsilon$ and $f^\varepsilon$ and by using the fact that 
   the stochastic average of $E^\varepsilon$ vanishes.
 \item[4.] The parameter $\uptau$ is reminiscent of the  autocorrelation time of particles $\uptau$, which is
   introduced in case of  the non-self-consistent stochastic
   electric field (see Section~\ref{s:NSCSC}).
  \end{itemize}
\end{remark}

{\bf \noindent An explicit form of the diffusion matrix in the
  non-self-consistent deterministic case}.  \\
\newline
\noindent Without being able to prove some time decorrelation
properties for the electric field $E^\varepsilon$, it is difficult to
deduce the structure of the diffusion matrix $\mathscr{D}$.  However,
as in Section~\ref{ss:GIA}, we can design a non-self-consistent
deterministic electric field by using the WKB expansion \eqref{WKB_exp} to
obtain an explicit  diffusion matrix.  This gives a formal example for
which the diffusion matrix \eqref{def:CDif:lita} is not zero and
non-negative.  Using the  WKB approximation \eqref{WKB_app}, the notation 
  \[
  \Delta \Omega:= \partial_t\Omega(t,k)-  k\cdot v,
  \]
  and the parity of functions $\Omega(t,k)$ and $\widehat{E}_0(t,k)$
  in the variable $k$, we obtain from similar computations leading to
  \eqref{twqld}, 
  \begin{align}
    \mathscr{D}(t,v) = &  \lim_{\varepsilon \rightarrow 0}
    \frac{1}{\uptau}\int_0^\uptau ds \int_{-\eta}^{s} d\sigma
    \sum_{k\in \Z^d}  e^{ \,{\rm i} k\cdot v
      (\sigma+\eta)}
     \widehat{E}^\varepsilon(t+s\varepsilon^2,k)
    \otimes\widehat{E}^\varepsilon(t+
    (s-\eta-\sigma)\varepsilon^2,k)^\ast \nonumber\\ =&
    \sum_{k\in \Z^d} \frac{1}{\uptau}\int_0^\uptau
    ds\int_{0}^{s+\eta}d\sigma  e^{-{\rm i}\sigma  \Delta \Omega(t,k)}
    \widehat{E}_0(t,k) \otimes  \widehat{E}_0(t,k)   \nonumber\\ =&
    \sum_{k\in \Z^d}  \left(\frac{1}{{\rm i} \Delta \Omega} -
    \frac{e^{-{\rm i}\eta  \Delta \Omega}}{{\rm i} \Delta \Omega} \,
    \frac{1-e^{-{\rm i}\uptau  \Delta \Omega}}{{\rm i}\uptau  \Delta
      \Omega} \right) \widehat{E}_0(t,k) \otimes  \widehat{E}_0(t,k)
    \nonumber\\ =&\sum_{k\in \Z^d} \frac{\sin( \uptau\Delta
      \Omega/2)}{\uptau\Delta \Omega/2} \frac{\sin(
      (\uptau/2+\eta)\Delta \Omega)}{\Delta \Omega} \widehat{E}_0(t,k)
    \otimes  \widehat{E}_0(t,k).\label{twqld-2}
  \end{align}
  At this point we obtain some important limits with respect to the
  parameters $\eta$ and $\uptau$.  The parameter $\uptau$ is the same
  as the one that we have  defined in Section~\ref{s:hwtr} and used in
  Section~\ref{s:NSCSC}. Hence parameters $\uptau$ and $\eta$ can be
  seen as normalized particle autocorrelation times.   

  The first significant limit is $\eta \rightarrow + \infty$.  Indeed
  using the limit $  \lim_{\eta\rightarrow +\infty } {\sin(\eta
    \tau)}/{\tau} = \pi \delta(\tau) $ in $\mathcal{D}'(\R)$, and
  taking the limit $\eta \rightarrow + \infty$ in \eqref{twqld-2}, we
  recover the same diffusion matrix \eqref{SCDQL} that we have
  obtained with the global-in-time approach of
  Section~\ref{ss:GIA}. Therefore, we  recover the quasilinear
  diffusion matrix \eqref{DQLTXTB} too. In a sense, the limit $\eta
  \rightarrow + \infty$ corresponds to take into account all the past
  of the distribution function
  and particularly the initial
  condition as it was done in the global-in-time approach of
  Section~\ref{ss:GIA}. Therefore it is consistent to obtain the
  result of the global-in-time approach by taking the limit $\eta
  \rightarrow + \infty$ in the local-in-time approach.

  The second  significant limit is $\eta \rightarrow 0$.  Indeed for
  $\eta=0$, we obtain from \eqref{twqld-2}, the following non-negative
  diffusion matrix,
  \begin{equation}
    \label{DifMGen}
  \mathscr{D}(t,v) = \frac{\uptau}{2}\sum_{k\in \Z^d}
  \left(\frac{\sin\big ( \frac{\uptau}{2}
    \big(\partial_t\Omega(t,k)-k\cdot v\big) \big)}
       {\frac{\uptau}{2} \big(\partial_t\Omega(t,k)-
         k\cdot v\big)}\right)^2 \widehat{E}_0(t,k) \otimes
       \widehat{E}_0(t,k).
  \end{equation}
  This diffusion matrix seems more regular in velocity than the
  quasilinear diffusion matrix \eqref{DQLTXTB}.  This regularity improvement
  in velocity is reminiscent of the finite-$\uptau$ effect that we observe 
  in the framework of the  non-self-consistent stochastic electric field  (see
  Section~\ref{s:NSCSC}) and also in the resonance
  broadening like theory (see Section~\ref{ss:RBT}). 

  The last significant limit is to  keep $\eta$ finite and to pass to
  the limit $\uptau \rightarrow +\infty$ in \eqref{twqld-2} or
  \eqref{DifMGen}. In this limit, we also recover the quasilinear
  diffusion matrix  \eqref{SCDQL} or \eqref{DQLTXTB}, and Remark~\ref{rem:DEGPHI}
  still holds true.  This result is
  consistent with the developements of Section~\ref{ss:QLT} for the
  non-self-consistent stochastic electric field.

  \begin{remark} As in Remark~\ref{rem:DEGPHI}, if  the electric field $E^\varepsilon$ derives from
    a potential $\Phi^\varepsilon$, then in the diffusion matrices \eqref{twqld-2} and \eqref{DifMGen}
    we should replace the matrix $\widehat{E}_0(t,k) \otimes \widehat{E}_0(t,k)$ by the matrix
    $|\widehat{E}_0(t,k)|^2\, k \otimes k / |k|^2$. 
  \end{remark}

\section{The non-self-consistent stochastic case}
\label{s:NSCSC}

In this section we deal with the non-self-consistent stochastic
case. Before stating our main Theorem~\ref{thm:dloveiwtr} in
Section~\ref{s:mthm}, we start by describing the features of the
stochastic electric field in the following section.

\subsection{The turbulent electric field}
\label{ss:tef}
Here, electrostatic turbulence is modeled through the random vector
field  $E^\varepsilon$.  Let $(\Upomega, \mathcal{F}, \mathbb{P})$ be a
probability space, with $\mathbb{P}$ being a $\sigma$-finite measure.  A
random vector $F$ is real vector-valued function defined on
$\Upomega$. When $F:\Upomega \rightarrow \R^d$ is an integrable random
vector, its expectation is given by
\begin{equation*}
\label{def:expectation}
\mathbb{E}[F]= \int_{\Upomega} {\rm d}\mathbb{P}(\upomega)\, F(\upomega).
\end{equation*}  
From considerations of Section~\ref{ss:dp}, the turbulent electric
field $E^\varepsilon$ has two time scales, one slow and the other fast. 
We then choose a turbulent electric
field $E^\varepsilon$ given by
\begin{equation}
  E^\varepsilon(t,x)=E(t,t/\varepsilon^2,x; \upomega),
  \label{def:Eepsilon}
\end{equation}
where, the integrable random vector field $E$ satisfies the following
``stochastic'' assumptions:
\begin{itemize}
\item[$({\rm H}1)$:] The random vector field $E$ is centered, i.e.
  \begin{equation*}
    \label{H1}
      \mathbb{E}[E(t,\tau,x)] = 0, \quad \forall (t,\tau,x)\in \R^+\times
      \R^+ \times \mathbb{T}^d.
  \end{equation*}
\item[$({\rm H}2)$:] There exists a constant $\uptau>0$ such that
  for every $x, y \in\R^d$ and for every $\tau, \sigma \in\R^+$
  the electric fields  $E(t,\tau,x)$ and
  $E(s,\sigma,y)$  are independent random vector fields as soon as
  $|\tau-\sigma|\geq \uptau$. The autocorrelation time $\uptau$ is supposed
  fixed and finite, hence independent of $\varepsilon$.
\item[$({\rm H}3)$:] There exists a matrix-valued function
  $\mathcal{R}_\uptau:\R^+\times\R^+\times\R\times \mathbb{T}^d\rightarrow
  \R^{2d}$, called the autocorrelation matrix or the Reynolds electric stress tensor,  such that
  \begin{equation}
    \label{H3}
    \mathbb{E}[E(t,\tau,x)\otimes E(s,\sigma,y)]
    =\mathcal{R}_\uptau(t,s,\tau-\sigma,x-y).
  \end{equation}
\end{itemize}

Hypothesis $({\rm H}1)$ sets the stochastic average of $E^\varepsilon$
to zero, which is standard and not restrictive. Assumption $({\rm
  H}2)$ means that the turbulent electric field $E^\varepsilon$ is
time decorrelated on a time scale $\varepsilon^2\uptau$. Assumption
$({\rm H}2)$ can be seen as a hypothesis of propagation of
``stochasticity'' or propagation of independence of random vector
fields.  Therefore, two evaluations in time of the electric field,
separated by  a lapse of time larger than $\varepsilon^2\uptau$, are
independent random vector fields.  Hypothesis $({\rm H}3)$ is the standard
spatio-temporal homogeneity property of the turbulence,  i.e. the
spatio-temporal autocorrelation of the electric field  $E^\varepsilon$
is invariant under space and time translations. These assumptions are
similar to the ones of \cite{PV03}.  

\begin{remark}
  In the nonlinear regime, the property of time decorrelation seems to
  be the cornerstone of the diffusion process, for both the
  self-consistent and the non-self-consistent setting  \cite{Dav72,
    AAPSS75, Bal05}. An open and very difficult problem remains to
  show mathematically such time decorrelation property from only the
  deterministic Vlasov-Poisson system \eqref{eqn:V}-\eqref{eqn:GL} and
  random initial data $f_0$. In a sense, this is what has been shown
  numerically in \cite{BEEB11, BEEB11-2}. This property of propagation
  of ``stochasticity'' is reminiscent of the property of propagation
  of chaos in statistical mechanics. 
\end{remark}

In order to justify rigorously the diffusion limit, the  stochastic
electric field requires the following regularity assumptions.
\begin{itemize}
\item[$({\rm H}4)$:] The regularity of $E$ is such that
  \begin{equation*}
    \label{H4}
    E \in L^\infty\big(\R^+\times \R^+;
    W^{2,\infty}(\mathbb{T}^d)\big), \quad \mbox{ and } \quad
    \mathbb{E}\left[ \| E\|_{ L^\infty(\R^+\times \R^+;
        W^{2,\infty}(\mathbb{T}^d))}^3\right] =:C_E< \infty.
  \end{equation*}
\end{itemize}  
Assumption $({\rm H}4)$ imposes the regularity (especially in space)
of the random vector field $E$.
It is worthwhile to end this section by giving an explicit example of a random field
$E$, which satisfies assumptions $({\rm H}1)$-$({\rm H}4)$. Following
the spirit of Example~2 in \cite{PV03}, we construct in Appendix~\ref{ACRE}
a random vector field $E$ satisfying these requirements.
  
\subsection{Main theorem}
\label{s:mthm}

Concerning the non-self-consistent stochastic case, we establish

\begin{theorem}
\label{thm:dloveiwtr}
Let $E$ be an integrable random vector field satisfying assumptions
$({\rm H}1)$-$({\rm H}4)$, and let $E^\varepsilon$ be given by
\eqref{def:Eepsilon}.  Let $\{f_0^\varepsilon\}_{\varepsilon>0}$ be a
sequence of independent random non-negative initial data  and $C_0$ be a
positive constant such that for a.e. $\upomega \in \Upomega$,
$\|f_0^\varepsilon \|_{L^1(Q)}+\|f_0^\varepsilon \|_{L^\infty(Q)} \leq C_0<\infty$.
Let $\mathscr{D}_\uptau=\mathscr{D}_\uptau(t,v)$ be the matrix-valued function defined
by
\begin{equation}
\label{def:CoefDiff}
\mathscr{D}_\uptau(t,v)= \int_0^\uptau d\sigma\,\mathcal{R}_\uptau(t,t,\sigma,\sigma v), 
\end{equation}
the properties of which, are described in
Proposition~\ref{prop:Dprop}.  Let $f^\varepsilon$ be the unique weak
solution of  the Vlasov equation \eqref{eqn:V2}, with  initial data
${f^\varepsilon}_{|_{t=0}}=f_0^\varepsilon$.\\  Then up to extraction
of a subsequence, $\mathbb{E}[f_0^\varepsilon]$ converges in
$L^\infty(Q)$ weak--$\ast$ to a function $f_0\in L^1\cap L^\infty(Q)$,
$\mathbb{E}[f^\varepsilon]$  converges in $L^\infty(\R^+;L^\infty(Q))$
weak--$\ast$ to a function $f\in L^\infty(\R^+; L^1\cap
L^\infty(\R^d))$, and $\mathbb{E}[\fint dx\, f^\varepsilon]$
converges in $L^\infty(\R^+;L^\infty(\R^d))$ weak--$\ast$ to
$f$. Moreover $\mathbb{E}[\fint dx\, f^\varepsilon]$  converges in
$\mathscr{C}(0,T;L^p(\R^d)-weak)$ to $f$, for $1<p<\infty$ and for all
$T>0$. The limit point $f=f(t,v)$ is solution of the following
diffusion equation in the sense of distributions:
\begin{eqnarray}
&&\partial_t f -\nabla_v \cdot(\mathscr{D}_\uptau\nabla_v f) =0, \quad \mbox{
    in } \  \mathcal{D}'(\R^+\times \R^d), \label{def:diffeq} \\ &&
  {f}_{|_{t=0}} = \fint dx\, f_0. \label{def:condinit} \nonumber
\end{eqnarray}  
\end{theorem}

The Proof of Theorem~\ref{thm:dloveiwtr} is postponed to
Section~\ref{s:proof}. General properties of the diffusion matrix $\mathscr{D}_\uptau$ and the
autocorrelation matrix $\mathcal{R}_\uptau$ of Theorem~\ref{thm:dloveiwtr}
are stated in 

\begin{proposition}{(properties of  the diffusion matrix $\mathscr{D}_\uptau$)}
  \label{prop:Dprop}
  Under  assumptions $({\rm H}1)$-$({\rm H}4)$, the matrix-valued
  function $\mathcal{R}_\uptau:\R^+\times\R^+\times\R\times
  \mathbb{T}^d\rightarrow \R^{2d}$, and the diffusion matrix
  $\mathscr{D}_\uptau:\R^+\times\R^d\rightarrow \R^{2d}$ satisfy the
  following properties:

\begin{itemize}
\item[$i)$] $\mathcal{R}_\uptau(t,t,\tau,x)=\mathcal{R}_\uptau^T(t,t,-\tau,-x)$, and
  $\ \mathcal{R}_\uptau(t,t,\tau,x+2\pi k)=\mathcal{R}_\uptau(t,t,\tau,x), \quad \forall
  k\in\Z$.
\item[$ii)$]  $\mathcal{R}_\uptau\in L^\infty(\R^+\times\R^+\times\R;
  W^{2,\infty}(\mathbb{T}^d))$, and $\ {\rm supp}(\mathcal{R}_\uptau) \subset
  \R^+\times\R^+ \times [-\uptau, \uptau]\times\mathbb{T}^d$.
\item[$iii)$] $\mathscr{D}_\uptau\in
  L^\infty(\R^+; W^{2,\infty}(\R^d))$, and  $\ {\rm supp}(\mathscr{D}_\uptau)
  \subset \R^+\times\R^d$.
\item [$iv)$] The  symmetric part of $\mathscr{D}_\uptau$ is non-negative,
  i.e.  $ X^T\mathscr{D}_\uptau X \geq 0, \  \forall \in X \in \R^d$.
\end{itemize}
\end{proposition}

\begin{proof}
We start with property $i)$. Using $({\rm H}3)$, we obtain
\[
\mathcal{R}_{\uptau\,ij}(t,t,\tau-\sigma,x-y) =
\mathbb{E}[E_i(t,\tau,x)E_j(t,\sigma,y)]=
\mathbb{E}[E_j(t,\sigma,y)E_i(t,\tau,x)]=\mathcal{R}_{\uptau\,ji}(t,t,-(\tau-\sigma),-(x-y)), 
\]
and for all $k,k' \in \Z,$
\begin{eqnarray*}
\mathcal{R}_\uptau(t,t,\tau-\sigma,x-y) &=& \mathbb{E}[E(t,\tau,x)\otimes
  E(t,\sigma,y)]= \mathbb{E}[E(t,\tau,x+2\pi k)\otimes
  E(t,\sigma,y+2\pi k')]\\ &=&\mathcal{R}_{\uptau\,ij}(t,t,\tau-\sigma,x-y +2\pi(k-k')).
\end{eqnarray*}
The regularity  property $ii)$ for $\mathcal{R}_\uptau$ comes immediately from
the regularity assumption $({\rm H}4)$ for $E$ and the definition $({\rm H}3)$
for $\mathcal{R}_\uptau$. The regularity
property  $iii)$ for $\mathscr{D}_\uptau$ is the straight consequence of the
definition \eqref{def:CoefDiff} for  $\mathscr{D}_\uptau$ and  the regularity
property $ii)$ for $\mathcal{R}_\uptau$. The
support of $\mathscr{D}_\uptau$ is obvious, while the support of
$\mathcal{R}_\uptau$ results from assumptions $({\rm H}1)$-$({\rm
  H}3)$. Indeed, if $|\tau-\sigma|> \uptau$, then  $({\rm H}1)$-$({\rm
  H}3)$ imply that
$\mathcal{R}_\uptau(t,t,\tau-\sigma,x-y)=\mathbb{E}[E(t,\tau,x)\otimes
  E(t,\sigma,y)]=\mathbb{E}[E(t,\tau,x)]\otimes
  \mathbb{E}[E(t,\sigma,y)]=0$.  We end with the property $iv)$. Using properties
$i)$ and $ii)$, we obtain
\begin{eqnarray}
\label{eqn:PD}  
X^T\mathscr{D}_\uptau X &=&\sum_{i,j} X_iX_j\int_0^\uptau d\sigma\,
\mathcal{R}_{\uptau\,ij}(t,t,\sigma,\sigma v)= \frac{1}{2}\sum_{i,j}
X_iX_j\int_{-\uptau}^\uptau  d\sigma\, \mathcal{R}_{\uptau\,ij}(t,t,\sigma,\sigma
v)\nonumber\\ &=&\frac{1}{2}\sum_{i,j}
X_iX_j\int_{-\infty}^{+\infty}d\sigma\,\mathcal{R}_{\uptau\,ij}(t,t,\sigma,\sigma
v).
\end{eqnarray}
Using Lemma~3.1 of \cite{PV03}, which states that for all $g\in
L^1(\R)$ we have
\[
\int_{-\infty}^{+\infty} g(s)ds = \lim_{R\rightarrow
  +\infty}\frac{1}{2R}\int_{-R}^Rds \int_{-R}^R dt \, g(s-t),
\]
we obtain from \eqref{eqn:PD},
\begin{eqnarray*}
  X^T\mathscr{D}_\uptau X &=& \lim_{R\rightarrow
    +\infty}\frac{1}{4R}\sum_{i,j} X_iX_j\int_{-R}^{R}
  d\sigma\int_{-R}^{R} d\theta\, \mathbb{E}[E_{i}(t,\sigma,\sigma v)
    E_{j}(t,\theta,\theta v)]\\ &=& \lim_{R\rightarrow
    +\infty}\frac{1}{4R}\int_{-R}^{R} d\sigma\int_{-R}^{R} d\theta\,
  \mathbb{E}[X\cdot E(t,\sigma,\sigma v) \,X\cdot E(t,\theta,\theta
    v)]\\ &=& \lim_{R\rightarrow +\infty}\frac{1}{4R}\mathbb{E}\left[
    \left(\int_{-R}^{R}d\sigma\, X\cdot E(t,\sigma,\sigma
    v)\right)^2\right]  \geq 0,
\end{eqnarray*}  
which ends the Proof of Proposition~\ref{prop:Dprop}.
\end{proof}

To end this section, it is worthwhile to compare
the diffusion matrix obtained here for the stochastic case and
the one obtained for the deterministic case in Section~\ref{ss:IDF}.
If $\varepsilon$ is small enough such that $t/\varepsilon^2 > \uptau$, then
using $({\rm H}3)$, \eqref{def:Eepsilon}, and the compact support
of $\mathcal{R}_\uptau(\cdot,\cdot,\sigma,\cdot)$ in the variable $\sigma$
(included in $[-\uptau,\uptau]$; see Proposition~\ref{prop:Dprop}), 
the stochastic diffusion matrix $\mathscr{D}_\uptau$, defined by \eqref{def:CoefDiff},
rewrites as
\begin{eqnarray}
  \mathscr{D}_\uptau(t,v)&=&\lim_{\varepsilon \rightarrow 0}\int_0^\uptau
  d\sigma\, \mathbb{E}[ E^\varepsilon(t,x) \otimes E^\varepsilon(t-\varepsilon^2\sigma,x-\sigma
    v)]\nonumber\\
  &=&\lim_{\varepsilon \rightarrow 0}\int_{\R^+}
  d\sigma\, \chi_{[0,t/\varepsilon^2]}(\sigma)\mathbb{E}[ E^\varepsilon(t,x) \otimes E^\varepsilon(t-\varepsilon^2\sigma,x-\sigma
    v)]
  \label{Dsto}. 
\end{eqnarray}
Comparing the deterministic diffusion matrix \eqref{def:CDif}
(obtained for the global-in-time approach in Section~\ref{ss:GIA}) and
stochastic diffusion matrix \eqref{Dsto}, we observe that they are the
same except that the space average is
replaced by the statistical average. If we
suppose that in definition \eqref{def:CDif} of the deterministic
diffusion matrix, the electric field $E^\varepsilon$ is a random
vector field, and if we take the expectation value of
\eqref{def:CDif}, then the homogeneity property  $({\rm H}3)$ implies that
the space average is trivially the identity. Therefore
we recover  the stochastic diffusion matrix \eqref{Dsto} or
\eqref{def:CoefDiff} from the ``deterministic'' one \eqref{def:CDif} by a
statistical average and the homogeneity property.
The link between the deterministic and stochastic diffusion matrices
is also reinforced by

\begin{corollary}
  \label{cor:dloveiwtr}
  Let $(t,k)\mapsto \Omega(t,k):= \omega(k)t $ be a real-valued
  function, where the given real-valued function $k\mapsto \omega(k)$
  is odd in the variable $k$. Let  $(t,k) \mapsto \widehat{E}_0(t,k)$
  be a given real vector-valued function, which is even in the variable
  $k$ and such that $|k||\widehat{E}_0|\in L^\infty(\R^+;
  \ell^2(\Z^d))$. Then there exists a
  matrix-valued function $\mathcal{R}_\uptau^0:\R^+\times\R^+\times\R\times
  \mathbb{T}^d\rightarrow \R^{2d}$ such that the associated diffusion
  matrix, defined by formula \eqref{def:CoefDiff} of
  Theorem~\ref{thm:dloveiwtr}, is
  \begin{equation}
    \label{DifMGen_2}
  \mathscr{D}_\uptau(t,v) = \frac{\uptau}{2}\sum_{k\in \Z^d}
  \left(\frac{\sin\big ( \frac{\uptau}{2}
    \big(\partial_t\Omega(t,k)- k\cdot v\big) \big)}
       {\frac{\uptau}{2} \big(\partial_t\Omega(t,k)-
         k\cdot v\big)}\right)^2 \widehat{E}_0(t,k) \otimes
       \widehat{E}_0(t,k).
  \end{equation}
  Moreover the autocorrelation matrix $\mathcal{R}_\uptau^0$ and its
  associated diffusion matrix \eqref{DifMGen_2} satisfy properties
  $i)$--$iv)$ of Proposition~\ref{prop:Dprop}.
\end{corollary}  

Before giving the proof of Corollary~\ref{cor:dloveiwtr}, we observe
that the diffusion matrix \eqref{DifMGen_2} is the same as the diffusion matrix
\eqref{DifMGen}, obtained for the local-in-time approach of the
non-self-consistent deterministic case in Section~\ref{ss:LITA}.

\begin{namedproof}{{\it of Corollary \ref{cor:dloveiwtr}.}}
  We consider the Fourier series decomposition of a given smooth electric field $E$,
  \[
  E(t,\tau,x) =\sum_{k\in \Z^d} e^{\,{\rm i} k\cdot
    x}\widehat{E}(t,\tau,k),
  \]
  where, without loss of generality, we take
  \[
  \widehat{E}(t,\tau,k)=\widehat{E}_0(t,\tau,k)e^{-{\rm i}
    \Omega(\tau,k)}=\widehat{E}_0(t,\tau,k)e^{-{\rm i} \omega(k)\tau},
  \]
  with the function $\Omega(t,k)$ (respectively, $\omega(k)$) chosen like in
  the statement of Corollary~\ref{cor:dloveiwtr}.  Since
  $\widehat{E}(t,\tau,k)$ is Hermitian
  (i.e. $\widehat{E}^\ast(t,\tau,k)=\widehat{E}(t,\tau,-k)$), and because
  the function $k\mapsto\omega (k)$ is odd, we could choose either
  $\widehat{E}_0(t,\tau,k)$ Hermitian, or real and even in $k$.  We
  restrict ourselves to the case where $\widehat{E}_0(t,\tau,k)$ is real
  and even in $k$. In order to be consistent with hypothesis $({\rm H}3)$, 
  using the real vector-valued function $\widehat{E}_0(t,k)$ like in the
  statement of Corollary~\ref{cor:dloveiwtr}, we can choose  $\widehat{E}_0(t,\tau,k)$
  such that
  \begin{equation}
    \label{DEFC}
    \mathbb{E}\big[\widehat{E}_0(t,\tau,k)\otimes\widehat{E}_0(t,\sigma,k')\big]
    = \widehat{E}_0(t,k)\otimes \widehat{E}_0(t,k)
    A_\uptau(\tau-\sigma) \delta(k+k'),
  \end{equation}
  where  $s \mapsto A_\uptau(s):\R \rightarrow \R^+$ is  a real non-negative even bounded function.  
  From \eqref{DEFC} we obtain
  \begin{equation*}
    \mathbb{E}[E(t,\tau,x) \otimes E(t,\sigma,y)]=\sum_{k\in\Z^d}
    \widehat{E}_0(t,k)\otimes \widehat{E}_0(t,k)\,A_\uptau(\tau-\sigma)
    e^{\,{\rm i} k\cdot(x-y)}e^{-{\rm
        i}\omega(k)(\tau-\sigma)},
  \end{equation*}
  and then from  \eqref{H3} we obtain
  \begin{equation}
    \label{RRBT}
    \mathcal{R}_\uptau^0(t,t,\tau,x)= \sum_{k\in\Z^d}
    \,\widehat{E}_0(t,k)\otimes \widehat{E}_0(t,k)\, A_\uptau(\tau)
    e^{\,-{\rm i}(\omega(k)\tau - k\cdot x)}.
  \end{equation}
  The autocorrelation matrix $\mathcal{R}_\uptau^0$ given by  \eqref{RRBT}
  satisfies the property $i)$ of Proposition~\ref{prop:Dprop}.  Setting
  the resonance function
  \begin{equation*}
    \label{DefResF}
    R_\uptau(\xi) := \frac{\uptau}{2}
    \left(\frac{\sin(\uptau\xi/2)}{\uptau \xi/2}\right)^2,
  \end{equation*}
  we define the time autocorrelation function $A_\uptau$ as the
  inverse Fourier transform of the function $2R_\uptau$, i.e.
  \begin{equation}
    \label{AIFTR}
    A_\uptau(\sigma):=\frac{1}{\pi} \int_{-\infty}^{+\infty}e^{{\rm i}
      \xi \sigma}R_\uptau(\xi)d\xi=\Lambda(\sigma/\uptau),
  \end{equation}  
  where the function $s\mapsto \Lambda(s)$ is the triangular function
  of support $]-1,1[$.  As a consequence $\| A_\uptau\|_{L^\infty(\R)}
      \leq 1$, and using  $|k||\widehat{E}_0|\in L^\infty(\R^+;
      \ell^2(\Z^d))$, we obtain that  $\mathcal{R}_\uptau^0\in
      L^\infty(\R^+\times\R^+\times\R;
      W^{2,\infty}(\mathbb{T}^d))$. In addition from \eqref{AIFTR}, we
      deduce that  $\ {\rm supp}(\mathcal{R}_\uptau^0) \subset \R^+\times\R^+
      \times [-\uptau, \uptau]\times\mathbb{T}^d$.  Therefore the
      autocorrelation matrix $\mathcal{R}_\uptau^0$ given by \eqref{RRBT}-\eqref{AIFTR}
      satisfies the property $ii)$ of
      Proposition~\ref{prop:Dprop}. Finally, let us compute the
      diffusion matrix $\mathscr{D}_\uptau$ from \eqref{def:CoefDiff} and
      \eqref{RRBT}-\eqref{AIFTR}. Using parity properties of the
      functions $\omega$, $A_\uptau$ and $\widehat{E}_0$,  we obtain
  \begin{eqnarray}
    \mathscr{D}_\uptau(t,v)&=& \int_0^\uptau d\sigma\,
    \mathcal{R}_\uptau^0(t,t,\sigma,\sigma v) \nonumber\\ &=&
    \sum_{k\in\Z^d}\widehat{E}_0(t,k)\otimes \widehat{E}_0(t,k)\,
    \int_0^\uptau  d\sigma\, e^{\,{\rm i}(\omega(k)-
      k\cdot v)\sigma} A_\uptau(\sigma) \nonumber\\ &=&
    \sum_{k\in\Z^d} \widehat{E}_0(t,k)\otimes \widehat{E}_0(t,k)\,
    \int_{\R^+}  d\sigma\, e^{\,{\rm i}(\omega(k)-
      k\cdot v)\sigma}\Lambda(\sigma/\uptau)
    \nonumber\\ &=&\frac{1}{2}\sum_{k\in\Z^d}
    \widehat{E}_0(t,k)\otimes
    \widehat{E}_0(t,k)\,\int_{\R}  d\sigma\, e^{\,{\rm
        i}(\omega(k)-k\cdot
      v)\sigma}\Lambda(\sigma/\uptau)
    \nonumber\\ &=&\frac{\uptau}{2}\sum_{k\in \Z^d}
    \widehat{E}_0(t,k)\otimes \widehat{E}_0(t,k) \left(\frac{\sin\big
      ( \frac{\uptau}{2} \big(\omega(k)- k\cdot v\big)
      \big)} {\frac{\uptau}{2} \big(\omega(k)- k\cdot
      v\big)}\right)^2,
     \label{diffMforprof}
  \end{eqnarray}
  which is \eqref{DifMGen_2}. By a direct differentiation of
  \eqref{diffMforprof} with respect to $v$, we verify easily
  the property $iii)$ of Proposition~\ref{prop:Dprop}, while the
  property $iv)$ of Proposition~\ref{prop:Dprop} is obvious from the
  structure of \eqref{diffMforprof}.
\end{namedproof}

\subsection{Proof of Theorem~\ref{thm:dloveiwtr}}
\label{s:proof}
Let us rewrite the Vlasov equation \eqref{eqn:V2} in the following form,
\begin{eqnarray}
  && \partial_t f^\varepsilon + \frac{1}{\varepsilon^2}
  \mathcal{L}f^\varepsilon = \mathcal{N}_t^\varepsilon f^\varepsilon,
  \label{eqn:V3}\\
 && {f^\varepsilon }_{|_{t=0}}=f_0^\varepsilon, \label{eqn:V3CI}
\end{eqnarray}  
where the linear operators $\mathcal{L}$ and
$\mathcal{N}_t^\varepsilon$ are defined by
\begin{equation}
  \label{defLN}
  \mathcal{L} = v\cdot\nabla_x, \quad \quad
  \mathcal{N}_t^\varepsilon
  =-\frac{1}{\varepsilon}E^\varepsilon(t,x)\cdot \nabla_v
  =-\frac{1}{\varepsilon}E(t,t/\varepsilon^2,x)\cdot \nabla_v.
\end{equation}
Obviously the operators $\mathcal{L}$ and $\mathcal{N}_t^\varepsilon$
are skew-adjoint for the scalar product of $L^2(Q)$, while
the operator $\mathcal{L}$ and the deterministic group $S_t^\varepsilon$,
generated by $\varepsilon^{-2}\mathcal{L}$ (see
Section~\ref{ss:Not}), commute with $\mathbb{E}$.
Of course space and statistical averages commute.
From hypothesis $({\rm H}2)$, the random operators $\mathcal{N}_t^\varepsilon$ and
$\mathcal{N}_s^\varepsilon$ are independent as soon as $|t-s|>\varepsilon^2\uptau$.

The next useful proposition states that time decorrelation of the stochastic
electric field also entails time decorrelation between the
distribution function and the electric field.
\begin{proposition}{(time decorrelation property between $f^\varepsilon$ and $E^\varepsilon$)}
  \label{decoref}
  Assume $({\rm H}2)$. Suppose that the random initial data
  $f_0^\varepsilon$ and the electric field $E^\varepsilon$ are
  independent. Then   $\mathcal{N}_{s}^\varepsilon$ is independent of
  $f^\varepsilon(t)$ as soon as $s\geq t+\varepsilon^2 \uptau$.   
\end{proposition}

\begin{proof} 
  From the Duhamel formula
  \[
  f^\varepsilon(t)=  S_t^\varepsilon f_0^\varepsilon + \int_0^t d\sigma\,
  S_{t-\sigma}^\varepsilon \mathcal{N}_{\sigma}^\varepsilon
  f^\varepsilon(\sigma),
      \]
  where $t\mapsto S_t^\varepsilon$ is the deterministic group
  generated by $\varepsilon^{-2}\mathcal{L}$ (see
  Section~\ref{ss:Not}),      we observe that $f^\varepsilon(t)$
  depends only on $f_0^\varepsilon$ and
  $\mathcal{N}_{\sigma}^\varepsilon$ (or
  $E^\varepsilon(\sigma,\cdot)$)   for $\sigma \leq t$.       Since
  $f_0^\varepsilon$ is independent of $E^\varepsilon(t,\cdot)$,
  $\forall t\in\R$, and since the electric fields
  $E^\varepsilon(s,\cdot)$ and $E^\varepsilon(t,\cdot)$ are
  independent as soon as $s > t + \varepsilon^2 \uptau$ (assumption
  $({\rm H}2)$), we obtain from the Duhamel formula the desired
  result.
\end{proof}

We start our analysis by recalling basic statements that we collect in

\begin{proposition}
\label{prop:WPV}
Assume  $({\rm H}4)$ and consider a sequence
$\{f_0^\varepsilon\}_{\varepsilon>0}$ of initial data such that
\[
 f_0^\varepsilon \geq 0, \ \  \mbox{ and for a.e. }\upomega \in \Upomega,
 \   \|f_0^\varepsilon \|_{L^1(Q)} +\|f_0^\varepsilon \|_{L^\infty(Q)} \leq C_0<\infty.
\]
Then, for any $\varepsilon>0$, the Cauchy's problem
\eqref{eqn:V3}-\eqref{eqn:V3CI} has a unique non-negative solution
$f^\varepsilon \in \mathscr{C}(\R^+,L^1\cap L^\infty(Q))$, which is
given by
\begin{equation}
  \label{LSVE}
f^\varepsilon(t,x,v) =
f_0^\varepsilon(X^\varepsilon(0;t,x,v),V^\varepsilon(0;t,x,v)),  
\end{equation}
where the characteristic curves $(X^\varepsilon,V^\varepsilon)$ are
solutions to the ODEs,
\begin{equation}
  \frac{dX^\varepsilon}{dt}(t) =
  \frac{1}{\varepsilon^2}V^\varepsilon(t), \quad
  \frac{dV^\varepsilon}{dt}(t) =
  \frac{1}{\varepsilon}E^\varepsilon(t,X(t)), \quad
  X^\varepsilon(0;0,x,v)=x, \ \ V^\varepsilon(0;0,x,v) =
  v. \label{eqcc}
\end{equation}
Moreover we have the a priori estimates,
\[
\| f^\varepsilon(t) \|_{L^p(Q)} =\| f_0^\varepsilon \|_{L^p(Q)}, \quad
1\leq p \leq \infty.  
\]
In addition, there exist a function $f_0\in L^1\cap L^\infty(\R^d)$, and
a function $f\in L^\infty(\R^+; L^1\cap L^\infty(\R^d))$, such as, up to
subsequences,
\begin{equation*}
\mathbb{E}[f_0^\varepsilon] \rightharpoonup f_0 \ \mbox{ in }
\ L^\infty(Q) \ \ \mbox{weak}\!-\!\ast, \quad \mbox{ and } \quad
\mathbb{E}[f^\varepsilon ] \rightharpoonup f \ \mbox{ in }
\ L^\infty(\R^+;L^\infty(Q)) \ \ \mbox{weak}\!-\!\ast.
\end{equation*}  
The limit point $f$ is such that $\fint dx\,f =f \in  L^\infty(\R^+;
L^1\cap L^\infty(\R^d))$.  The function
$\mathbb{E}[\fint dx\, f^\varepsilon]$ is the solution of 
\begin{eqnarray}
&&\partial_t \mathbb{E}\left[\fint dx\,f^\varepsilon\right]  +
  \nabla_v \cdot \mathbb{E}\left[\fint dx\, \frac{E^\varepsilon f^\varepsilon}{\varepsilon}\right]=0,
  \quad \mbox{ in }\ \mathcal{D}'(\R^+\times \R^d), \label{eqn:EAV_2}\\ &&
        {\mathbb{E}\left[\fint dx\,f^\varepsilon\right]}_{|_{t=0}}=
        \mathbb{E}\left[\fint dx\,f_0^\varepsilon\right]. \label{eqn:EAV_2_0}
\end{eqnarray}  
\end{proposition}  
\begin{proof}
  Since $\|\mathbb{E}[f_0^\varepsilon] \|_{L^1(Q)} +\|\mathbb{E}[f_0^\varepsilon] \|_{L^\infty(Q)} \leq
  \mathbb{E}[\|f_0^\varepsilon \|_{L^1(Q)}] +\mathbb{E}[\|f_0^\varepsilon \|_{L^\infty(Q)}] \leq C_0 <\infty $,
  by weak compactness arguments there exists a function $f_0\in L^1\cap
  L^\infty(Q)$ such that $\mathbb{E}[f_0^\varepsilon]$ (up to a
  subsequence) converges in $L^\infty(Q)$ weak--$\ast$ to $f_0$.
  Using the regularity hypothesis $({\rm H}4)$ for the electric field
  $E$, the Cauchy--Lipschitz--Picard theorem for ODEs gives existence
  and uniqueness of a regular Lagrangian flow
  $(X^\varepsilon,V^\varepsilon)$, which is a solution of
  \eqref{eqcc}. It follows from standard results on first-order
  transport equations (see, e.g., \cite{BGP00}) that the Lagrangian
  solution to  \eqref{eqn:V3}-\eqref{eqn:V3CI} is given by
  \eqref{LSVE}.  From \eqref{LSVE}, we obtain
  $\|\mathbb{E}[f^\varepsilon]
  \|_{L^\infty(Q)}=\|\mathbb{E}[f_0^\varepsilon] \|_{L^\infty(Q)} \leq
  C_0 <\infty$.  Moreover using skew-adjointness of $\mathcal{L}$ and
  $\mathcal{N}_t^\varepsilon$, we can prove, following standard lines,
  that $\partial_t \| f^\varepsilon (t)\|_{L^p(Q)}=0$, for $1\leq p
  \leq\infty$. This leads to $\mathbb{E}[\|f^\varepsilon
    \|_{L^\infty(\R^+;L^p(Q))}] = \mathbb{E}[\|f_0^\varepsilon
    \|_{L^p(Q)} ]\leq C_0 <\infty$ and $\|\mathbb{E}[f^\varepsilon
  ]\|_{L^\infty(\R^+;L^p(Q))} \leq \mathbb{E}[\|f^\varepsilon
    \|_{L^\infty(\R^+;L^p(Q))}]\leq C_0 <\infty$.  By weak compactness arguments
  there exists a function $f\in L^\infty(\R^+; L^1\cap L^\infty(Q))$,
  such that  $\mathbb{E} [f^\varepsilon]$   (up to a subsequence)
  converges in $L^\infty(\R^+;L^\infty(Q))$ weak--$\ast$ to $f$.  Now
  we claim that  $\varepsilon^2 \mathcal{N}_t^\varepsilon
  f^\varepsilon\rightarrow 0$, in $\mathcal{D}'(\R^+\times Q)$ as
  $\varepsilon$ tends to zero.  Indeed, we have for all $\varphi \in
  \mathcal{D}(\R^+\times Q)$,
  \[
  |\langle \varepsilon^2 \mathcal{N}_t^\varepsilon f^\varepsilon,
  \varphi\rangle| = \varepsilon^2 |\langle f^\varepsilon ,
  \mathcal{N}_t^{\varepsilon \,\ast} \varphi\rangle|\leq \varepsilon
  \| f_0^\varepsilon \|_{L^\infty(Q)} \|
  \varphi\|_{L^\infty(\R^+;W^{1,1}(Q))} \|
  E\|_{L^\infty(\R^+\times\R^+\times \mathbb{T}^d)}  \longrightarrow 0,
  \]
  as $\varepsilon \rightarrow 0$. Then multiplying the Vlasov equation \eqref{eqn:V3} by
  $\varepsilon^2$, taking its expectation value, and letting
  $\varepsilon$  go to zero, we find
  \begin{equation}
    \mathcal{L} f =0, \quad \mbox{ in } \ \mathcal{D}'(\R^+\times Q),
    \label{Lfzero}
  \end{equation}
  where we have used the commutation property between $\mathbb{E}$ and
  $\mathcal{L}$.  From Lemma~\ref{lem:EFF} and \eqref{Lfzero}, we infer that $f$ is
  independent of $x$ and  $\fint dx\,f=f \in L^\infty(\R^+; L^1\cap L^\infty(\R^d)) $.
  Finally, the Vlasov equation \eqref{eqn:V3} is averaged in space and
  then rewritten in a weak form.   Taking the expectation value  of
  the result, we obtain \eqref{eqn:EAV_2}-\eqref{eqn:EAV_2_0}.
\end{proof}

The rest of the proof is devoted to pass to the limit in
equation \eqref{eqn:EAV_2}. For this, we can first
start with a simple iteration of the Duhamel formula as it was done in
Section~\ref{ss:LITA} for the deterministic case.
As explained in Section~\ref{ss:SIDF} this method fails.
To solve this problem we use the method of \cite{PV03},
which consists to apply a double iteration of the Duhamel formula for $f$.
This is described in Section~\ref{ss:DIDF}.

\subsubsection{Simple iteration of the Duhamel formula}
\label{ss:SIDF}
Following what we have done in Section~\ref{ss:LITA}, where
a simple iteration of the Duhamel formula is used, we obtain,
for all $\varphi \in \mathcal{D}(\R^+\times \R^d)$,
\begin{equation}
\label{lita:eqn:3:sto}
\int_{\R^+}dt \int_{\R^d}dv \, \varphi(t,v) \frac{\mathbb{E}\big[\fint dx\, f^\varepsilon(t+\theta)\big] -
  \mathbb{E}\big[\fint dx\, f^\varepsilon(t)\big]}{\theta}
=  \mathcal{T}_{1}^\varepsilon(\varphi)  +
\mathcal{T}_{2}^\varepsilon(\varphi), 
\end{equation}
where
\begin{equation}
\label{lita:eqn:T1:sto}
\mathcal{T}_{1}^\varepsilon(\varphi) :=  \int_{\R^+}
dt \int_{\R^d}dv \int_{t}^{t+\theta}ds \fint dx\, \frac{1}{\varepsilon \theta}
 \mathbb{E}\big[{E^\varepsilon(s)}\cdot \nabla_v
  \varphi(t,v) S_{s-t+\hat{\theta}}^\varepsilon
  f^\varepsilon(t-\hat{\theta})\big],
\end{equation}
and 
\begin{equation*}
\label{lita:eqn:T2:sto}
\mathcal{T}_{2}^\varepsilon(\varphi) := 
\mathcal{J}^\varepsilon(\varphi) + 
\mathcal{M}^\varepsilon(\varphi),
\end{equation*}
with
\begin{multline}
\label{lita:eqn:T21:sto}
\mathcal{J}^\varepsilon(\varphi):= \int_{\R^+} dt
\int_{\R^d}dv \\ \int_{t}^{t+\theta}ds\int_{t-\hat{\theta}}^{s}d\sigma
\frac{1}{\varepsilon^2 \theta} f(t,v) \nabla_v \cdot \Big(\fint dx\,
\mathbb{E}\big[ S_{s-\sigma}^\varepsilon
  E^\varepsilon(\sigma,x)\otimes E^\varepsilon(s,x) \nabla_v
  \varphi(t,v)\big]\Big), 
\end{multline}
and
\begin{multline*}
\label{lita:eqn:T22:sto}
\mathcal{M}^\varepsilon(\varphi):=  \int_{\R^+} dt
\int_{\R^d}dv  \int_{t}^{t+\theta}ds\int_{t-\hat{\theta}}^{s}d\sigma
\\ \frac{1}{\varepsilon^2 \theta} \fint dx\,
\mathbb{E}\Big[\big(S_{s-\sigma}^\varepsilon f^\varepsilon(\sigma)-
  f(t,v)\big)\nabla_v \cdot \Big( S_{s-\sigma}^\varepsilon
  E^\varepsilon(\sigma,x)\otimes E^\varepsilon(s,x) \nabla_v
  \varphi(t,v)\Big)\Big]. 
\end{multline*}
By choosing $\hat{\theta}\geq \varepsilon^2 \uptau$, we can use the time
decorrelation hypothesis $({\rm H}2)$, Proposition~\ref{decoref} and
assumption  $({\rm H}1)$ to show that for equation
\eqref{lita:eqn:T1:sto} we have $ \mathcal{T}_{1}^\varepsilon(
\varphi)=0$.  The term $\mathcal{J}^\varepsilon(
\varphi)$, defined by equation \eqref{lita:eqn:T21:sto}, can be
treated as it was done in Section~\ref{ss:IDF} and  it gives the
diffusion term. It remains to deal with the error term $
\mathcal{M}^\varepsilon(\varphi)$, which is of order zero with
respect to $\varepsilon$, i.e.  $ \mathcal{M}^\varepsilon(
\varphi) =\mathcal{O}(\varepsilon^0)$. For this, we observe
that we have to evaluate the expectation of a cubic product between
$f^\varepsilon(\sigma)-f(t)$, $E^\varepsilon(\sigma)$ and
$E^\varepsilon(s)$, i.e. schematically
$\mathbb{E}[(f^\varepsilon(\sigma)-f(t)) E^\varepsilon(\sigma)
  E^\varepsilon(s)]$.  Using hypotheses $({\rm H}2)$-$({\rm H}3)$ and
Proposition~\ref{decoref}, we would like to replace an expression of
the form $\mathbb{E}[(f^\varepsilon(\sigma)-f(t))
  E^\varepsilon(\sigma) E^\varepsilon(s)]$ by an expression of the
form $\mathbb{E}[f^\varepsilon-f]\mathbb{E}[E^\varepsilon
  E^\varepsilon]$, because from Proposition~\ref{prop:WPV} we know
that $\mathbb{E}[f^\varepsilon - f]\rightharpoonup 0\,$ in $\,
L^\infty(\R^+;L^\infty(Q))$  weak--$\ast$.  Unfortunately it is not
possible, since $f^\varepsilon(\sigma)$ and $E^\varepsilon(\sigma)$
are evaluated at the same time $\sigma$.  To remedy this problem we
follow the procedure of \cite{PV03}, which consists of using a
double iteration of the Duhamel formula for $f$. A double instead of a
simple iteration of the Duhamel formula is used  to go back in time
far enough in order to use the time decorrelation property $({\rm
  H}2)$.  The price of this procedure is the introduction of a second
error term, namely $\mu_t^\varepsilon$ (defined by \eqref{def:eta}).
For the error term  $\mu_t^\varepsilon$, we face the same problem as
for the term $\mathcal{M}^\varepsilon$, i.e.  we cannot use the time
decorrelation hypotheses $({\rm H}2)$-$({\rm H}3)$ and
Proposition~\ref{decoref}.  Nevertheless, the error term
$\mu_t^\varepsilon$ is of the order $\mathcal{O}(\varepsilon)$. Therefore, to
show that $\lim_{\varepsilon\rightarrow 0}{\mu_t^\varepsilon}=0$ in the distributional
sense, we do not use time decorrelation hypotheses (which are useless),
but we appeal to the regularity hypotheses on the electric field
$E^\varepsilon$, namely  $({\rm H}4)$.

\subsubsection{Double iteration of the Duhamel formula}
\label{ss:DIDF}

First we recall that $t\mapsto S_t^\varepsilon$ is the (deterministic) group on
$L^p(Q)$,  $1\leq p\leq \infty$, generated by
$\varepsilon^{-2}\mathcal{L}$ (see equation \eqref{def:Set}).
Using the group $S_t^\varepsilon$ and
the Duhamel formula, the formal solution to \eqref{eqn:V3} is given by
\begin{equation}
\label{duhamel1}
f^\varepsilon(t)=  S_{t-s}^\varepsilon f^\varepsilon(s) + \int_s^td\tau\,
S_{t-\tau}^\varepsilon\mathcal{N}_{\tau}^\varepsilon
f^\varepsilon(\tau).
\end{equation}
Taking $s=t-\varepsilon^2\uptau$ in \eqref{duhamel1}, and making the
change of variable $\tau=t-\sigma$, we obtain from \eqref{duhamel1},
\begin{equation}
\label{duhamel2}
f^\varepsilon(t)=  S_{\varepsilon^2\uptau}^\varepsilon
f^\varepsilon(t-\varepsilon^2\uptau) + \int_0^{\varepsilon^2\uptau}d\sigma\,
S_{\sigma}^\varepsilon\mathcal{N}_{t-\sigma}^\varepsilon
f^\varepsilon(t-\sigma).
\end{equation}
In the integral term of \eqref{duhamel2}, we observe that the electric
field and the distribution function are evaluated at the same time
$t-\sigma$. As a consequence,
if we substitute \eqref{duhamel2} to
$f^\varepsilon$ in the right hand side of \eqref{eqn:EAV_2}
(like it was done in Section~\ref{ss:SIDF}), we obtain
a quadratic term with respect to the electric field that we cannot
decorrelate in time from the distribution function.  For this reason
and following \cite{PV03},
we iterate a second time the Duhamel formula. In the same way that we
obtained \eqref{duhamel1}, we obtain
\begin{equation}
\label{duhamel3}
f^\varepsilon(t-\sigma)=  S_{2\varepsilon^2\uptau-\sigma}^\varepsilon
f^\varepsilon(t-2\varepsilon^2\uptau) +
\int_0^{2\varepsilon^2\uptau-\sigma}ds\,
S_{s}^\varepsilon\mathcal{N}_{t-\sigma-s}^\varepsilon
f^\varepsilon(t-\sigma-s).
\end{equation}
Substituting the right-hand side of \eqref{duhamel3} to
$f^\varepsilon(t-\sigma)$ in the right-hand side of \eqref{duhamel2},
and using the properties of the group $S_t^\varepsilon$, we obtain
\begin{multline}
  \label{dduhamel1}
f^\varepsilon(t)=  S_{\varepsilon^2\uptau}^\varepsilon
f^\varepsilon(t-\varepsilon^2\uptau) + \int_0^{\varepsilon^2\uptau} d\sigma\,
S_{\sigma}^\varepsilon\mathcal{N}_{t-\sigma}^\varepsilon
S_{-\sigma}^\varepsilon S_{2\varepsilon^2\uptau}^\varepsilon
f^\varepsilon(t-2\varepsilon^2\uptau) \\ +
\int_0^{\varepsilon^2\uptau} d\sigma
\int_0^{2\varepsilon^2\uptau-\sigma}  ds\,
S_{\sigma}^\varepsilon\mathcal{N}_{t-\sigma}^\varepsilon
S_{s}^\varepsilon \mathcal{N}_{t-\sigma-s}^\varepsilon
f^\varepsilon(t-\sigma-s). 
\end{multline}
Applying the operator $\mathcal{N}_t^\varepsilon$ to \eqref{dduhamel1},
and then applying successively the  average in space and the expectation
value, we obtain
\begin{multline}
  \label{dduhamel2}
  -\nabla_v\cdot \mathbb{E}\left[\fint dx\,\frac{E^\varepsilon(t) f^\varepsilon(t)}{\varepsilon} \right]=
  \fint dx\,\mathbb{E}\big[\mathcal{N}_t^\varepsilon
    S_{\varepsilon^2\uptau}^\varepsilon
    f^\varepsilon(t-\varepsilon^2\uptau)\big]\\ +
  \int_0^{\varepsilon^2\uptau} d\sigma \fint dx\, \mathbb{E}\left[
    \mathcal{N}_t^\varepsilon
    S_{\sigma}^\varepsilon\mathcal{N}_{t-\sigma}^\varepsilon
    S_{-\sigma}^\varepsilon S_{2\varepsilon^2\uptau}^\varepsilon
    f^\varepsilon(t-2\varepsilon^2\uptau) \right] +
  \mu_t^\varepsilon,
\end{multline}
with
\begin{equation}
  \label{def:eta}
\mu_t^\varepsilon= \int_0^{\varepsilon^2\uptau} d\sigma
\int_0^{2\varepsilon^2\uptau-\sigma}  ds
\fint dx\,\mathbb{E}\left[\mathcal{N}_t^\varepsilon
  S_{\sigma}^\varepsilon\mathcal{N}_{t-\sigma}^\varepsilon
  S_{s}^\varepsilon \mathcal{N}_{t-\sigma-s}^\varepsilon
  f^\varepsilon(t-\sigma-s)\right]. 
\end{equation}
Using Proposition~\ref{decoref}, we obtain that $f^\varepsilon(t)$ is
independent of  $\mathcal{N}_{s}^\varepsilon$ as soon as $s\geq t+
\varepsilon^2 \uptau$. Then, using  hypothesis $({\rm H}1)$,
we obtain 
\begin{equation}
  \label{dduhamel3}
  \mathbb{E}\big[\mathcal{N}_t^\varepsilon
    S_{\varepsilon^2\uptau}^\varepsilon
    f^\varepsilon(t-\varepsilon^2\uptau)\big] =
  \mathbb{E}\big[\mathcal{N}_t^\varepsilon \big]
  S_{\varepsilon^2\uptau}^\varepsilon
  \mathbb{E}\big[f^\varepsilon(t-\varepsilon^2\uptau)\big]=0.
\end{equation}
We note that the analysis of the term \eqref{dduhamel3} is the same as
the one done for the term \eqref{lita:eqn:T1:sto}. Therefore the first
term of the right-hand side of \eqref{dduhamel2} vanishes. Since
Proposition~\ref{decoref} implies that
$\mathcal{N}_{t}^\varepsilon$ and $\mathcal{N}_{t-\sigma}^\varepsilon$
are independent  of $f^\varepsilon(t-2\varepsilon^2\uptau)$, for
$0\leq \sigma \leq\varepsilon^2\uptau$,  we obtain from \eqref{dduhamel2},
\begin{equation}
  \label{dduhamel4}
\begin{aligned}
  -\nabla_v\cdot \mathbb{E}\left[\fint dx\,
    \frac{E^\varepsilon(t) f^\varepsilon(t)}{\varepsilon} \right]
  &= \
  \int_0^{\varepsilon^2\uptau}d\sigma \fint dx\,\mathbb{E}\left[\mathcal{N}_t^\varepsilon
    S_{\sigma}^\varepsilon\mathcal{N}_{t-\sigma}^\varepsilon
    S_{-\sigma}^\varepsilon\right]
  \mathbb{E}\left[S_{2\varepsilon^2\uptau}^\varepsilon
    f^\varepsilon(t-2\varepsilon^2\uptau) \right]  +
  \mu_t^\varepsilon \\
  &=\
  \int_0^{\varepsilon^2\uptau}d\sigma \fint dx\, \mathbb{E}\left[
    \mathcal{N}_t^\varepsilon
    S_{\sigma}^\varepsilon\mathcal{N}_{t-\sigma}^\varepsilon
    S_{-\sigma}^\varepsilon\right] \mathbb{E}\left[ f^\varepsilon(t)
    \right] 
  \\
  &+ \ \int_0^{\varepsilon^2\uptau}d\sigma \fint dx\,
  \mathbb{E}\left[ \mathcal{N}_t^\varepsilon
    S_{\sigma}^\varepsilon\mathcal{N}_{t-\sigma}^\varepsilon
    S_{-\sigma}^\varepsilon\right]
  \mathbb{E}\left[S_{2\varepsilon^2\uptau}^\varepsilon
    f^\varepsilon(t-2\varepsilon^2\uptau) -f^\varepsilon(t)\right] \\
  &+ \ \mu_t^\varepsilon.
\end{aligned}  
\end{equation}

In fact, we have to consider a weak form of \eqref{dduhamel4}, which is given by

\begin{proposition}
\label{prop:wfrhs}  
We define  the differential operator $\Theta_t^\varepsilon$ as
\begin{equation}
  \label{def:theta}
  \Theta_t^\varepsilon \varphi = \int_0^{\uptau} d\sigma\,
  (\sigma \nabla_x \cdot + \nabla_v \cdot\,)\, 
  \mathbb{E}\left[ E^\varepsilon(t-\varepsilon^2\sigma,x-\sigma v) \otimes
     E^\varepsilon(t,x) \right]\nabla_v \varphi, \quad \varphi \in \mathcal{D}(\R^+\times \R^d),
\end{equation}
and the  bilinear form $\nu_t^\varepsilon$ as
\begin{equation}
  \label{nuform}
  \nu_t^\varepsilon(\psi, \varphi)=  \int_{\R^+}dt \int_{\R^d}dv \fint dx\,\psi  \Theta_t^\varepsilon \varphi,
  \quad \forall \psi\in L^\infty(\R^+\times Q), \ \ \forall \varphi \in \mathcal{D}(\R^+\times \R^d).
\end{equation}  
Then, the weak formulation of \eqref{dduhamel4} reads: $\forall \varphi \in \mathcal{D}(\R^+\times \R^d)$,
\begin{multline}
  \int_{\R^+}dt \int_{\R^d}dv\, \nabla_v\varphi \cdot
  \mathbb{E}\left[\fint dx\, \frac{f^\varepsilon(t) E^\varepsilon(t) }{\varepsilon}\right]\\
  =\nu_t^\varepsilon\left(\mathbb{E}[ f^\varepsilon(t)], \varphi\right)+
  \nu_t^\varepsilon\left(\mathbb{E}\left[S_{2\varepsilon^2\uptau}^\varepsilon
  f^\varepsilon(t-2\varepsilon^2\uptau) -f^\varepsilon(t)\right], \varphi\right)+
  \mu_t^\varepsilon(\varphi),
\label{rhsve}
\end{multline}
where the remainder term $\mu_t^\varepsilon(\varphi)$ is given by
\begin{multline}
\label{wfmu}
\mu_t^\varepsilon(\varphi)
= -\varepsilon^4 \int_{\R^+}dt \int_{\R^d} dv \int_{0}^{\uptau} d\sigma \int_{0}^{2 \uptau-\sigma} ds \fint dx\,\\
\mathbb{E}\left[
f^\varepsilon(t-\varepsilon^2(\sigma+s))\,
\mathcal{N}_{t-\varepsilon^2(\sigma+s)}^\varepsilon
S_{-\varepsilon^2s}^\varepsilon\mathcal{N}_{t-\varepsilon^2\sigma}^\varepsilon
S_{-\varepsilon^2\sigma}^\varepsilon \mathcal{N}_t^\varepsilon
\varphi\right].
\end{multline}
\end{proposition}  

\begin{proof}
  We have to show that the weak formulation of \eqref{dduhamel4} is given
  by \eqref{rhsve}. The left hand side of \eqref{rhsve} is obtained straightforwardly
  from the left hand side of \eqref{dduhamel4}.
  The first two terms
  of the right-hand side of \eqref{rhsve} can be obtained in a similar way
  from  the first two terms of the right-hand side of  \eqref{dduhamel4} respectively.
  Indeed, we consider a non-random function $\psi=\psi(t,x,v)$, which can
  be either $\mathbb{E}[ f^\varepsilon(t)]$ or
  $\mathbb{E}\left[S_{2\varepsilon^2\uptau}^\varepsilon
    f^\varepsilon(t-2\varepsilon^2\uptau) -f^\varepsilon(t)\right]$.
  Then the quantity of interest is
  \begin{equation}
    \int_0^{\varepsilon^2\uptau}d\sigma \fint dx\, \mathbb{E}\left[
    \mathcal{N}_t^\varepsilon
    S_{\sigma}^\varepsilon\mathcal{N}_{t-\sigma}^\varepsilon
    S_{-\sigma}^\varepsilon\right] \psi.
    \label{QI1}
  \end{equation}
  Multiplying \eqref{QI1} by a test function $\varphi\in \mathcal{D}(\R^+\times \R^d)$,
  and integrating with respect to the time and velocity variables,
  we obtain, after expanding all operators,
  \begin{multline}
    \frac{1}{\varepsilon^2}  
    \int_{\R^+}dt \int_{\R^d} dv\,   \varphi(t,v)  \fint dx \int_0^{\varepsilon^2\uptau}d\sigma\, \\ 
    \nabla_v\cdot \left (\mathbb{E}\left[E^\varepsilon(t,x) \otimes 
    E^\varepsilon(t-\sigma,x-\sigma v/\varepsilon^2)\right]
    \left( \frac{\sigma}{\varepsilon^2} \nabla_x + \nabla_v 
    \right)\psi(t,x,v) \right).
    \label{QI2}
  \end{multline}
  Using integrations by parts with respect to the variables $x$ and $v$,
  and making the change of time variable $\sigma'=\sigma /\varepsilon^2$, we obtain
  from \eqref{QI2},
  \begin{multline}
      \int_{\R^+}dt \int_{\R^d} dv \fint dx\,  \psi(t,x,v)  
      \int_0^{\uptau}d\sigma
      \left(\sigma \nabla_x\cdot + \nabla_v\cdot \,
    \right)
      \left (\mathbb{E}\left[E^\varepsilon(t-\varepsilon^2\sigma,x-\sigma v) \otimes E^\varepsilon(t,x)\right]
    \nabla_v\varphi(t,v) \right)\\
    =  \int_{\R^+}dt \int_{\R^d} dv \fint dx\, 
     \psi \Theta_t^\varepsilon \varphi =\nu_t^\varepsilon(\psi,\varphi).
    \label{QI3}
  \end{multline}
  The first two terms of the right-hand side of \eqref{rhsve} follow by
  replacing respectively $\psi$ by  $\mathbb{E}[ f^\varepsilon(t)]$ and
  $\mathbb{E}\left[S_{2\varepsilon^2\uptau}^\varepsilon
  f^\varepsilon(t-2\varepsilon^2\uptau) -f^\varepsilon(t)\right]$ in \eqref{QI3}.  
  To obtain the weak formulation of the remainder term $\mu_t^\varepsilon$, we use
  the skew-adjointness of $\mathcal{N}_t^\varepsilon$ and 
  the dual of $S_t^{\varepsilon}$ given by $S_t^{\varepsilon\,\ast}=S_{-t}^\varepsilon$.
  Using  multiple velocity integrations by parts, several changes of variables in space,
  and the changes of time variables $\sigma'=\sigma/\varepsilon^2$
  and $s'=s/\varepsilon^2$, we obtain from the third term of the right-hand side of \eqref{rhsve},
  \begin{align*}
    \mu_t^\varepsilon(\varphi) &=\int_{\R^+}dt \int_{\R^d} dv\,  \varphi  \mu_t^\varepsilon\nonumber\\
    &= \int_{\R^+}dt \int_{\R^d} dv\,  \varphi
    \int_{0}^{\varepsilon^2 \uptau} d\sigma \int_{0}^{2\varepsilon^2     
      \uptau-\sigma} ds \fint dx\, \mathbb{E}\left[
      \mathcal{N}_t^\varepsilon
      S_{\sigma}^\varepsilon\mathcal{N}_{t-\sigma}^\varepsilon
      S_{s}^\varepsilon \mathcal{N}_{t-\sigma-s}^\varepsilon
      f^\varepsilon(t-\sigma-s)
      \right]\nonumber \\
    &=
    \int_{\R^+}dt 
    \int_{0}^{\varepsilon^2\uptau} d\sigma \int_{0}^{2\varepsilon^2 \uptau-\sigma} ds\,
    \mathbb{E}\left[
    \int_{\R^d} dv \fint dx\, \varphi 
      \mathcal{N}_t^\varepsilon
      S_{\sigma}^\varepsilon\mathcal{N}_{t-\sigma}^\varepsilon
      S_{s}^\varepsilon \mathcal{N}_{t-\sigma-s}^\varepsilon
      f^\varepsilon(t-\sigma-s)\right] \nonumber \\ 
    &= -\int_{0}^{\varepsilon^2 \uptau} d\sigma
    \int_{0}^{2\varepsilon^2 \uptau-\sigma} ds\,
    \mathbb{E}\left[
      \int_{\R^d} dv \fint dx\, f^\varepsilon(t-\sigma-s)
      \mathcal{N}_{t-\sigma-s}^\varepsilon
      S_{-s}^\varepsilon\mathcal{N}_{t-\sigma}^\varepsilon
      S_{-\sigma}^\varepsilon\mathcal{N}_t^\varepsilon
      \varphi\right] \nonumber \\
    &= -\varepsilon^4\int_{\R^+}dt \int_{\R^d} dv\int_{0}^{
      \uptau} d\sigma \int_{0}^{2 \uptau-\sigma} ds \fint dx\,\nonumber\\
    & \quad \quad \quad  
    \mathbb{E}\left[
      f^\varepsilon(t-\varepsilon^2(\sigma+s))
      \mathcal{N}_{t-\varepsilon^2(\sigma+s)}^\varepsilon
      S_{-\varepsilon^2s}^\varepsilon\mathcal{N}_{t-\varepsilon^2\sigma}^\varepsilon
      S_{-\varepsilon^2\sigma}^\varepsilon \mathcal{N}_t^\varepsilon
      \varphi\right],
    \label{eqn:remainder:0}
  \end{align*}
  which is \eqref{wfmu}. This ends the Proof of Proposition~\ref{prop:wfrhs}.
\end{proof}

The next lemma states the limit of the operator $\Theta_t^\varepsilon$
as $\varepsilon\rightarrow 0$.  

\begin{lemma}
  \label{lem:CVTheta}
  Under hypothesis $({\rm H}3)$,
  for all $\varphi\in \mathcal{D}(\R^+\times \R^d)$, the operator $ \Theta_t^{\varepsilon}$
 defined by \eqref{def:theta} becomes
\begin{equation}
    \label{def:Thetaeps}
    \Theta_t^{\varepsilon} \varphi =
    \nabla_v \cdot \left(
\left(\int_0^{\uptau}d\sigma\, \mathcal{R}_\uptau(t-\varepsilon^2\sigma,t,-\sigma, -\sigma v)
\right)\nabla_v \varphi\right),
\end{equation}
and  we obtain
  \begin{equation}
    \label{eqn:limitTheta}
    \Theta_t^{\varepsilon} \varphi \longrightarrow  \Theta_t^{0}
    \varphi \quad \mbox{ in } \  L^1(\R^+\times \R^d),
\end{equation}
where the operator $ \Theta_t^{0}$ is defined by
\begin{equation}
\label{def:Theta0}
\Theta_t^{0} \varphi = \nabla_v \cdot \left(
\left(\int_0^{\uptau}d\sigma\, \mathcal{R}_\uptau(t,t,-\sigma, -\sigma v)
\right)\nabla_v \varphi\right).
\end{equation}
\end{lemma}

\begin{proof}
  Using assumption $({\rm H}3)$, i.e.
  $
    \mathbb{E}[E(t,\tau,x)\otimes E(s,\sigma,y)]
    =\mathcal{R}_\uptau(t,s,\tau-\sigma,x-y),
  $
  for all $\varphi \in \mathcal{D}(\R^+\times \R^d)$,
  we obtain from
  \eqref{def:theta},
  \begin{equation}
    \label{def:ThetaEps:dual2}
    \begin{aligned}
      \Theta_t^{\varepsilon}  \varphi
      &= \int_0^{\uptau} d\sigma\,
      (\sigma \nabla_x \cdot + \nabla_v \cdot\,)\, 
      \mathbb{E}\left[ E^\varepsilon(t-\varepsilon^2\sigma,x-\sigma v) \otimes
        E^\varepsilon(t,x) \right]\nabla_v \varphi\nonumber \\
      &= \int_0^{\uptau}d\sigma\,
      \left(\sigma \nabla_x \cdot\,  + \, \nabla_v\cdot\,
      \right)\mathcal{R}_\uptau(t-\varepsilon^2\sigma,t,-\sigma,-\sigma v)
      \nabla_v \varphi \\
      &= \int_0^{\uptau}d\sigma\,
       \nabla_v\cdot\left(
       \mathcal{R}_\uptau(t-\varepsilon^2\sigma,t,-\sigma,-\sigma v)
       \nabla_v \varphi\right)\\
     &=  -\int_0^{\uptau}d\sigma\, \sigma
       (\nabla_x\cdot
       \mathcal{R}_\uptau)(t-\varepsilon^2\sigma,t,-\sigma,-\sigma v)
       \cdot \nabla_v \varphi
       +\int_0^{\uptau}d\sigma\,
       \mathcal{R}_\uptau(t-\varepsilon^2\sigma,t,-\sigma,-\sigma v)
       : \nabla_v^2 \varphi.
    \end{aligned}
  \end{equation}
  Since $\mathcal{R}_\uptau\in L^\infty(\R^+\times\R^+\times\R; W^{2,\infty}(\mathbb{T}^d))$,
  and because $\uptau$ is bounded,
  we obtain
  \begin{equation*}
    \label{def:ThetaEps:dual3}
    \int_0^{\uptau}d\sigma\,\sigma (\nabla_x\cdot\mathcal{R}_\uptau)(t-\varepsilon^2\sigma,t,-\sigma,-\sigma v)
    \cdot \nabla_v \varphi  \rightarrow  \int_0^{\uptau}d\sigma\,\sigma 
      (\nabla_x\cdot\mathcal{R}_\uptau)(t,t,-\sigma,-\sigma v) \cdot \nabla_v \varphi, 
    \end{equation*}
  for almost every $(t,v) \in \R^+\times \R^d$, as
  $\varepsilon \rightarrow 0$,
  and
  \begin{equation*}
    \label{def:ThetaEps:dual3bis}
     \int_0^{\uptau}d\sigma\,\mathcal{R}_\uptau(t-\varepsilon^2\sigma,t,-\sigma,-\sigma v) :
    \nabla_v^2 \varphi   \rightarrow  \int_0^{\uptau}d\sigma\,
    \mathcal{R}_\uptau(t,t,-\sigma,-\sigma v) :\nabla_v ^2 \varphi,
    \end{equation*}
  for almost every $(t,v) \in \R^+\times \R^d$, as
  $\varepsilon \rightarrow 0$.
  Moreover using the regularity of $\mathcal{R}_\uptau$ we obtain
  \begin{equation}
  \label{eqn:TETI}
  |  \Theta_t^{\varepsilon \, }  \varphi| \leq \| \sigma
    \nabla_x\mathcal{R}_\uptau\|_{L^\infty(\R_t^+\times\R_t^+;
      L^1([-\uptau,\uptau]_\sigma;L^\infty(\mathbb{T}^d)))} |\nabla_v
    \varphi| +
    \|\mathcal{R}_\uptau\|_{L^\infty(\R_t^+\times\R_t^+;
      L^1([-\uptau,\uptau]_\sigma;L^\infty(\mathbb{T}^d)))} 
    |\nabla_{v}^2 \varphi|,
  \end{equation}
  where the right-hand side of \eqref{eqn:TETI} defines a function in
  $L^1(\R^+\times \R^d)$ independent of $\varepsilon$.
  Therefore, using the Lebesgue dominated convergence theorem, we obtain the
  limit  \eqref{eqn:limitTheta}, where the limit point is given by
  \eqref{def:Theta0}.  This ends the proof.
\end{proof}

To deal with the second term of right-hand side of \eqref{rhsve},
we use
\begin{lemma}
  \label{lem:SecTerm}
  Under hypothesis $({\rm H}3)$,
  for all $\varphi \in \mathcal{D}(\R^+\times  \R^d)$, we obtain
  \begin{equation*}
    \lim_{\varepsilon \rightarrow 0} \,\nu_t^\varepsilon\left(
    \mathbb{E}[S_{2\varepsilon^2\uptau}^\varepsilon
      f^\varepsilon(t-2\varepsilon^2\uptau)-f^\varepsilon(t)],
    \varphi\right)= 0.
  \end{equation*}  
\end{lemma}
\begin{proof}
  Using definitions \eqref{nuform} and \eqref{def:Thetaeps} of respectively
  the bilinear form $\nu_t^\varepsilon$ and the operator $\Theta_t^\varepsilon$, 
  we obtain
  \begin{align}  
   \nu_t^\varepsilon\left( \mathbb{E}[S_{2\varepsilon^2\uptau}^\varepsilon
     f^\varepsilon(t-2\varepsilon^2\uptau)-f^\varepsilon(t)], \varphi\right) 
   &= \int_{\R^+}dt \int_{\R^d} dv  \fint dx\, 
   \mathbb{E}[S_{2\varepsilon^2\uptau}^\varepsilon
     f^\varepsilon(t-2\varepsilon^2\uptau)-f^\varepsilon(t)] \Theta_t^\varepsilon \varphi\nonumber\\
   &= \int_{\R^+}dt \int_{\R^d} dv\,\Theta_t^\varepsilon \varphi\ \mathbb{E}\left[\fint dx\,
     f^\varepsilon(t-2\varepsilon^2\uptau)- \fint dx\, f^\varepsilon(t)\right].
   \label{red2ndT}
\end{align}
On the one hand, from  Proposition~\ref{prop:WPV}, we obtain $\mathbb{E}[
  \fint dx\,f^\varepsilon]= \fint dx\,\mathbb{E}[ f^\varepsilon]
\rightharpoonup \fint dx\, f = f$ in $L^\infty(\R^+\times \R^d)$
weak--$\ast$, as $\varepsilon \rightarrow 0$. Then,
$ \mathbb{E}[\fint dx\,
 f^\varepsilon(t-2\varepsilon^2\uptau)-\fint dx\, f^\varepsilon(t)]
\rightharpoonup 0$ in $L^\infty(\R^+\times \R^d)$ weak--$\ast$ as
$\varepsilon \rightarrow 0$. On the other hand,  from
Lemma~\ref{lem:CVTheta}, we obtain $ \Theta_t^{\varepsilon}
\varphi \rightarrow \Theta_t^{0} \varphi$ in
$L^1(\R^+\times \R^d)$ strong as $\varepsilon \rightarrow 0$. Therefore,
we can pass to the limit $\varepsilon \rightarrow 0$ in
\eqref{red2ndT}, and we obtain the desired result.
\end{proof}

The asymptotic behavior of the term $\mu_t^\varepsilon$ is given by
\begin{lemma}
  \label{lem:remainder}
  Under hypothesis $({\rm H}4)$,
  for all $\varphi\in \mathcal{D}(\R^+\times \R^d)$, we obtain
  \begin{equation}
    \label{eqn:remainder}
    |\mu_t^\varepsilon(\varphi) | \leq  \varepsilon
    \uptau^4 C_0C_E\|\varphi\|_{L^\infty(\R^+;W^{3,1}(\R^d))}.
   \end{equation} 
\end{lemma} 
\begin{proof}
  Introducing the operator
  $\Gamma_{t,\sigma}^\varepsilon$, defined by
  \begin{equation}
    \label{def:Gamma}
    \Gamma_{t,\sigma}^\varepsilon =
    \mathcal{N}_{t-\varepsilon^2\sigma}^\varepsilon
    S_{-\varepsilon^2\sigma}^\varepsilon,
  \end{equation}
  we obtain from definition \eqref{wfmu} of $\mu_t^\varepsilon(\varphi)$,
  \begin{eqnarray}
   |\mu_t^\varepsilon (\varphi)| &\lesssim&  \varepsilon^4
   \uptau^2 \|f^\varepsilon(t-\varepsilon^2(\sigma+s))
   \|_{L^\infty(\R_t^+\times [0,\uptau]_\sigma \times[0,2\uptau]_s
     \times Q)} \nonumber\\ && \mathbb{E}\left[\|
     \Gamma_{t-\varepsilon^2\sigma,s}^\varepsilon
     \Gamma_{t,\sigma}^\varepsilon \mathcal{N}_t^\varepsilon
     \varphi\|_{L^\infty(\R_t^+\times [0,\uptau]_\sigma
       \times[0,2\uptau]_s; L^1(Q))} \right] \nonumber \\ &\lesssim &
   C_0 \varepsilon^4 \uptau^2 \mathbb{E}\left[ \|
     \Gamma_{t-\varepsilon^2\sigma,s}^\varepsilon
     \Gamma_{t,\sigma}^\varepsilon \mathcal{N}_t^\varepsilon
     \varphi \|_{L^\infty(\R_t^+\times [0,\uptau]_\sigma
       \times[0,2\uptau]_s; L^1(Q))} \right].
     \label{eqn:remainder:3}
  \end{eqnarray}
  Introducing the translation operator $T_{-\sigma v}$ in the
  $x$-direction, defined by
  \[
  T_{-\sigma v}\psi(t,x,v)= \psi(t,x+\sigma v, v), 
  \]
  we obtain for any
  smooth function $\psi=\psi(t,x,v)$,
  \begin{eqnarray}
    \label{eqn:gamma:1}
    \Gamma_{t,\sigma}^\varepsilon \psi &=&
    -\varepsilon^{-1}E^\varepsilon(t-\varepsilon^2 \sigma,x)\cdot
    \nabla_v(T_{-\sigma v}\psi)
    \nonumber\\ &=&-\varepsilon^{-1}E(t-\varepsilon^2
    \sigma,t/\varepsilon^2 - \sigma,x)\cdot \left( \sigma T_{-\sigma
      v} \nabla_x\psi + T_{-\sigma v} \nabla_v\psi \right).
  \end{eqnarray}
  Using \eqref{def:Gamma} and \eqref{eqn:gamma:1}, we obtain for any
  smooth function $\psi=\psi(t,x,v)$,
  \begin{eqnarray}
    \Gamma_{t-\varepsilon^2\sigma,s}^\varepsilon
    \Gamma_{t,\sigma}^\varepsilon \psi &=&
    \mathcal{N}_{t-\varepsilon^2(\sigma+s)}^\varepsilon
    S_{-\varepsilon^2 s}^\varepsilon \Gamma_{t,\sigma}^\varepsilon
    \psi \nonumber\\ &=& -\varepsilon^{-1}
    \mathcal{N}_{t-\varepsilon^2(\sigma+s)}^\varepsilon
    \big( T_{-sv}E^\varepsilon(t-\varepsilon^2 \sigma,x)\big) \cdot
    T_{-(\sigma+s)v}(\sigma \nabla_x \psi + \nabla_v \psi)
    \nonumber\\ &=&\varepsilon^{-2}\sum_{i,j}E_i^\varepsilon(t-\varepsilon^2
    (\sigma+s),x) \Big\{ s
    T_{-sv}\partial_{x_i}E_j^\varepsilon(t-\varepsilon^2 \sigma,x)
    \Big [ \sigma  T_{-(\sigma+s)v}\partial_{x_j} \psi \nonumber \\ &&
      + \  T_{-(\sigma+s)v}\partial_{v_j} \psi \Big ] \ +
    \   T_{-sv}E_j^\varepsilon(t-\varepsilon^2 \sigma,x)  \Big [
      \sigma(\sigma+s)  T_{-(\sigma+s)v} \partial_{x_ix_j}^2 \psi
      \nonumber\\ &&  +\ \sigma  T_{-(\sigma+s)v} \partial_{v_ix_j}^2
      \psi \ +\  (\sigma+s)T_{-(\sigma+s)v} \partial_{x_iv_j}^2 \psi
      \ + \ T_{-(\sigma+s)v} \partial_{v_iv_j}^2 \psi\Big]
    \Big\}. \label{eqn:Gamma2}
  \end{eqnarray}
  Using \eqref{eqn:Gamma2}, we obtain for all $\psi \in
  L^\infty(\R^+;W^{2,1}(Q))$,
  \begin{equation}
    \label{eqn:Gamma3}
    \|  \Gamma_{t-\varepsilon^2\sigma,s}^\varepsilon
    \Gamma_{t,\sigma}^\varepsilon \psi \|_{L^\infty(\R_t^+\times
      [0,\uptau]_\sigma \times[0,2\uptau]_s; L^1(Q))}  \lesssim
    \frac{\uptau^2}{\varepsilon^2} \|E \|_{L^\infty(\R^+\times
      \R^+;W^{1,\infty}(\mathbb{T}^d))}^2  \| \psi
    \|_{L^\infty(\R^+;W^{2,1}(Q))}.
  \end{equation}
  Using \eqref{eqn:Gamma3} we obtain for all $\varphi \in
  L^\infty(\R^+;W^{3,1}(\R^d))$,
  \begin{equation}
    \label{eqn:Gamma4}
    \|  \Gamma_{t-\varepsilon^2\sigma,s}^\varepsilon
    \Gamma_{t,\sigma}^\varepsilon \mathcal{N}_t^\varepsilon
    \varphi \|_{L^\infty(\R_t^+\times [0,\uptau]_\sigma
      \times[0,2\uptau]_s; L^1(Q))} \lesssim
    \frac{\uptau^2}{\varepsilon^3} \|E \|_{L^\infty(\R^+\times
      \R^+;W^{2,\infty}(\mathbb{T}^d))}^3  \| \varphi
    \|_{L^\infty(\R^+;W^{3,1}(\R^d))}.
  \end{equation}  
  Combining \eqref{eqn:remainder:3} and \eqref{eqn:Gamma4}, we obtain
  from   hypothesis $({\rm H}4)$ the final estimate \eqref{eqn:remainder},
  which ends the proof of Lemma~\ref{lem:remainder}.
\end{proof}

We are now able to conclude the proof of Theorem~\ref{thm:dloveiwtr} 
by showing that we can pass to the limit $\varepsilon \rightarrow 0$ in \eqref{eqn:EAV_2}. 
From \eqref{eqn:EAV_2} and  \eqref{rhsve}, we have for all $\varphi
\in\mathcal{D}(\R^+\times  \R^d)$,
\begin{equation}
  \label{wfEAV_2}
   \int_{\R^+}dt \int_{\R^d} dv\,  \mathbb{E}\left[\fint dx\,f^\varepsilon\right] \partial_t \varphi
 -  \int_{\R^+}dt \int_{\R^d}dv\, \nabla_v\varphi \cdot
  \mathbb{E}\left[\fint dx\, \frac{f^\varepsilon E^\varepsilon }{\varepsilon}\right]=0,
\end{equation}
with
\begin{multline}
  \int_{\R^+}dt \int_{\R^d}dv\, \nabla_v\varphi \cdot
  \mathbb{E}\left[\fint dx\, \frac{f^\varepsilon E^\varepsilon }{\varepsilon}\right]\\
  =\nu_t^\varepsilon\left(\mathbb{E}[ f^\varepsilon], \varphi\right)+
  \nu_t^\varepsilon\left(\mathbb{E}\left[S_{2\varepsilon^2\uptau}^\varepsilon
  f^\varepsilon(t-2\varepsilon^2\uptau) -f^\varepsilon(t)\right], \varphi\right)+
  \mu_t^\varepsilon(\varphi).
\label{rhsve2}
\end{multline}
From Proposition~\ref{prop:WPV}, we have $\mathbb{E}[ f^\varepsilon]
 \rightharpoonup f $ in $L^\infty(\R^+\times Q)$ weak--$\ast$, as
 $\varepsilon \rightarrow 0$. Then, we
 obtain $\mathbb{E}[ \fint dx\,f^\varepsilon]= \fint dx\,\mathbb{E}[
   f^\varepsilon] \rightharpoonup \fint dx\, f = f$ in
 $L^\infty(\R^+\times \R^d)$ weak--$\ast$, as $\varepsilon \rightarrow
 0$. Using this weak limit, we  obtain for the first term of \eqref{wfEAV_2},
\begin{equation}
  \label{FTODEI}
 \int_{\R^+}dt \int_{\R^d} dv\,  \mathbb{E}\left[\fint dx\,f^\varepsilon\right] \partial_t \varphi
  \longrightarrow
  \int_{\R^+}dt \int_{\R^d} dv\, f \partial_t \varphi, \quad \mbox{ as }
  \varepsilon \rightarrow 0.
\end{equation}
Using definitions \eqref{nuform} and \eqref{def:Thetaeps} of respectively
the bilinear form $\nu_t^\varepsilon$ and the operator $\Theta_t^\varepsilon$,
we obtain
\begin{align}
  \nu_t^\varepsilon\left(\mathbb{E}[ f^\varepsilon], \varphi\right)
  &= \int_{\R^+}dt \int_{\R^d} dv  \fint dx\,
   \mathbb{E}[f^\varepsilon] \Theta_t^\varepsilon \varphi \nonumber\\
   &= \int_{\R^+}dt \int_{\R^d} dv\,\mathbb{E}\left[\fint dx\,
     f^\varepsilon\right] \Theta_t^\varepsilon \varphi.
   \label{eqn:wfdt}
\end{align}
Moreover, from Lemma~\ref{lem:CVTheta}, we obtain $
\Theta_t^{\varepsilon} \varphi \rightarrow
\Theta_t^{0}\varphi$ in $L^1(\R^+\times \R^d)$ strong,
as $\varepsilon \rightarrow 0$. Using the weak limit
$\mathbb{E}[\fint dx\,f^\varepsilon] \rightharpoonup  f$ in $L^\infty(\R^+\times
\R^d)$ weak--$\ast$, as $\varepsilon \rightarrow 0$, we can then pass to
the limit in \eqref{eqn:wfdt} or in the first term of the right-hand side of
\eqref{rhsve2} and we obtain
\begin{equation}
  \nu_t^\varepsilon\left(\mathbb{E}[ f^\varepsilon], \varphi\right)
  \longrightarrow
  \int_{\R^+}dt \int_{\R^d}dv\, f 
  \Theta_t^{0}\varphi,
  \quad \mbox{ as }
  \varepsilon \rightarrow 0.
   \label{eqn:wfdt2}
\end{equation}
Using Lemma~\ref{lem:SecTerm}, the second term of
the right-hand side of \eqref{rhsve2} vanishes as $\varepsilon
\rightarrow 0$. Finally, for the third term of the right-hand side of
\eqref{rhsve2}, we obtain from Lemma~\ref{lem:remainder} that
$\mu_t^\varepsilon  \rightharpoonup 0$ in $\mathcal{D}'(\R^+\times
\R^d)$, as $\varepsilon \rightarrow 0$. Therefore, we obtain from
\eqref{rhsve2} and  \eqref{eqn:wfdt2},
\begin{equation}
   \label{wfdduhamel5:2}
   \int_{\R^+}dt \int_{\R^d}dv\, \nabla_v\varphi \cdot
   \mathbb{E}\left[\fint dx\, \frac{f^\varepsilon E^\varepsilon }{\varepsilon}\right]
   \longrightarrow
   \int_{\R^+}dt \int_{\R^d}dv\, f 
  \Theta_t^{0}\varphi, \quad \mbox{ as }
  \varepsilon \rightarrow 0,
\end{equation}
where
\begin{equation}
  \int_{\R^+}dt \int_{\R^d}dv\, f 
  \Theta_t^{0}\varphi =
   \int_{\R^+}dt \int_{\R^d}dv\, \varphi
    \nabla_v \cdot \left( \left(\int_0^{\uptau}d\sigma\,
   \mathcal{R}_\uptau(t,t, \sigma, \sigma v) \right) \nabla_v  f
   \right).\label{wflimit}
\end{equation}
Passing to the limit $\varepsilon \rightarrow 0$ in \eqref{wfEAV_2},
and using \eqref{FTODEI}-\eqref{wflimit},  we obtain
\[
\partial_t f - \nabla_v \cdot \left( \left(\int_0^{\uptau} d\sigma\,
\mathcal{R}_\uptau(t,t,\sigma,\sigma v)\right) \nabla_v  f
\right)=0, \quad \mbox{ in }\  \mathcal{D}'(\R^+ \times \R^d),
\]
which shows \eqref{def:diffeq} with \eqref{def:CoefDiff}.

It remains to show  time continuity of the limit point $f$.  From \eqref{eqn:EAV_2}
and the convergence result  \eqref{wfdduhamel5:2}, we have  for all
$\varphi \in\mathcal{D}(\R^d)$,
\[
\frac{d}{dt} \int_{\R^d}dv\,\varphi \mathbb{E}\left[\fint dx\, f^\varepsilon\right]
=
\int_{\R^d}dv\, \nabla_v \varphi \cdot
\mathbb{E}\left[\fint dx\, \frac{f^\varepsilon E^\varepsilon }{\varepsilon}\right]
\leq C(\varphi), \quad \forall t \in \R^+,
\]
where the constant $C(\varphi)>0$ depends on $\varphi$, but not on $t$ and
$\varepsilon$.  Therefore we obtain
\begin{multline*}
\left | \int_{\R^d} dv \,\left ( \mathbb{E}\left[\fint dx\, f^\varepsilon(t)\right] -
\mathbb{E}\left[\fint dx\, f^\varepsilon(t-2\varepsilon^2\uptau)\right]\right) \varphi
\right| \\ \leq \left | \int_{t-2\varepsilon^2\uptau}^{t} ds \int_{\R^d}
dv \fint dx\,\, \varepsilon^{-1}\mathbb{E}[E^\varepsilon(s,x)\cdot \nabla_v
  f^\varepsilon(s,x,v)] \varphi(v) \right| \\ \leq
\varepsilon^{-1}\|f_0^\varepsilon \|_{L^\infty(Q)}
\int_{t-2\varepsilon^2\uptau}^{t}ds\int_{\R^d}dv \fint dx\, |\nabla_v
\varphi(v)|\mathbb{E}[|E^\varepsilon(s,x)|] \\ \leq
2\varepsilon\uptau C_0C_E \| \varphi \|_{W^{1,1}(\R^d)} \longrightarrow 0,
\quad \mbox{ as } \varepsilon\rightarrow 0.
\end{multline*}
Therefore the set $\{\int_{\R^d} dv\, \mathbb{E}[\fint dx\, f^\varepsilon(t)]
\varphi \}_{\varepsilon >0}$ is uniformly equi-continuous in time,
and by Ascoli-Arzela theorem this set is relatively compact in
$\mathscr{C}([0,T])$.  Then, there exists  a subsequence,  still
labeled by $\varepsilon$ such that $\int_{\R^d} dv\, \mathbb{E}[\fint dx\,
  f^\varepsilon(t)] \varphi $ converges uniformly in time to  $
\int_{\R^d}  f(t) \varphi dv$, with in particular
$f_{|_{t=0}}=\fint dx\,f_0$. Finally the bound
\begin{eqnarray*}
\left\| \mathbb{E}\left[ \fint dx\,f^\varepsilon\right]  \right\|_{L^\infty(0,T; L^p(\R^d))}
&\leq& (2\pi)^{-d/p}\,
\mathbb{E}\left[  \|f^\varepsilon  \|_{L^\infty(0,T; L^p(Q))} \right]\\
&=&(2\pi)^{-d/p}\, \mathbb{E}\left[  \|f_0^\varepsilon  \|_{L^p(Q)} \right] \leq (2\pi)^{-d/p} C_0<
\infty,
\end{eqnarray*}
(obtained by using H\"older's inequality)
and the density of $\mathcal{D}(\R^d)$ in $L^{q}(\R^d)$ with $1/p+1/q=1$ and
$1<q<\infty$, imply by standard arguments that $\mathbb{E}[
  \fint dx\, f^\varepsilon(t)]  \rightharpoonup f $ in
$\mathscr{C}([0,T]; L^p(\R^d)-weak)$, for all $T>0$ and
$1<p<\infty$. This completes the proof of Theorem~\ref{thm:dloveiwtr}.

\subsection{Links with some kinetic turbulence theories of plasma physics}
\label{s:LinkKT}
In this section we relate the results of Section~\ref{s:mthm} to two
kinetic turbulence theories of plasma physics.
First, when the
autocorrelation time of particles $\uptau$ is finite, our result leads
to a diffusion matrix, which is reminiscent of the diffusion matrix of
the resonance broadening theory  \cite{DU66, Aam67, Wei69, Pri69, Dav72}, a
refinement of the quasilinear theory. This diffusion matrix falls into
the framework of  Theorem~\ref{thm:dloveiwtr}.
Second, even if Theorem~\ref{thm:dloveiwtr} and in particular hypothesis
$({\rm H}2)$ hold only for $\uptau$ finite, it is worthwhile to pass, at least
formally, to the limit $\uptau \rightarrow +\infty$.
When the autocorrelation time of particles $\uptau$ tends to infinity, we recover
formally the structure of the diffusion matrix of the quasilinear
theory \cite{DP62, VVS62} from the diffusion matrix
\eqref{def:CoefDiff}. 

\subsubsection{Resonance broadening like theory: Finite $\uptau$}
\label{ss:RBT}
Here, we show how to relate the diffusion matrix \eqref{def:CoefDiff}
to the diffusion matrix obtained in the resonance broadening theory  \cite{DU66, Aam67, Wei69, Pri69, Dav72}.
The diffusion matrix \eqref{def:CoefDiff} obtained in Theorem~\ref{thm:dloveiwtr}, can be
recast as
\begin{equation}
  \label{eqn:coefd1_1}
  \mathscr{D}_\uptau(t,v) = \int_{0}^\uptau d\sigma\,
  \mathbb{E}[  E(t,0,x)\otimes E(t,-\sigma,x-\sigma v)].
\end{equation}
Introducing the Fourier series decomposition of $E$,
\[
E(t,\tau,x) =\sum_{k\in \Z^d} e^{\,{\rm i} k\cdot
  x}\widehat{E}(t,\tau,k),
\]
we can suppose without loss of generality that
$\widehat{E}(t,\tau,k)=e^{-{\rm i}\omega(k)\tau}\widetilde{E}(t,\tau,k)$ where
the real-valued function $\Z^d \ni k\mapsto\omega(k) \in \R $ is odd,
i.e.  $\omega(-k)=-\omega(k)\ $ for all $\ k\in \Z^d$. This
transformation, which can be seen as a WKB ansatz, is just a change of unknown functions. In dimensional
variables,  $\omega(k)$ should scale as $\omega_p$, which means that
$\tau$ (respectively, $t$) represents the fast (respectively, slow) time variable.
In the framework
of the resonance broadening theory of plasma physics, which is self-consistent,
frequencies $\omega(k)$ are given by the resolution of the dispersion
relation \eqref{DEW} or by its approximation \eqref{BGR}. Since here
we work in a non-self-consistent frame, we suppose that frequencies
$\omega(k)$ are simply given by a suitable function of $k$, which is
regular and bounded with respect to $k$.  In the same spirit as
 assumption $({\rm H}3)$, we now make the following assumption:

\begin{itemize}
\item[$({\rm H}3')$:] There exist
  a non-negative real-valued function $\mathcal{E}(t,k):\R^+\times \Z^d \rightarrow \R^+$,
  with $\mathcal{E}(t,k)=\mathcal{E}(t,-k)$ and  $|k|^2|\mathcal{E}(t,k)|^{1/2}\in L^\infty(\R^+;\ell^1(\Z^d))$, and
  a bounded function 
  $A_\uptau(\tau,k):[-\uptau,\uptau]\times \Z^d\rightarrow \R^+$, even and compactly supported in $\tau$, such that
  \begin{equation*}
    \label{decorrel}
    \mathbb{E}\big[\widetilde{E}(t,\tau,k)\otimes\widetilde{E}(t,\sigma,k')\big]
    = A_\uptau(\tau-\sigma,k)\mathcal{E}(t,k)\,  \frac{k \otimes k}{|k|^2} \,\delta(k+k').
  \end{equation*}
\end{itemize}
The term $\delta(k+k')$, which provides spatial homogeneity, is reminiscent of what plasma physicists
call the random phase approximation. Indeed the random phase approximation
assumes that $\widetilde{E}(t,\tau,k) \propto \exp(\,{\rm i}\phi_k)$,
where $(\phi_k)_{k\in \Z^d}$ are independent random variables equidistributed
on $[0,2\pi]$ such that $\phi_{-k}=-\phi_{k}$. The matrix $k \otimes k/|k|^2$
means that we choose an electric field which is the gradient of an electric potential.
The function $A_\uptau$ corresponds to a
time autocorrelation function.  We suppose that the function $\tau\mapsto
A_\uptau(\tau,k)$ is bounded, even, and with compact support included in
$[-\uptau,\uptau]$, for all $k\in\Z^d$. The function $\mathcal{E}(t,k)$ corresponds to the energy of the $k$-th mode of the
electric field, whose time scale of evolution in
dimensional variables is $\tau_L$, i.e. a slow time scale in comparison to $1/\omega_p$
(remember that $1/(\omega_p\tau_L)=\varepsilon^2$).
In the self-consistent framework of resonance broadening
theory, the function $\mathcal{E}(t,k)$ are given by equation \eqref{EDOE}
in which the energy of the $k$-th mode of the electric field, $|E(t,k)|^2$, is replaced by  $\mathcal{E}(t,k)$ and the grow rate $\gamma(t,k)$
is given by a $d$-dimensional version of \eqref{DEW}
(or its approximation \eqref{QLGR}). Since here we
work in a non-self-consistent frame, we also suppose that energy amplitudes
 $\mathcal{E}(t,k)$ are simply given by a suitable bounded function, which
decreases fast enough in the variable $k$ to satisfy hypothesis  $({\rm H}4)$. Hypothesis $({\rm H}3')$,
which is a particular case of hypothesis $({\rm H}3)$,
is consistent with the spatio-temporal homogeneity property
of the turbulence.
Actually, the property $({\rm H}3')$ implies the property $({\rm H}3)$, in other words the property
$({\rm H}3')$ is less general than the property $({\rm H}3)$.
Indeed, we obtain  from $({\rm H}3')$,
\[
\mathbb{E}[E(t,\tau,x) \otimes E(t,\sigma,y)]=\sum_{k\in\Z^d}
A_\uptau(\tau-\sigma,k) \mathcal{E}(t,k)  \frac{k\otimes k}{|k|^2} e^{\,{\rm
    i} k\cdot(x-y)}e^{-{\rm i}\omega(k)(\tau-\sigma)},
\]
from which we easily observe that the spatio-temporal autocorrelation function
$\mathbb{E}[E(t,\tau,x) \otimes E(t,\sigma,y)]$ is invariant under time
and space translations. In terms of the autocorrelation matrix
$\mathcal{R}_\uptau$,  hypothesis  $({\rm H}3')$ is equivalent to
\begin{equation}
  \label{ROQL}
\mathcal{R}_\uptau(t,t,\sigma,x)= \sum_{k\in\Z^d} e^{\,-{\rm
    i}\omega(k)\sigma}A_\uptau(\sigma,k) \mathcal{E}(t,k) \frac{k\otimes k}{|k|^2} e^{\,{\rm i} k\cdot x}.
\end{equation}
Using assumption $({\rm H}3')$ in \eqref{eqn:coefd1_1}, where the
electric field $E$ is written in terms of its Fourier series
decomposition, we obtain
\begin{equation}
  \label{eqn:coefd1_2}
  \mathscr{D}_\uptau(t,v) = \sum_{k\in \Z^d} \mathcal{E}(t,k)\frac{k\otimes k}{|k|^2}\int_{0}^\uptau  d\sigma\, e^{-{\rm i} (
    \omega(k)- k\cdot v)\sigma} A_\uptau(\sigma,k).
\end{equation}
The diffusion matrix \eqref{eqn:coefd1_2} can be rewritten as
\begin{equation}
  \label{eqn:DQLR}
  \mathscr{D}_\uptau(t,v) = \sum_{k\in \Z^d}  \mathcal{E}(t,k) \frac{k\otimes k}{|k|^2} R_\uptau( \omega(k)- k\cdot v,k),
\end{equation}
where the resonance function $R_\uptau$ is given by 
\begin{eqnarray}
  \label{RF}
  R_\uptau( \omega(k)- k\cdot v,k)&=&
  \Re e\int_{0}^\uptau d\sigma\, e^{-{\rm i} ( \omega(k)-k\cdot
    v)\sigma} A_\uptau (\sigma,k) \nonumber\\ &=&
  \frac{1}{2}\int_{-\infty}^{+\infty} d\sigma\, e^{-{\rm i} ( \omega(k)- k\cdot
    v)\sigma}\,{ A_\uptau}(\sigma,k).
\end{eqnarray}
In \eqref{RF} we have used the even parity and the compact support (included in $[-\uptau, \uptau]$)
of the real function $\tau\mapsto A_\uptau(\tau,k)$ to obtain the last equality.\\

Substituting hypothesis $({\rm H3'})$ for hypothesis $({\rm H3})$ in Theorem~\ref{thm:dloveiwtr} we obtain

\begin{corollary}
\label{cor2:dloveiwtr}
Let $E$ be an integrable random vector field satisfying the assumptions
$({\rm H}1)$-$({\rm H}2)$ and $({\rm H}3')$-$({\rm H}4)$, and let $E^\varepsilon$ be given by
\eqref{def:Eepsilon}.  Let $\{f_0^\varepsilon\}_{\varepsilon>0}$ be a
sequence of independent random non-negative initial data  and $C_0$ be a
positive constant such that for a.e. $\upomega \in \Upomega$,
$\|f_0^\varepsilon \|_{L^1(Q)} +\|f_0^\varepsilon \|_{L^\infty(Q)} \leq C_0<\infty$.
Let $\mathscr{D}_\uptau$ be the matrix-valued function defined
by \eqref{eqn:DQLR}-\eqref{RF}. Let $f^\varepsilon$ be the unique weak
solution of  the Vlasov equation \eqref{eqn:V2}, with  initial data
${f^\varepsilon}_{|_{t=0}}=f_0^\varepsilon$. Then
\begin{itemize}
\item[1.] Up to extraction
of a subsequence,
 $\mathbb{E}[f_0^\varepsilon]$ converges in
$L^\infty(Q)$ weak--$\ast$ to a function $f_0\in L^1\cap L^\infty(Q)$,
$\mathbb{E}[f^\varepsilon]$  converges in $L^\infty(\R^+;L^\infty(Q))$
weak--$\ast$ to a function $f\in L^\infty(\R^+; L^1\cap
L^\infty(\R^d))$, and $\mathbb{E}[\fint dx\, f^\varepsilon]$
converges in $L^\infty(\R^+;L^\infty(\R^d))$ weak--$\ast$ to
$f$. Moreover $\mathbb{E}[\fint dx\, f^\varepsilon]$  converges in
$\mathscr{C}(0,T;L^p(\R^d)-weak)$ to $f$, for $1<p<\infty$ and for all
$T>0$. The limit point $f=f(t,v)$ is solution of the following
diffusion equation in the sense of distributions:
\begin{eqnarray*}
&&\partial_t f -\nabla_v \cdot(\mathscr{D}_\uptau\nabla_v f) =0, \quad \mbox{
    in } \  \mathcal{D}'(\R^+\times \R^d),  \\ &&
  {f}_{|_{t=0}} = \fint dx\, f_0.
\end{eqnarray*}  
\item[2.]
  The diffusion matrix $\mathscr{D}_\uptau$ is symmetric, non-negative and analytic
  in the velocity variables.
\end{itemize}  
\end{corollary}

\begin{proof}
  The proof of point 1 of Corollary~\ref{cor2:dloveiwtr} is the same
  as the proof of Theorem~\ref{cor2:dloveiwtr}. It remains to deal with the
  proof of point 2. Symmetry of the diffusion matrix $\mathscr{D}_\uptau$ is obvious,
  while reality of the function  $R_\uptau$ defined by \eqref{RF} follows from the parity (even) of the
  function  $\sigma\mapsto A_\uptau(\sigma,k)$.
  Non-negativeness of the diffusion matrix $\mathscr{D}_\uptau$ comes from the 
  non-negativeness of the function  $R_\uptau$, which results  
  from both the Bochner  theorem (see, e.g., Theorem~2~p.~346 in \cite{Yos80}) and
  the fact that the function $\sigma\mapsto A_\uptau(\sigma,k)$ is positive definite
  in the following sense:
  \begin{equation}
    \label{pdA}
  \int_\R d\tau \int_\R d \sigma A_{\uptau}(\tau-\sigma,k) \varphi(\tau) \varphi(\sigma) \geq 0,   
  \end{equation}
  for every  continuous function $\varphi$ with compact support.
  Indeed, from  assumption $({\rm H}3')$ and using  $\widetilde{E}(t,\tau,-k)=\widetilde{E}^\ast(t,\tau,k)$,
  for all vector $X\in\R^d\setminus\{0,k^\perp\}$
  and $\varphi \in \mathscr{C}_c^0(\R;\R)$ (the set of continuous and compactly supported functions
  from $\R$ to $\R$), we obtain
  \begin{multline*}
    \frac{|k\cdot X|^2}{|k|^2} \mathcal{E}(t,k)
    \int_\R d\tau \int_\R d \sigma A_{\uptau}(\tau-\sigma,k) \varphi(\tau) \varphi(\sigma)
    \\= \mathbb{E}\left[
      \left(\int_\R d\tau \widetilde{E}(t,\tau,k)\cdot X \varphi(\tau)\right)
      \left(\int_\R d\sigma \widetilde{E}(t,\sigma,k)\cdot X \varphi(\sigma)\right)^\ast
     \right]
   \geq 0,
  \end{multline*}
  which implies \eqref{pdA}. 
  Finally,  $\mathscr{D}_\uptau(t,v)$,
  defined by \eqref{eqn:DQLR}-\eqref{RF},
  is an analytic function in $v$ because the function $R_\uptau$ is the Fourier transform
  of the compactly supported function $A_\uptau$.
\end{proof}

The diffusion matrix \eqref{eqn:DQLR} is reminiscent of
the diffusion matrix of resonance broadening theory \cite{DU66,Wei69, Dav72}. 
In Remark~\ref{rk:RBT},
we present very briefly the result of the resonance broadening theory to point out
the similarities and the differences between the resonance broadening effect
of  Corollary~\ref{cor2:dloveiwtr} and that of the resonance broadening theory.
As already mentioned in Section~\ref{ss:LITA}, from a mathematical point of view,
the consequence of the finite-$\uptau$ or resonance broadening effect of  Corollary~\ref{cor2:dloveiwtr}
is a regularity improvement of the diffusion matrix in the velocity variables
(by comparison with the diffusion matrix of the quasilinear theory, see Section~\ref{ss:QLT}). 
From a physical point of view, our broadening resonance effect corresponds to a broadening in time frequency of a width
$\Delta \omega =1/\uptau$ for the resonance $\omega(k) - k\cdot v =0$
($\Delta \omega=0$ for the quasilinear theory, see
Section~\ref{ss:QLT}).
Here, up to hypothesis $({\rm H}3')$, we can choose freely the time autocorrelation function 
$A_\uptau$. From \eqref{RF}, we observe that the resonance function $R_\uptau$ is
obtained as the Fourier transform of $A_\uptau/2$. In particular, we can recover
diffusion matrices \eqref{DifMGen} and \eqref{DifMGen_2} obtained
respectively in Section~\ref{ss:LITA} and  Corollary~\ref{cor:dloveiwtr}.

\begin{remark}(resonance broadening theory)
  \label{rk:RBT}
  The finite-$\uptau$ or resonance broadening effect obtained in  Corollary~\ref{cor2:dloveiwtr} is
  reminiscent of the resonance broadening theory \cite{DU66, Aam67,Wei69, Pri69, Dav72},
  but it is not exactly the same since the derivation of resonance broadening theory is
  quite different and more involved. The resonance broadening theory is
  a correction to the quasilinear theory, which takes
  into account the  particle diffusion coming from $\mathscr{D}_\uptau$, for
  calculating  $\mathscr{D}_\uptau$ itself.  Indeed, in the original
  derivation of the quasilinear theory, the diffusion matrix of the
  quasilinear is obtained by assuming that the perturbed dynamics of
  particles (called fluctuations) can be approximated by the ballistic
  motion or the free flow because fluctuations are small enough. This
  approximation is consistent with the limit $\uptau \rightarrow
  +\infty$ (or $\tau_D \rightarrow +\infty$ in dimensional variables)
  since for such limit, particles follow almost straight lines.  If the
  fluctuation amplitude is sufficiently large and/or the wave spectrum
  is sufficiently narrow in $k$-space, the diffusion of particle
  trajectories can produce an appreciable broadening of wave-particle
  resonances $ \omega(k)- k\cdot v \simeq 0$, even when
  $\mathcal{E}_{el}/\mathcal{E}_{kin}\ll 1$.  To take these effects
  into account, the resonance broadening theory aims at modifying
  appropriately the diffusion matrix from the particle dynamics in a
  self-consistent way, i.e. by incorporating the diffusion of
  particle trajectories. Using a statistical approach \cite{DU66, Wei69, Dav72}, the
  broadening resonance theory states that the autocorrelation function
  should be given by
  \begin{equation}
    \label{ARBT}
    \begin{tabular}{l}
      $A_\uptau(\sigma,k)=\exp(-(\sigma/\uptau)^3/3), \ \ \mbox{ with }
      \ \  \uptau = (k\otimes k : \mathscr{D}_{\rm rb}(\sigma,v))^{-1/3}, \ \ \mbox{ i.e. }$ \vspace{2pt} \\
      $A_\uptau(\sigma,k)={A}_{\rm rb}(\sigma,k,\mathscr{D}_{\rm rb}(\sigma,v)):=\exp \left(-\frac{1}{3} k \otimes k:
      \mathscr{D}_{\rm rb}(\sigma,v) \sigma^3\right).$
      \end{tabular}
  \end{equation}
  We refer the reader to Remark~\ref{rem:RBT} below for a rough but
  short explanation of this kind of result in one dimension (see \cite{DU66, Wei69} for an
  original and detailed derivation). We now observe
  that the autocorrelation time $\uptau$ is no more a free parameter
  but is determined by the diffusion matrix $\mathscr{D}_{\rm rb}$. Likewise, the
  time autocorrelation function ${A}_{\rm rb}$ can no more be chosen freely
  but is an explicit function of the diffusion matrix $\mathscr{D}_{\rm rb}$.
  According to the resonance broadening theory, the diffusion matrix
  $\mathscr{D}_{\rm rb}$ is now a solution to the nonlinear functional integral
  equation given by
  \begin{eqnarray}
  \mathscr{D}_{\rm rb}(t,v) &=& \sum_{k\in \Z^d} \mathcal{E}(t,k)\, \frac{k \otimes
    k}{|k|^2}\, \Re e \int_{0}^\infty  d\sigma\, e^{-{\rm i} ( \omega(k)-k\cdot
    v)\sigma} {A}_{\rm rb}(\sigma,k,\mathscr{D}_{\rm rb}(\sigma,v))  \label{DRBT}\\
  &=&\sum_{k\in \Z^d} \mathcal{E}(t,k)\, \frac{k \otimes
    k}{|k|^2} R_{\rm rb}(\omega(k)-k\cdot
    v,k,\mathscr{D}_{\rm rb}). \nonumber
  \end{eqnarray}
  The resonance function $R_{\rm rb}$ can no more be seen as the Fourier transform (problem
  of convergence for negative time)
  of the time autocorrelation function ${A}_{\rm rb}$, but as the Laplace
  transform of it. The broadening of wave-particle
  resonance is produced by the term $k \otimes k:
  \mathscr{D}_{\rm rb}(\sigma,v) \sigma^3/3$, which can be seen an approximation of  the
  ensemble-average mean square deviation from the mean of particle
  trajectories in the turbulent electric field (see
  Remark~\ref{rem:RBT} for a rough but short explanation, and
  \cite{DU66, Wei69} for a more rigorous and exhaustive
  one). Resonance broadening theory would certainly deserve 
  a rigorous mathematical treatment, but it is out of the scope of this paper.
  This remark just aims at enlightening similarities and differences with the
  resonance broadening effect obtained in  Corollary~\ref{cor2:dloveiwtr}.
\end{remark}
  
\begin{remark}
  \label{rem:RBT}
  In this remark we roughly explain where typical expression \eqref{ARBT} comes from.
  Here, we give a short and simplified derivation of \eqref{ARBT} just to give
  a flavor. A more involved and detailed derivation can be found in \cite{DU66, Wei69}.
  In addition, without loss of generality, we  consider the dimension $d=1$
  to simplify the calculation.  In \eqref{RF} or \eqref{DRBT}, the term $ \exp({{\rm
  i}k v\sigma})$, can be seen as the result of approximating
  characteristic curves by the ballistic motion (like in the
  quasilinear theory), i.e., $ X(\sigma; 0,x,v)\simeq x+v\sigma $ and
  $ V(\sigma;0,x,v)\simeq v$. Indeed, we observe that,
  \[
  \exp({{\rm i}k v\sigma)} =  \mathbb{E}[\exp({\rm i}k\Delta
    X(\sigma))],
  \]
  with $(\Delta X(\sigma), \Delta V(\sigma)) =
  (X(\sigma)-X(0),V(\sigma)-V(0))= (v\sigma,v)$.  Considering a
  higher approximation of the characteristic curves (in particular
  an approximation that takes into account the diffusion of
  particles), we should add to the free-flow approximation a
  correction $(\delta X(\sigma), \delta V(\sigma))$, such that we
  have, $ X(\sigma; 0,x,v)\simeq x+v\sigma + \delta X(\sigma) $ and
  $ V(\sigma;0,x,v)=v+\delta V(\sigma)$. The autocorrelation
  function $\sigma\mapsto A_{\rm rb}(\sigma,k)$ is then defined  by 
  \begin{equation}
    \label{eqn:AARBT}
    A_{\rm rb}(\sigma,k):=\mathbb{E}[\exp({\rm i}k\delta X(\sigma))]=
    \mathbb{E}\left[\exp\left({\rm i}k\int_0^\sigma ds\, \delta V(s) 
        \right)\right].
  \end{equation}
  We now suppose that $\delta V(\sigma)$ is a Gaussian random
  function, with a probability distribution $\mathbb{P}$ given by
  \[
  \mathbb{P}(\delta V(\sigma)) = \frac{1}{\sqrt{2\pi
      \mathscr{D}(v)\sigma}} \exp\left(- \frac{(\delta V(\sigma))^2}{{
      2\mathscr{D}(v)\sigma}}\right),
  \]
  where the standard deviation or the diffusion coefficient
  $\mathscr{D}(v)$  is constant in time.
  Therefore we have,
  \begin{equation}
    \label{eqn:moms}
    \mathbb{E}[\delta V(\sigma)]=0, \quad \mathbb{E}[(\delta
      V(\sigma))^2]=\mathscr{D}(v)\sigma, \quad \mbox{ and }\quad
      \mathbb{E}[\delta V(\sigma)\delta
        V(\sigma')]=\mathscr{D}(v)\min(\sigma,\sigma'),
    \end{equation}
    which means that particle velocity follows a diffusion process
    characterized by the diffusion coefficient $\mathscr{D}$.  Since
    $\delta V(\sigma)$ is a Gaussian random function, using
    \eqref{eqn:moms}, we obtain from \eqref{eqn:AARBT},
    \begin{eqnarray*}
      A_{\rm rb}(\sigma,k)&=&  \mathbb{E}\left[\exp\left({\rm
          i}k\int_0^\sigma\delta V(s) ds \right)\right] \\ &=&
      \exp\left(-\frac12 k^2\int_0^\sigma ds \int_0^\sigma ds'\,
      \mathbb{E}[\delta V(s)\delta V(s')] \right) \\ &=&
      \exp\left(-\frac12 k^2\int_0^\sigma ds \int_0^\sigma ds'\,
      \mathscr{D}(v)\min(s,s') \right) \\ &=& \exp\left(-\frac16k^2
      \mathscr{D}(v)\sigma^3\right),  
    \end{eqnarray*}
   which gives, up to a mutiplicative constant in the exponential, the same result as
   \eqref{ARBT}. 
  \end{remark}

\subsubsection{Quasilinear theory: Infinite $\uptau$ }
\label{ss:QLT}
Here, we show that we can retreive formally the diffusion matrix
of the quasilinear theory  \cite{DP62,
  VVS62, Dav72} from the diffusion matrix \eqref{def:CoefDiff} or
\eqref{eqn:DQLR} by taking the formal limit $\uptau \rightarrow +\infty$.
For this, we choose an autocorrelation function $A_\uptau$, such that $A_\uptau
\rightarrow 1$ a.e. as $\uptau \rightarrow +\infty$
(e.g. $A_\uptau(\sigma,k)=\mathbbm{1}_{[-\uptau,\uptau]}(\sigma)$).
Then, letting $\uptau$ go to infinity we obtain in the sense of distributions,
\[
\lim_{\uptau \rightarrow +\infty}\int_{0}^\uptau d\sigma\, e^{-{\rm i} ( \omega(k)- k\cdot
  v)\sigma} A_\uptau (\sigma,k)  = \pi \delta( \omega(k)-
 k\cdot v) \, - \, {\rm i\,}{\rm
  p.v.}\left(\frac{1}{ \omega(k)- k\cdot v}\right),
\]
and 
\begin{equation}
  \label{slrf}
\lim_{\uptau \rightarrow +\infty}
R_\uptau( \omega(k)- k\cdot v,k) =\pi\delta (\omega(k)- k\cdot v).  
\end{equation}
Therefore, the diffusion matrix \eqref{eqn:DQLR} becomes the
quasilinear diffusion matrix of plasma physics literature (see, e.g.,
\cite{KT73, Dav72}),
\begin{equation}
  \label{eqn:DQL}
  \mathscr{D}_\infty(t,v) = \pi\sum_{k\in \Z^d} \mathcal{E}(t,k)\ \frac{k\otimes k}{|k|^2}
  \ \delta (\omega(k) - k\cdot v).
\end{equation}
The diffusion matrix \eqref{eqn:DQL} seems not very regular, since it
involves a sum of Dirac masses, namely resonances $\delta (
\omega(k)-k\cdot v)$.  Then, well-posedness of the diffusion equation
\eqref{def:diffeq} with such a diffusion matrix remains an open
issue at least for non-smooth distribution functions.
We note that we can also recover the quasilinear diffusion matrix  \eqref{eqn:DQL}
from the diffusion matrix $ \mathscr{D}_{\rm rb}$ of the resonance broadening theory (see Remark~\ref{rk:RBT})
by observing that $R_{\rm rb}(\omega(k) - k\cdot v,k,0)=\delta (\omega(k) - k\cdot v)$.
From \eqref{ARBT}, this still corresponds to an infinite autocorrelation time $\uptau$.
The limit $\uptau \rightarrow +\infty$ is a singular limit
from different points of view:
\begin{itemize}
\item[1.] When $\uptau \rightarrow \infty$, the autocorrelation
  matrix \eqref{ROQL} is no nore integrable with respect to correlation
  time $\sigma$,  but only locally integrable. This is the same for
  the  autocorrelation matrix $\mathcal{R}_\uptau$ constructed 
  in the Appendix~\ref{ACRE}. This loss of integrability entails
  a loss of regularity in the velocity variables for the diffusion matrix. This loss
  of regularity in velocity variables is even more striking
  when we observe the singular limit \eqref{slrf} for a smooth
  resonance function $R_\uptau$.
  \item[2.]  When $\uptau \rightarrow \infty$, hypothesis $({\rm H}2)$
  does not hold anymore. Indeed the stochastic electric field
  defined in Section~\ref{ss:tef} no longer satisfies a time decorrelation
  property since its decorrelation time tends to infinity. It is like
  falling back to the deterministic case.
\end{itemize}

 When $\uptau \rightarrow \infty$, the autocorrelation time of particles
 tends to infinity and the time decorrelation of the
 electric field defined in Section~\ref{ss:tef} occurs at infinite
 time. This can be interpreted as follows. The electric field becomes
 deterministic and particles trajectories are almost straight lines.
 This seems consistent with
 the fact that the original derivation of the QL theory performed
 by physicists  \cite{DP62, VVS62} is deterministic. Indeed this deterministic
 derivation is based on two main ingredients.  First the
 wave-particle interaction is assumed perturbative, and the
 perturbed dynamics of particles is approximated by the free flow
 or the ballistic motion. Second, all nonlinear wave-wave
 interactions, except for their effect on the space-averaged
 distribution function, are neglected. After the original 1962
 derivation which is deterministic, other derivations of the QL theory  (see, e.g.,
 \cite{Dav72, AAPSS75, Bal05}) appeal to some statistical arguments
 and decorrelation hypotheses, to establish the QL
 diffusion. Therefore, considering quasilinear theory as a
 probabilistic or deterministic theory remains an open question.
 
 Finally, we note that the resonance broadening theory \cite{DU66,
   Wei69} is actually a statistical (probabilistic) theory of the
 Vlasov equation and does not have a deterministic counterpart
 in the plasma physics literature. Nevertheless, in Section~\ref{ss:LITA}
 for the deterministic case, we have been able to introduce a finite
 autocorrelation time of particles $\uptau$, and to derive formally a diffusion
 matrix that is consistent with the quasilinear one in the limit
 $\uptau\rightarrow \infty$.

\section*{Acknowledgements}
The first author would like to acknowledge the Observatoire de la C\^ote d'Azur
and the Laboratoire  J.-L. Lagrange for their hospitality and
financial support. The authors would like to thank the anonymous reviewers
whose comments and suggestions helped to improve the final version of this article.
This work has been carried out within the framework of the EUROfusion Consortium and has received
funding from the Euratom research and training program 2014-2018 and 2019-2020 under grant agreement
No 633053. The views and opinions expressed herein do not necessarily reflect those of the European Commission.

\section*{Data Availability Statement}
Data sharing is not applicable to this article as no new data were created or analyzed in this study.

\appendix 
\section{An example of a random field $E$}
\label{ACRE}
In this appendix we construct an example of random electric field
$E=E(t,\tau,x)$ that satisfies hypotheses $({\rm H}1)$-$({\rm H}4)$
of Section~\ref{ss:tef}. For this, we are inspired by  
the Example~2 in \cite{PV03}. Let $r$ and $\varrho$ be positive real numbers such
that $\varrho \geq 1/r$ and $r\geq 1$. We set $\uptau=1/r +\varrho$. With this
decomposition of the autocorrelation time $\uptau$, we can choose
$\uptau$ as small as we want by taking finite but large $r$, and choose
also $\uptau$ as large as we want by taking any fixed $r$ and $\varrho$
large enough. Let $(T_k^n,X_k^n)\in \R^{1+d}$ with
$(n,k)\in \Z\times \Z^d$ be independent random variables equidistributed in
\begin{equation*}
\left\{  \frac{n}{r} + \left[-\frac{1}{2r}, \frac{1}{2r}\right] \right\} \times
\left\{ k + [-1/2,1/2]^d \right\}.
\end{equation*}
We consider also other independant random variables $\alpha_k^n$, with $(n,k)\in \Z\times \Z^d$,
such that $\mathbb{E}[\alpha_k^n]=0$ and  $\mathbb{E}[(\alpha_k^n)^2]=1$.  
Let $\eta\in \mathscr{C}_c^\infty(\R^+\times\R\times \R^d; \R)$ be a real scalar function
that is compactly supported in $\R^+\times[-\varrho/2,\varrho/2]\times \R^d$.
We define the random function $\mathcal{E}=\mathcal{E}(t,\tau,x)$ by
\begin{equation*}  
\mathcal{E}(t,\tau,x)= \sum_{(n,k)\in \Z^{1+d}} \alpha_{k}^n\eta(t,\tau-T_k^n,x-X_k^n).
\end{equation*}
Since $\eta$ has a compact support, the sum defining $\mathcal{E}$ in this equation is finite.
Obviously,  $\mathcal{E}\in \mathscr{C}_c^\infty(\R^+\times\R\times \R^d; \R)$
and  $\mathbb{E}[\mathcal{E}]=0$, which means that hypotheses $({\rm H}1)$ and $({\rm H}4)$
are satisfied. For a fixed $\tau\in \R$, the function
$\mathcal{E}(t,\tau,x)$ depends only on $(T_l^n,X_l^n,\alpha_l^n)$ with
\begin{equation*}
n\in \left[\tau-\frac{1}{2r}-\frac{\varrho}{2},\,\tau+\frac{1}{2r}+\frac{\varrho}{2}\right].
\end{equation*}
Similarly, we define
\begin{equation*}  
\mathcal{E}'(s,\sigma,y)= \sum_{(n,k)\in \Z^{1+d}} \alpha_{k}^n\eta'(s,\sigma-T_k^n,y-X_k^n),
\end{equation*}
with $\eta'\in \mathscr{C}_c^\infty(\R^+\times\R\times \R^d; \R)$  being a real scalar function
which is also compactly supported in $\R^+\times[-\varrho/2,\varrho/2]\times \R^d$.
For a fixed $\sigma\in \R$, the function
$\mathcal{E}'(s,\sigma,x)$ depends only on $(T_k^m,X_k^m,\alpha_k^m)$ with
\begin{equation*}
m\in \left[\sigma-\frac{1}{2r}-\frac{\varrho}{2},\,\sigma+\frac{1}{2r}+\frac{\varrho}{2}\right].
\end{equation*}
Then, as soon as $|\tau-\sigma|\geq 1/r +\varrho=:\uptau$, the random functions $\mathcal{E}(t,\tau,x)$ and
$\mathcal{E}'(s,\sigma,y)$ are independent, which means that  hypothesis $({\rm H}2)$
is satisfied. It remains to show  hypothesis $({\rm H}3)$. For this, we estimate
$\mathbb{E}[\mathcal{E}(t,\tau,x)\mathcal{E}'(s,\sigma,y)]$ as follows.
Using the independence of random variables $(T_k^n,X_k^n)$ and $\alpha_k^n$, and using
$\mathbb{E}[\alpha_k^n\alpha_l^m]=\delta_{mn}\delta_{kl}$,
 we obtain
\begin{eqnarray*}
  \mathbb{E}[\mathcal{E}(t,\tau,x)\mathcal{E}'(s,\sigma,y)]&=&
  \sum_{(n,k)\in \Z^{1+d}}\sum_{(m,l)\in \Z^{1+d}} \mathbb{E}[\alpha_k^n\alpha_l^m]\\
  &&\mathbb{E}[\eta(t,\tau-T_k^n,x-X_k^n)\eta'(s,\sigma-T_l^m,y-X_l^m)]\\
  &=& \sum_{(n,k)\in \Z^{1+d}} \mathbb{E}[(\alpha_k^n)^2]
  \mathbb{E}[\eta(t,\tau-T_k^n,x-X_k^n)\eta'(s,\sigma-T_k^n,y-X_k^n)]\\
  &=& \sum_{(n,k)\in \Z^{1+d}} \int_{[-1/2,1/2]^d}dz \int_{-\frac{1}{2r}}^{\frac{1}{2r}}d\theta\\
  &&\eta(t,\tau-n/r-\theta,x-k-z)\eta'(s,\sigma-n/r-\theta,y-k-z)\\
  &=&\sum_{(n,k)\in \Z^{1+d}} \int_{k+[-1/2,1/2]^d}dz \int_{(n-1/2)/r}^{(n+1/2)/r}d\theta\\
  &&\eta(t,\tau-\theta,x-z)\eta'(s,\sigma-\theta,y-z)\\
  &=& \int_{\R^d}dz \int_{\R}d\theta\, \eta(t,\tau-\theta,x-z)\eta'(s,\sigma-\theta,y-z)\\
  &=& \int_{\R^d}dz \int_{\R}d\theta\, \eta(t,\tau-\sigma+\theta,x-y+z)\eta'(s,\theta,z)\\
  &=:& \widetilde{\mathcal{R}}_{\eta,\eta'}(t,s,\tau-\sigma,x-y).
\end{eqnarray*}  
We then obtain $\widetilde{\mathcal{R}}_{\eta,\eta'}\in \mathscr{C}_c^\infty(\R^+\times\R^+\times\R\times \R^d; \R)$.
To construct the autocorrelation matrix $(\mathcal{R}_{\uptau\, ij})_{(i,j)}$, we first choose
real scalar functions  $\eta_i\in \mathscr{C}_c^\infty(\R^+\times\R\times \R^d; \R)$, for $i\in\{1,\ldots,d \}$, 
which are compactly supported in $\R^+\times[-\varrho/2,\varrho/2]\times \R^d$. Second, we define
the i-th component of the random electric field $E$ by
\begin{equation*}
  E_i(t,\tau,x)=\sum_{(n,k)\in \Z^{1+d}} \alpha_{k}^n\eta_i(t,\tau-T_k^n,x-X_k^n).
\end{equation*}
Finally, the autocorrelation matrix $\mathcal{R}_\uptau$ is defined by
\begin{equation*}
  \mathcal{R}_{\uptau\, ij}(t,s,\tau-\sigma,x-y):=\widetilde{\mathcal{R}}_{\eta_i,\eta_j}(t,s,\tau-\sigma,x-y).
\end{equation*}  
Another way to obtain the autocorrelation matrix $\mathcal{R}_\uptau$ is to take the
gradient of $\widetilde{\mathcal{R}}_{\eta,\eta'}(t,s,\tau-\sigma,x-y)$ with respect to the variables $x$ and $y$,
i.e. 
\begin{eqnarray*}
  \mathcal{R}_{\uptau\, ij}(t,s,\tau-\sigma,x-y)&:=&
  \partial_{x_i}\partial_{y_j}\widetilde{\mathcal{R}}_{\eta,\eta'}(t,s,\tau-\sigma,x-y)
  =\widetilde{\mathcal{R}}_{\partial_{x_i}\eta,\partial_{y_j}\eta'}(t,s,\tau-\sigma,x-y)\\
  &=&  \mathbb{E}[\partial_{x_i}\mathcal{E}(t,\tau,x)\partial_{y_j}\mathcal{E}'(s,\sigma,y)].
\end{eqnarray*}

\end{document}